\documentclass[reqno]{amsart}

\usepackage{amsmath,amssymb,amsthm,amsfonts}
\usepackage{mathrsfs}
\usepackage{bbm}
\usepackage{hyperref}

\usepackage{color}

\newtheorem{lemma}{Lemma}[section]
\newtheorem{theorem}{Theorem}[section]

\newtheorem{proposition}{Proposition}[section]
\newtheorem{remark}{Remark}[section]
\newtheorem{corollary}{Corollary}[section]

\numberwithin{equation}{section}

\newcommand{\dis}{\displaystyle}

\newcommand{\R}{\mathbb{R}}

\renewcommand{\S}{\mathbb{S}}

\newcommand{\FG}{\mathbf{G}}

\newcommand{\FM}{\mathbf{M}}
\newcommand{\FP}{\mathbf{P}}

\newcommand{\CA}{\mathcal{A}}

\newcommand{\CE}{\mathcal{E}}

\newcommand{\CH}{\mathcal{H}}
\newcommand{\CI}{\mathcal{I}}
\newcommand{\CJ}{\mathcal{J}}
\newcommand{\CK}{\mathcal{K}}

\newcommand{\al}{\alpha}
\newcommand{\be}{\beta}
\newcommand{\ga}{\gamma}

\newcommand{\la}{\lambda}
\newcommand{\de}{\delta}
\newcommand{\si}{\sigma}
\newcommand{\pa}{\partial}
\newcommand{\ka}{\kappa}
\newcommand{\eps}{\epsilon}
\newcommand{\rh}{\rho}
\newcommand{\ta}{\theta}

\newcommand{\Ta}{\Theta}
\newcommand{\De}{\Delta}

\newcommand{\eqdef}{\overset{\mbox{\tiny{def}}}{=}}

\begin{document}
\title[Rarefaction wave of Vlasov-Poisson-Boltzmann system]{Global stability of the rarefaction wave of the Vlasov-Poisson-Boltzmann system}

\author[R.-J. Duan]{Renjun Duan}
\address[RJD]{Department of Mathematics, The Chinese University of Hong Kong,
Shatin, Hong Kong, P.R.~China}
\email{rjduan@math.cuhk.edu.hk}

\author[S.-Q. Liu]{Shuangqian Liu}
\address[SQL]{Department of Mathematics, Jinan University, Guangzhou 510632, P.R.~China\\
and Department of Mathematics, The Chinese University of Hong Kong,
Shatin, Hong Kong, P.R.~China}
\email{tsqliu@jnu.edu.cn}

\begin{abstract}
This paper is devoted to the study of the nonlinear stability of the rarefaction waves of the Vlasov-Poisson-Boltzmann system with slab symmetry in the case where the electron background density satisfies an analogue of the Boltzmann relation. We allows that the electric potential may take distinct constant states at both far-fields. The rarefaction wave whose strength is not necessarily small is constructed through the quasineutral Euler equations coming from the zero-order fluid dynamic approximation of the kinetic system. We prove that the local Maxwellian with macroscopic quantities determined by the quasineutral rarefaction wave is time-asymptotically stable under small perturbations for the corresponding Cauchy problem on the Vlasov-Poisson-Boltzmann system. The main analytical tool is the combination of techniques we developed in \cite{DL} for the viscous compressible fluid with the self-consistent electric field and the reciprocal energy method based on the macro-micro decomposition of the Boltzmann equation around a local Maxwellian. Both the time decay property of the rarefaction waves and the structure of the Poisson equation play a key role in the analysis.
\end{abstract}
\maketitle

\thispagestyle{empty}
\setcounter{tocdepth}{1}
\tableofcontents

\section{Introduction}

\subsection{Problem}
There recently has been some progress on the nonlinear stability of three basic wave patterns for the Boltzmann equation with slab symmetry for the shock, rarefaction wave and contact discontinuity, respectively, cf.~\cite{Yu2,LYYZ,HXY}, for instance.  However, for the Boltzmann equation with forces (cf.~\cite{DD,Guo3,Mis,YZ1}), to the best of our knowledge, there are few results on the same issue. In this paper, we will study the time-asymptotic stability of the rarefaction wave for the Boltzmann equation with a self-consistent potential force on the line.

In the absence of the magnetic field, the dynamics of ions in a collisional  plasma with slab symmetry can be described by
 the following one-species Vlasov-Poisson-Boltzmann (VPB for short in the sequel) system (cf.~\cite[Chapter 6.6]{KT}):
\begin{eqnarray}\label{VPB}
\left\{\begin{array}{rll}
\pa_t F +\xi_1\pa_x F-\pa_x\phi\pa_{\xi_1}F =&Q(F,F),\\[2mm]
-\pa_x^2\phi=&\rho-\rho_e(\phi),\ \
\rho=\dis{\int_{\R^3}F}\,d\xi.
\end{array}
\right.
\end{eqnarray}
Here $F=F(t,x,\xi)\geq 0$
stands for the  density distribution function of the only ions particles which have position
$x\in\R$ and velocity
$\xi=(\xi_1,\xi_2,\xi_3)\in\R^3$ at time $t\geq 0$. The slab symmetry with respect to the first coordinate in the spatial domain $\R^3$ has been assumed.
The self-consistent electric potential $\phi=\phi(t, x)$ is induced by the total charges through the Poisson
equation. $Q(F,F)$ is the binary Boltzmann collision term whose explicit formula will be given later on, and collisions by ions with other particles are ignored.
The system (\ref{VPB}) is supplemented with initial data
\begin{equation}\label{BE.Idata1}
F(0,x,\xi)=F_0(x,\xi),
\end{equation}
and with boundary data at far fields
\begin{equation}
\label{VPB.b1}
\lim\limits_{x\rightarrow\pm\infty}F_0(x,\xi)=\FM_{[\rho_\pm,u_\pm,\theta_\pm]}(\xi),\quad u_\pm=[u_{1\pm}, 0, 0],
\end{equation}
and
\begin{equation}\label{con.phi}
\lim\limits_{x\rightarrow\pm\infty}\phi(t,x)=\phi_\pm,
\end{equation}
satisfying the quasineutral assumption
\begin{equation}
\label{VPB.b2}
\rho_\pm=\rho_e(\phi_\pm).
\end{equation}
Here $[\rho_\pm,u_\pm,\theta_\pm]$ and $\phi_\pm$ are assumed to be constant states, and  $\FM_{[\rho_\pm,u_\pm,\theta_\pm]}(\xi)$ are global Maxwellians defined in terms of \eqref{local-Maxwellian}.  Note that $\phi_\pm$ and $\rho_\pm$ can be distinct.

The density $\rho_e=\rho_e(\phi)$ of electrons in \eqref{VPB} depends only on the potential in the sense of  an analogue of the so-called Boltzmann relation, cf.~\cite{Ch}. Specifically, through the paper we suppose that

\begin{description}
  \item[$(\CA)$] $\rho_e(\phi): (\phi_m,\phi_M)\to (\rho_m,\rho_M)$ is a positive smooth function with 
 \begin{equation*}
\rho_m=\inf\limits_{\phi_m<\phi<\phi_M} \rho_e(\phi),\quad \rho_M=\sup\limits_{\phi_m<\phi<\phi_M} \rho_e(\phi),
\end{equation*}
satisfying the following three assumptions:

\medskip
$(\mathcal {A}_1)$ $\rho_e(0)=1$ with $0\in (\phi_m,\phi_M)$;

$(\mathcal {A}_2)$ $\rho_e(\phi)>0$, $\rho'_e(\phi)>0$ for each $\phi\in (\phi_m,\phi_M)$; 

$(\mathcal {A}_3)$ $\rho_e(\phi)\rho''_e(\phi)\leq [\rho'_e(\phi)]^2$ for each $\phi\in (\phi_m,\phi_M)$.

\end{description}
%
%
%
%

\medskip


\noindent Since the electric potential in \eqref{VPB} can be up to an arbitrary constant, the assumption $(\mathcal {A}_1)$ just means that  the electron density has been normalized to be unit when the potential is zero. The other two assumptions $(\CA_2)$ and $(\CA_3)$ assert that the pressure $P^\phi(\rho)$ generated by the potential force under the quasineutral assumption $\rho=\rho_e(\phi)$ is a positive, increasing and convex function of $\rho\in (\rho_m,\rho_M)$; it is to be further clarified later on, see \eqref{pr.pp}.
A typical example satisfying $(\mathcal {A})$ takes the form of
\begin{equation}
\label{D-dene}
\rho_e(\phi)=\left[1+\frac{\ga_e-1}{\ga_e} \frac{\phi}{A_e}\right]^{\frac{1}{\ga_e-1}}, \quad\phi_{m}=-\frac{\ga_e}{\ga_e-1}A_e,\quad \phi_M=+\infty,
\end{equation}
with $\ga_e\geq 1$ and $A_e>0$ being constants.
Note that $\rho_e(\phi)\to e^{\frac{\phi}{A_e}}$ and $\phi_m\to -\infty$
as $\ga_e\to1^+$, which corresponds to the classical Boltzmann relation. Formally, \eqref{D-dene} can be deduced from the momentum equation of the isentropic Euler-Poisson system for the fluid of electrons with the adiabatic exponent $\ga_e$ under the zero-limit of electron mass, namely,
$
\pa_x\left(A_e\rho_e^{\ga_e}\right)=\rho_e\pa_x\phi.
$

The Boltzmann collision operator $Q(\cdot, \cdot)$ in \eqref{VPB} is assumed to be for the hard sphere model (cf.~\cite{CIP,CC}), taking the following non-symmetric form
\begin{equation}
\label{def.bo}
\begin{split}
Q(H_1,H_2)=&\int_{\R^3\times \S_+^2}|(\xi-\xi_\ast)\cdot\omega|
\left[H_1(\xi_\ast')H_2(\xi')-H_1(\xi_\ast)H_2(\xi)\right]\,d\xi_\ast d\omega \\
=&Q_{\textrm{gain}}(H_1,H_2)-Q_{\textrm{loss}}(H_1,H_2),
\end{split}
\end{equation}
where $\S_+^2=\{\omega\in\S^2: (\xi-\xi_\ast)\cdot\omega\geq0\}$,
and $(\xi,\xi_{\ast})$ and $(\xi',\xi_{\ast}')$ denote velocities of two particles before and after an elastic collision, respectively, satisfying
\begin{equation}\label{re.ppc}
\xi'=\xi-[(\xi-\xi_\ast)\cdot\omega]\,\omega,\quad
\xi_\ast'=\xi_\ast+[(\xi-\xi_\ast)\cdot\omega]\,\omega,
\end{equation}
in terms of the conservations of momentum and  energy
$$
\xi+\xi_{\ast}=\xi'+\xi'_{\ast},\quad
|\xi|^2+|\xi_{\ast}|^2=|\xi'|^2+|\xi'_{\ast}|^2.
$$
Note that $ |\xi-\xi_{\ast}| =|\xi'-\xi'_{\ast}|$ holds true.

In the paper, we are interested in the large time asymptotics toward the rarefaction wave of solutions to the Cauchy problem on the VPB system \eqref{VPB}, \eqref{BE.Idata1}, \eqref{VPB.b1}, \eqref{con.phi}, \eqref{VPB.b2}. Precisely, we will show that the local Maxwellian $\FM_{[\rho^r(t,x),u^r(t,x),\theta^r(t,x)]}(\xi)$, where $\left[\rho^r(t,x),u^r(t,x),\theta^r(t,x)\right]$ is a smooth rarefaction wave of the macroscopic quasineutral compressible Euler equations with the same far-field data $[\rho_\pm,u_\pm,\theta_\pm]$ as given in \eqref{VPB.b1}, is stable globally in time in a suitable Sobolev space under small perturbations, and further show that the solution $F(t,x,\xi)$ to the Cauchy problem converges in large time in a velocity-exponential weighted  $L^\infty_xL^2_\xi$ norm toward the local Maxwellian $\FM_{\left[\rho^R(x/t),u^R(x/t),\theta^R(x/t)\right]}(\xi)$ with $\left[\rho^R(x/t),u^R(x/t),\theta^R(x/t)\right]$ being the centred rarefaction wave of the corresponding Riemann problem and the electric potential $\phi(t,x)$ converges in large time in $L^\infty_x$ norm toward $\rho_e^{-1}(\rho^R(x/t))$.

\subsection{Macro-micro decomposition around local Maxwellians}
As in \cite{LYY}, letting $F(t,x,\xi)$ be the solution to the VPB system \eqref{VPB}, one can decompose it into the summation of the macroscopic (or fluid) part represented
by the local Maxwellian ${\bf M}={\bf M}(t,x,\xi)={\bf M}_{[\rho(t,x), u(t,x), \theta(t,x)]}(\xi)$, and the microscopic (or kinetic) part denoted by
${\bf G}={\bf G}(t,x,\xi)$ as
\begin{equation}
\label{def.mmd}
F(t,x,\xi)={\bf M}(t,x,\xi)+{\bf G}(t,x,\xi).
\end{equation}
Here, ${\bf M}(t,x,\xi)$ is defined by the solution $F(t,x,\xi)$ of the VPB system \eqref{VPB} through the five fluid quantities, i.e., the mass density $\rho(t,x)$, momentum density $m(t,x)=\rho(t,x)u(t,x)$, and the energy density $\CE(t,x)+\frac{1}{2}|u(t,x)|^2$, given by
\begin{eqnarray*}
\rho(t,x)&\equiv& \int_{{\R}^3}F(t,x,\xi)\,d\xi,\nonumber\\[2mm]
\rho(t,x)u_i(t,x)&\equiv&\int_{{\R}^3}\psi_i(\xi)F(t,x,\xi)\,d\xi,\ \ i=1,2,3,\\[2mm]
\left[\rho\left(\CE(t,x)+\frac{1}{2}|u(t,x)|^2\right)\right]&\equiv&\int_{{\R}^3}\psi_4(\xi)F(t,x,\xi)\,d\xi,\nonumber
\end{eqnarray*}
in the form of
\begin{equation}\label{local-Maxwellian}
{\bf M}_{[\rho(t,x), u(t,x), \theta(t,x)]}(\xi)\equiv\frac{\rho(t,x)}{(2\pi
R\theta(t,x))^{\frac{3}{2}}}\exp\left(-\frac{|\xi-u(t,x)|^2}{2R\theta(t,x)}\right),
\end{equation}
where $\theta(t,x)$ is the temperature which is related to the
internal energy $\CE(t,x)$  by $\CE=\frac{3}{2}R\theta=\theta$ with the gas constant $R$  chosen
to be $\frac{2}{3}$ for convenience, and
$u(t,x)=[u_1(t,x),u_2(t,x),u_3(t,x)]$ is the fluid velocity in $\R^3$. Also,
$\psi_i$, $i=0,1,2,3,4,$ are the five collision invariants
$$
\psi_0=1,\ \ \psi_i=\xi_i\ (i=1,2,3),\ \ \psi_4=\frac{1}{2}|\xi|^2,
$$
satisfying
\begin{equation}
\label{rjad-clq}
\int_{{\R}^3}\psi_i Q(F, F)\,d\xi=0\ \ \text{for}\ \ i=0,1,2,3,4.
\end{equation}
For any given Maxwellian $\widehat{{\bf M}}={\bf
M}_{\left[\widehat{\rho},\widehat{u},\widehat{\theta}\right]}$, we define an inner
product in $\xi\in{\R}^3$ as
$$
\langle f,g\rangle_{\widehat{{\bf
M}}}\equiv\int_{{\R}^3}\frac{f(\xi)g(\xi)}{\widehat{{\bf M}}}d\xi,
$$
for two functions $f$ and $g$ such that the integral on the right is well defined.

Using the above inner product with respect to the Maxwellian $\widehat{{\bf M}}$, the
following five functions spanning the macroscopic subspace, are mutually orthogonal:
\begin{eqnarray*}
\chi^{\widehat{{\bf M}}}_0\left(\xi;\widehat{\rho},\widehat{u},\widehat{\theta}\right)&\equiv& \frac{1}{\sqrt{\widehat{\rho}}}\widehat{{\bf M}},\\[2mm]
\chi^{\widehat{{\bf
M}}}_{i}\left(\xi;\widehat{\rho},\widehat{u},\widehat{\theta}\right)&\equiv&\frac{\xi_i-\widehat{u}_i}
{\sqrt{{R}\widehat{\rho}\,\widehat{\theta}}}\widehat{{\bf M}},\ \ i=1,2,3,\\[2mm]
\chi^{\widehat{{\bf
M}}}_{4}\left(\xi;\widehat{\rho},\widehat{u},\widehat{\theta}\right)&\equiv&\frac{1}{\sqrt{6\widehat{\rho}}}\left(\frac{|\xi-\widehat{u}|^2}
{{R}\widehat{\theta}}-3\right)\widehat{{\bf M}},\\[2mm]
\left\langle\chi^{\widehat{{\bf M}}}_i,\chi^{\widehat{{\bf
M}}}_j\right\rangle_{{\widehat{\FM}}}&=&\delta_{ij},\ \ \text{for}\ \ i,j=0,1,2,3,4,
\end{eqnarray*}
where $\delta_{ij}$ is the Kronecker delta.  With the above orthonormal set, the
macroscopic projection ${\bf P}^{\widehat{{\bf M}}}_0$ and the microscopic
projection ${\bf P}^{\widehat{{\bf M}}}_1$ can be defined as
$$
\left\{
\begin{array}{rll}{\bf P}^{\widehat{{\bf
M}}}_0h&\equiv&\sum\limits_{j=0}^4\left\langle h,\chi^{\widehat{{\bf
M}}}_j\right\rangle_{\widehat{{\bf M}}}\chi^{\widehat{{\bf M}}}_j,\\[2mm]
{\bf P}^{\widehat{{\bf M}}}_1h&\equiv&h-{\bf P}^{\widehat{{\bf M}}}_0h.
\end{array} \right.
$$
Notice that the operators ${\bf P}^{\widehat{{\bf M}}}_0$ and ${\bf
P}^{\widehat{{\bf M}}}_1$ are orthogonal (and thus self-adjoint)
projections with respect to the inner product
$\langle\cdot,\cdot\rangle_{\widehat{\FM}}$~, i.e.
$$
{\bf P}^{\widehat{{\bf M}}}_0{\bf P}^{\widehat{{\bf M}}}_0={\bf
P}^{\widehat{{\bf M}}}_0,\ \ {\bf P}^{\widehat{{\bf M}}}_1{\bf
P}^{\widehat{{\bf M}}}_1={\bf P}^{\widehat{{\bf M}}}_1,\ \ {\bf
P}^{\widehat{{\bf M}}}_0{\bf P}^{\widehat{{\bf M}}}_1={\bf
P}^{\widehat{{\bf M}}}_1{\bf P}^{\widehat{{\bf M}}}_0=0.
$$
Moreover, it is straightforward to check that
$$
\left\langle{\bf P}^{\widehat{{\bf M}}}_0h, {\bf P}^{\widehat{\widehat{\bf M}}}_1h\right\rangle_{\widehat{{\bf M}}}=\left\langle{\bf P}^{\widehat{\widehat{\bf
M}}}_0h, {\bf P}^{\widehat{{\bf M}}}_1h\right\rangle_{\widehat{\widehat{\bf M}}}=0
$$
holds true for any two Maxwellians $\widehat{{\bf M}}$ and $\widehat{\widehat{\bf M}}$.

Using notations above, the solution $F(t,x,\xi)$ of \eqref{VPB}
satisfies
$$
{\bf P}^{{\bf M}}_0F={\bf M},\quad \ {\bf P}^{{\bf M}}_1F={\bf G}.
$$
By the macro-micro decomposition, the Boltzmann equation in \eqref{VPB}
can be rewritten as
\begin{equation}\label{mmBE}
\pa_t({\bf M}+{\bf G})+\xi_1\pa_x({\bf M}+{\bf
G})-\pa_x\phi\pa_{\xi_1}({\bf M}+{\bf
G})=L_{\bf M}\FG+Q(\FG,\FG),
\end{equation}
where
$$
L_{\bf M}\FG=Q(\FG,\FM)+Q(\FM,\FG)
$$
is the linearized Boltzmann collision operator around the local Maxwellian ${\bf M}$.

Applying ${\bf P}^{{\bf M}}_0$ and ${\bf P}^{{\bf M}}_1$ to
\eqref{mmBE}, one has
\begin{equation*}
\pa_t{\bf M}+{\bf P}^{{\bf M}}_0\left(\xi_1\pa_x{\bf
M}\right)+{\bf P}^{{\bf M}}_0\left(\xi_1\pa_x{\bf G}\right)-\pa_x\phi\pa_{\xi_1}{\bf
M}=0,
\end{equation*}
and
\begin{equation}\label{micBE}
\pa_t{\bf G}+{\bf P}^{{\bf M}}_1\left(\xi_1\pa_x{\bf
M}\right)+{\bf P}^{{\bf M}}_1\left(\xi_1\pa_x{\bf
G}\right)-\pa_x\phi\pa_{\xi_1}{\bf G}=L_{\FM}\FG+Q(\FG,\FG),
\end{equation}
respectively. Notice that \eqref{micBE} further implies
\begin{equation}
\label{micBEc}
{\bf G}=L_{\FM}^{-1}\Big({\bf P}^{{\bf M}}_1\left(\xi_1\pa_x{\bf
M}\right)\Big)+\Ta,
\end{equation}
with
\begin{equation}\label{Ta.def}
\Ta=L^{-1}_{\FM}\left[\pa_t\FG+{\bf P}^{{\bf M}}_1\left(\xi_1\pa_x{\bf
G}\right)-\pa_x\phi\pa_{\xi_1}{\bf G}\right]-L^{-1}_{\FM}[Q(\FG,\FG)].
\end{equation}

\subsection{Macroscopic balance laws}
Now, {due to \eqref{rjad-clq},} from
$$
\int_{{\R}^3}\psi_i\left(\pa_t F +\xi_1\pa_x F-\pa_x\phi\pa_{\xi_1}F\right)d\xi=0,\quad i=0,1,2,3,4,
$$
the system of macroscopic moments takes the following form
\begin{eqnarray}\label{cons.law.}
\left\{
\begin{array}{clll}
\begin{split}
&\pa_t\rho+\pa_x (\rho u_1)=0,\\[2mm]
&\pa_t (\rho u_1)+\pa_x(\rho u_1^2)+\pa_x P+\rho\pa_x\phi=-\int_{{\R}^3}\xi_1^2\pa_x\FG\,d\xi,
\\[2mm]
&\pa_t (\rho u_i)+\pa_x(\rho {u_1u_i})=-\int_{{\R}^3}\xi_i\xi_1\pa_x\FG\,d\xi,\
\ i=2,3,\\[2mm]
&\pa_t \left[\rho\left(\CE+\frac{1}{2}|u|^2\right)\right]+\pa_x \left[u_1\left(\rho\left(\CE+\frac{1}{2}|u|^2\right)+P\right)\right]+\rho u_1\pa_x \phi=-\frac{1}{2}\int_{{\R}^3}|\xi|^2\xi_1\pa_x\FG\,d\xi.
\end{split}
\end{array}
\right.
\end{eqnarray}
Furthermore, by substituting  \eqref{micBEc}, the above Euler-type system \eqref{cons.law.} together with the Poisson equation in \eqref{VPB} lead to the following fluid-type system in the  Navier-Stokes-Poisson form (cf.~\cite{DL}):
\begin{eqnarray}\label{BE-NS}
\left\{
\begin{array}{clll}
\begin{split}
&\pa_t \rho+\pa_x (\rho u_1)=0,\\[2mm]
&\pa_t u_1+u_1\pa_x u_1+\frac{\pa_x P}{\rho}+\pa_x\phi=\frac{3}{\rho}\pa_x\left(\mu(\theta)\pa_x u_1\right)
-\frac{1}{\rho}\int_{{\R}^3}\xi_1^2\pa_x\Theta\,d\xi,\\[2mm]
&\pa_t u_i+{u_1\pa_x u_i}=\frac{1}{\rho}\pa_x\left(\mu(\theta)\pa_x u_i\right)-\frac{1}{\rho}\int_{{\R}^3}\xi_1\xi_i\pa_x\Theta\,d\xi,\
\ i=2,3,\\[2mm]
&\pa_t \left(\CE+\frac{1}{2}|u|^2\right)+u_1\pa_x\left(\CE+\frac{1}{2}|u|^2\right)+\frac{\pa_x (P u_1)}{\rho}+u_1\pa_x\phi\\
&\qquad=\frac{1}{\rho}\pa_x\left(\kappa(\theta)\pa_x \theta\right)+\frac{3}{\rho}\pa_x\left(\mu(\theta)u_1\pa_x u_1\right)+
\frac{1}{\rho}\sum\limits_{i=2}^3\pa_x\left(\mu(\theta)u_i\pa_x u_i\right)
-\frac{1}{2\rho}\int_{{\R}^3}|\xi|^2\xi_1\pa_x\Theta\,d\xi,\\
&-\pa^2_x\phi=\rho-\rho_e(\phi).
\end{split}
\end{array}
\right.
\end{eqnarray}
Note that system \eqref{BE-NS} is unclosed since $\Ta$ depends on the unknown function ${\bf G}$. Here and in the sequel,
$$
P=\frac{2}{3}\rho\CE=\frac{2}{3}\rho\ta
$$
is the pressure for the monatomic gas, and the viscosity coefficient $\mu(\theta)$ and the
heat conductivity coefficient $\kappa(\theta)$, both depending only on $\theta$,  are represented by
\begin{eqnarray}\label{Bur.fun.}
\left\{
\begin{array}{rllll}
\begin{split}
\mu(\ta)&=-\frac{1}{2\theta}\int_{{\R}^3}\xi_1^2L^{-1}_{\FM_{[1,u,\ta]}}
\left(\xi_1^2\FM_{[1,u,\ta]}\right)\,d\xi\\
&=-\frac{3}{2\theta}\int_{{\R}^3}\xi_1\xi_iL^{-1}_{\FM_{[1,u,\ta]}}
\left(\xi_1\xi_i\FM_{[1,u,\ta]}\right)\,d\xi>0,\
i=2, 3,\\
\ka(\ta)&=-\frac{3}{8\ta^2}\int_{{\R}^3}|\xi-u|^2\xi_iL^{-1}_{\FM_{[1,u,\ta]}}
\left(|\xi-u|^2\xi_i\FM_{[1,u,\ta]}\right)\,d\xi>0,\
i=1, 2, 3.
\end{split}
\end{array}
\right.
\end{eqnarray}
For completeness, we will deduce the above formulas in the appendix, see also \cite{Gr,Sone}.

Recalling $\CE=\theta$, the energy equation in \eqref{BE-NS} can be reduced to
\begin{equation}\label{theta}
\begin{split}
\pa_t \theta+u_1\pa_x\theta+\frac{P\pa_x u_1}{\rho}=&\frac{1}{\rho}\pa_x\left(\kappa(\theta)\pa_x \theta\right)+\frac{3}{\rho}\mu(\theta)(\pa_x u_1)^2
\\&+
\sum\limits_{i=2}^3\frac{1}{\rho}\mu(\theta)(\pa_x u_i)^2-\frac{1}{\rho}\int_{{\R}^3}\left(\frac{|\xi|^2}{2}-u\cdot\xi\right)\xi_1\pa_x\Theta\,d\xi.
\end{split}
\end{equation}
For later use, {as in \cite{LYY},} for given $\rho$ and $\theta$, we also define a corresponding entropy quantity $S$ as
\begin{equation}\label{BE.enp.}
S\eqdef -\frac{2}{3}\ln \rho+\ln \left(\frac{4}{3}\pi \theta\right)+1,
\end{equation}
and deduce from the first equation of \eqref{BE-NS} together with \eqref{theta} that $S$ satisfies
\begin{equation*}
\begin{split}
\pa_t S+u_1\pa_x S=&\frac{1}{\rho\theta}\pa_x\left(\kappa(\theta)\pa_x \theta\right)+\frac{3}{\rho\theta}\mu(\theta)(\pa_x u_1)^2
\\&+
\sum\limits_{i=2}^3\frac{1}{\rho\theta}\mu(\theta)(\pa_x u_i)^2-\frac{1}{\rho\theta}\int_{{\R}^3}\left(\frac{|\xi|^2}{2}-u\cdot\xi\right)\xi_1\pa_x\Theta\,d\xi.
\end{split}
\end{equation*}
Notice that from \eqref{BE.enp.}, we have
$$
\theta=\frac{3}{2}k e^{S}\rho^{2/3}
$$
with the constant $k$ given by $k\eqdef \frac{1}{2\pi e}$, so that the pressure can be written as
$$
P=\frac{2}{3}\rho\theta=k e^{S}\rho^{5/3}.
$$
Therefore, whenever $P$ is regarded as a function $v=1/\rho>0$ and $S$ given by $P=k e^{S}v^{-5/3}$, not only $P$ is convex in both $v$ and $S$ but also $P$ is uniformly convex in $[v,S]$. Similarly, it is also the case for $\theta=\frac{3}{2}k e^{S}v^{-2/3}$. We also remark that $v$, $\theta$, $P$ and $S$ obey the second law of thermodynamics
$$
\theta\,dS=d\theta+P\,d v,
$$
implying that any two  thermodynamical quantities among  $v$, $\theta$, $P$ and $S$ can uniquely determine all the other ones.

\subsection{Quasineutral Euler equations and rarefaction waves}

In order to study the large time behavior of the solution $[F(t,x,\xi),\phi(t,x)]$ to the Cauchy problem \eqref{VPB}, \eqref{BE.Idata1}, \eqref{VPB.b1}, \eqref{con.phi}, \eqref{VPB.b2} on the VPB system, we expect that the density distribution function $F(t,x,\xi)$ tends time-asymptotically to the local Maxwellian
$
\FM_{[\rho^R,u^R,\theta^R](x/t)}(\xi),
$
where $\left[\rho^R,u^R,\theta^R\right](x/t)$ with $u^R(x/t)=\left[u_1^R(x/t),0,0\right]$ is defined to be the centre-rarefaction wave solution to the Riemann problem on the macroscopic quasineutral Euler system
\begin{eqnarray}\label{MEt}
\left\{
\begin{array}{clll}
\begin{split}
&\pa_t\rho+\rho\pa_x u_1+ u_1\pa_x\rho=0,\\[2mm]
&\pa_t u_1+u_1\pa_x u_1+\frac{\pa_x P}{\rho}+\pa_x\phi=0,\\[2mm]
&\pa_t \theta+u_1\pa_x\theta+\frac{P\pa_x u_1}{\rho}=0,\\[2mm]
&\rho=\rho_e(\phi),
\end{split}
\end{array}
\right.
\end{eqnarray}
with Riemann initial data given by
\begin{equation}
\label{MEtid}
[\rho, u_1, \theta](0,x)=\left[\rho^R_0,u^R_{1,0}, \theta^R_0\right](x)\eqdef \left\{\begin{array}{rll}[\rho_-,u_{1-}, \theta_-],&\ \ x<0,\\[2mm]
[\rho_+,u_{1+}, \theta_+],&\ \ x>0.
\end{array}
\right.
\end{equation}
Here we recall $P=\frac{2}{3}\rho \theta$. Due to assumptions $(\CA_1)$ and $(\CA_2)$, $\rho_e^{-1}(\cdot)$ exists and the quasineutral equation $\rho=\rho_e(\phi)$ implies $\phi=\rho_e^{-1}(\rho)$, so that the electric potential $\phi(t,x)$  to the Cauchy problem \eqref{VPB}, \eqref{BE.Idata1}, \eqref{VPB.b1}, \eqref{con.phi}, \eqref{VPB.b2} correspondingly tends time-asymptotically to
\begin{equation*}
\phi^R\left(\frac{x}{t}\right)\eqdef \rho_e^{-1} \left(\rho^R\left(\frac{x}{t}\right)\right).
\end{equation*}

The rarefaction wave $\left[\rho^R,u_1^R,\theta^R\right](x/t)$ can be constructed as follows. Recalling \eqref{BE.enp.}, system \eqref{MEt} can be rewritten in terms of $[\rho,u_1,S]$ as
\begin{eqnarray}\label{ME}
\left\{
\begin{array}{clll}
\begin{split}
&\pa_t\rho+\rho\pa_x u_1+ u_1\pa_x\rho=0,\\[2mm]
&\pa_t u_1+u_1\pa_x u_1+\frac{\pa_x P}{\rho}+\pa_x\rho^{-1}_e(\rho)=0,\\[2mm]
&\pa_t S+u_1\pa_x S=0,
\end{split}
\end{array}
\right.
\end{eqnarray}
with $P=k e^{S}\rho^{5/3}$. We define
\begin{equation*}
P^\phi(\rho)=\int^\rho \frac{\varrho}{\rho_e'(\rho_e^{-1}(\varrho))}d\varrho,
\end{equation*}
which is called  the pressure generated by the potential force such that $\pa_x P^\phi(\rho)=\rho\pa_x\phi$ under the quasineutral assumption $\rho=\rho_e(\phi)$. It is straightforward to check
\begin{equation*}
\pa_\rho P^\phi(\rho)=\frac{\rho_e(\phi)}{\rho_e'(\phi)},\quad \pa_\rho^2P^\phi(\rho)=\frac{[\rho_e'(\phi)]^2-\rho_e(\phi) \rho_e''(\phi)}{[\rho_e'(\phi)]^3},
\end{equation*}
with $\phi=\rho_e^{-1}(\rho)$ on the right. Notice that due to the assumptions $(\CA_2)$ and $(\CA_3)$, one has
\begin{equation}
\label{pr.pp}
\pa_\rho P^\phi(\rho)>0,\quad  \pa_\rho^2P^\phi(\rho)\geq 0,
\end{equation}
for each $\rho\in (\rho_m,\rho_M)$.  The quasineutral Euler system \eqref{ME}
has three characteristics
\begin{eqnarray}\label{chara.}
\left\{\begin{array}{rll}
\la_1&=&\la_1(\rho,u_1,S)\equiv u_1-\sqrt{\pa_\rho P (\rho,S) + \pa_\rho P^\phi(\rho)},\\[3mm]
\la_2&=&\la_2(\rho,u_1,S)\equiv u_1,\\[3mm]
\la_3&=&\la_3(\rho,u_1,S)\equiv u_1+\sqrt{\pa_\rho P (\rho,S) + \pa_\rho P^\phi(\rho)}.
\end{array}\right.
\end{eqnarray}
In terms of  two Riemann invariants of the third eigenvalue $\la_3(\rho,u_1,S)$, regarding the original quasineutral Euler system \eqref{MEt} of variables $[\rho,u_1,\theta]$, we define the set of right constant states  $[\rho_+,u_{1+},\theta_+]$ to which a given left constant state $[\rho_-,u_{1-},\theta_-]$ with $\rho_->0$ and $\theta_->0$ is connected  through the $3$-rarefaction wave to be
\begin{multline}\label{def.r3}
R_3(\rho_-,u_{1-}, \theta_-)\equiv\bigg\{[\rho,u_1,\theta]\in\R_+\times\R \times\R_+\ \Big|\
\frac{\rho^{2/3}}{\theta}=\frac{\rho_-^{2/3}}{\theta_-},
\\
u_1-u_{1-}=\int_{\rho_-}^\rho
\frac{\sqrt{\pa_\rho P (\varrho,S_i) + \pa_\rho P^\phi(\varrho)}}{\varrho}
d\varrho,
\ {\rho}>\rho_-,\ {u}_1>u_{1-}\bigg\}.
\end{multline}
Here and in the sequel 
$S_i\eqdef -\frac{2}{3}\ln \rho_-+\ln (\frac{4}{3}\pi\theta_-) +1$ is a constant. Noticing $P(\rho,S_i)=A_i\rho^{5/3}$ with $A_i\eqdef ke^{S_i}$, one can also write
\begin{equation*}
{\pa_\rho P (\rho,S_i) + \pa_\rho P^\phi(\rho)=\frac{5}{3}A_i\rho^{\frac{2}{3}} +\rho\left(\frac{d}{d\rho} (\rho_e^{-1})\right)(\rho)}.
\end{equation*}
Throughout the paper, without loss of generality, we consider only the $3$-rarefaction wave, and the case for the $1$-rarefaction wave can be treated in a similar way. Now, letting $[\rho_+,u_{1+},\theta_+]\in R_3(\rho_-,u_{1-}, \theta_-)$,
the Riemann problem \eqref{MEt}, \eqref{MEtid} admits a self-similar solution, the $3$-rarefaction wave
$\left[\rho^R,u_1^R,\ta^R\right](z)$ with $z=x/t\in \R$, explicitly defined by
\begin{eqnarray}\label{Org.RW.}
\begin{split}
\left\{\begin{array}{rll}
&\la_3\left(\rho^R(z),{u_1^{R}(z)},S_i\right)=
\left\{\begin{array}{ll}
\la_3(\rho_-,u_{1-},S_i)&\quad \mbox{for}\  z<\la_3(\rho_-,u_{1-},S_i),\\[2mm]
z&\quad \mbox{for}\ \la_3(\rho_-,u_{1-},S_i)\leq z\leq  \la_3(\rho_+,u_{1+},S_i),\\[2mm]
\la_3(\rho_+,u_{1+},S_i) &\quad \mbox{for}\ z> \la_3(\rho_+,u_{1+},S_i),
\end{array}\right.\\[8mm]
&u^R_{1}(z)-u_{1-}=
\dis{\int_{\rho_-}^{\rho^R(z)}} \sqrt{\frac{5}{3} A_i \varrho^{-\frac{4}{3}}+{\varrho^{-1}\left(\frac{d}{d\rho} (\rho_e^{-1})\right)(\varrho)}}\,d\varrho,
\\[5mm]
&
\ta^R(z)=\frac{3}{2}A_i(\rho^R(z))^{2/3}.
\end{array}
\right.
\end{split}
\end{eqnarray}
Notice $S^R(z)\equiv S_i$ for
$
S^R(z)\eqdef -\frac{2}{3}\ln \rho^R(z)+\ln (\frac{4}{3}\pi \theta^R(z)) +1.
$

In order to justify the long-time asymptotic behavior of the solution $[F(t,x,\xi),\phi(t,x)]$ to the Cauchy problem on the VPB system to the profile
$
\left[\FM_{\left[\rho^R,u^R,\theta^R\right](x/t)}(\xi),\phi^R(x/t)\right]$,
it is a usual way to deal with the stability analysis of  its smooth approximation
$
\left[\FM_{[\rho^r,u^r,\theta^r](t,x)}(\xi),\phi^r(t,x)\right]$ in the framework of small perturbations,
where corresponding to \eqref{Org.RW.}, the smooth rarefaction wave $[\rho^r,u^r,\theta^r](t,x)$ and $\phi^r(t,x)$ with $u^r(t,x)=[u_1^r(t,x),0,0]$ are defined by
\begin{eqnarray}\label{1-RW.def.}
\begin{split}
\left\{\begin{array}{rll}
&
\la_3(\rho^r(t,x),u_1^r(t,x),S_i)=w(t,x),\\[3mm]
&u^r_1(t,x)-u_{1-}
=\dis{\int_{\rho_-}^{\rho^r(t,x)}} \sqrt{\frac{5}{3} A_i \varrho^{-\frac{4}{3}}
+{\varrho^{-1}\left(\frac{d}{d\rho} (\rho_e^{-1})\right)(\varrho)}}\,d\varrho,\\[4mm]
&\ta^r(t,x)=\frac{3}{2}A_i(\rho^r(t,x))^{2/3},\quad \phi^r(t,x)=\rho_e^{-1}(\rho^r(t,x)),\\[3mm]
&\lim\limits_{x\to \pm\infty}[\rho^r,u_1^r,\theta^r](t,x)=[\rho_\pm,u_{1\pm},\theta_\pm],\quad
[\rho_+,u_{1+},\theta_+]\in R_3(\rho_-,u_{1-}, \theta_-),
\end{array}
\right.
\end{split}
\end{eqnarray}
with $w=w(t,x)$ being the solution to the Burgers' equation
\begin{equation}\label{cl.Re.cons.in}
\left\{\begin{array}{l}
\dis \pa_tw+w\pa_xw=0,\\[3mm]
w(0,x)=w_{0}(x)\eqdef \frac{1}{2}(w_{+}+w_{-})+\frac{1}{2}(w_{+}-w_{-})\tanh (\eps x),\quad w_\pm\eqdef \la_3(\rho_\pm,u_{1\pm},S_i).
\end{array}\right.
\end{equation}
Here $\eps>0$ is a constant to be chosen later on.

\subsection{The main result}
We first introduce some notations.
Let $\FM_\ast =\FM_*(\xi)=\FM_{[\rho_*,u_*,\ta_*]}(\xi)$ be a global Maxwellian such that the constant state $[\rho_\ast,u_\ast,\theta_\ast]$ with $u_\ast=[u_{1\ast},0,0]$ satisfies
\begin{eqnarray}\label{global.M.}
\left\{\begin{array}{rll}
\frac{1}{2}\sup\limits_{(t,x)\in \R_+\times \R}\ta^r(t,x)<\ta_*<\inf\limits_{(t,x)\in \R_+\times \R}\ta^r(t,x),&\\[3mm]
\sup\limits_{(t,x)\in \R_+\times \R}\left\{|\rho^r(t,x)-\rho_*|+|u^r(t,x)-u_*|+|\ta^r(t,x)-\ta_*|\right\}<\eta_0,&
\end{array}\right.
\end{eqnarray}
for a constant $\eta_0>0$ which is not necessarily small. We say $g\in L^2_{\xi}\left(\frac{1}{\sqrt{\FM_*(\xi)}}\right)$
if $\frac{g}{\sqrt{\FM_*(\xi)}}\in L^2_\xi$.
For given $T\in (0,+\infty]$, 
we define the solution space
\begin{equation*}
\widetilde{\CE}([0,T])=\left\{ h(t,x,\xi)\left|
\begin{array}{rlll}
\begin{split}
\frac{\partial^{\al}\pa^\be
h(t,x,\xi)}{\sqrt{\FM_*(\xi)}}
\in  C\left([0,T]; L^2_{x,\xi}({\R}\times{\R}^3)\right)\
\textrm{for} \ \ \ |\al|+|\be|\leq 2
\end{split}
\end{array}
\right.\right\},
\end{equation*}
associated with the norm $\widetilde{\CE}_T(\cdot)$ defined by
$$
\widetilde{\CE}_T(h)\equiv \sup\limits_{0\leq t\leq T}\sum\limits_{|\al|+|\beta|\leq 2}
\int_{{\R}\times{\R}^3}\frac{\left|\partial^\al\partial^\beta h(t,x,\xi)\right|^2}{{\bf M}_*}dxd\xi ,
$$
where $\pa^{\al}\pa^\be=\pa_t^{\al_0}\pa_x^{\al_1}\pa_\xi^{\be}$, $\pa_\xi^\be=\pa_{\xi_1}^{\be_1}\pa_{\xi_2}^{\be_2}\pa_{\xi_3}^{\be_3}$, and $|\al|=\al_0+\al_1$, $|\be|=\be_1+\be_2+\be_3$.
For conveniences later on,
we also use the similar notation $\pa^{\al'}\pa^{\be'}=\pa_t^{\al'_0}\pa_x^{\al'_1}\pa_{\xi_1}^{\be'_1}\pa_{\xi_2}^{\be'_2}\pa_{\xi_3}^{\be'_3}$ with $|\al'|=\al'_0+\al'_1$ and $|\be|=\be'_1+\be'_2+\be'_3$.


The main result of the paper is stated as follows.

\begin{theorem}\label{main.Res.}
Assume that $[\rho_+,u_{1+},\theta_+]\in R_3(\rho_-,u_{1-},\theta_-)$, $\rho_\pm=\rho_e(\phi_\pm)$ with $\phi_\pm\in (\phi_m,\phi_M)$, and the function $\rho_e(\cdot)$ satisfies the assumption $(\CA)$.  Let $\delta_r=|\rho_+-\rho_-|+|u_{1+}-u_{1-}|+|\ta_+-\ta_-|$ be the wave strength which is not necessarily small. 
There are constants 
$\eps_0>0$, $0<\sigma_0<1/3$ and $C_0>0$, which may depend on $\de_r$ and $\eta_0$,
such that if $F_0(x,\xi)\geq 0$ and
\begin{equation}\label{ID.VPB-E}
\begin{split}
\sum\limits_{|\al|+|\be|\leq2}\left\|\pa^{\al}\pa^\be\left(F_{0}(x,\xi)
-\FM_{[\rho^r,u^r,\ta^r](0,x)}(\xi)\right)\right\|^2_{L_x^2\left(L^2_\xi\left(\frac{1}{\sqrt{\FM_*(\xi)}}\right)\right)}
+\eps\leq\eps^2_0,
\end{split}
\end{equation}
where $\eps>0$ is the parameter appearing in \eqref{cl.Re.cons.}, then the Cauchy problem \eqref{VPB}, \eqref{BE.Idata1}, \eqref{VPB.b1}, \eqref{con.phi} of the VPB system admits a unique global solution $[F(t,x,\xi), \phi(t,x)]$ satisfying $F(t,x,\xi)\geq0$ and
\begin{equation}\label{VPB.sol.}
\begin{split}
\sup\limits_{t\geq0}&\sum\limits_{|\al|+|\be|\leq2}\left\|\pa^{\al}\pa^\be\left(F(t,x,\xi)-\FM_{[\rho^r,u^r,\ta^r](t,x)}(\xi)\right)
\right\|^2_{L_x^2\left(L^2_\xi\left(\frac{1}{\sqrt{\FM_*(\xi)}}\right)\right)}
\\&+\sup\limits_{t\geq0}\sum\limits_{|\al|\leq2}\left\|\pa^{\al}\left(\phi(t,x)-\rho^{-1}_e(\rho^r(t,x))\right)\right\|^2_{H^1}\leq C_0\eps_0^{2\sigma_0}.
\end{split}
\end{equation}
Moreover, it holds that
\begin{equation}\label{sol.Lab}
\begin{split}
\sup\limits_{t\rightarrow +\infty}\sup\limits_{x\in\R}\left\{\left\|F(t,x,\xi)-\FM_{[\rho^R,u^R,\ta^R](x/t)}(\xi)
\right\|_{L^2_\xi\left(\frac{1}{\sqrt{\FM_*(\xi)}}\right)}
+\left|\phi(t,x)-\rho^{-1}_e\left(\rho^R(x/t)\right)\right|\right\}
=0.
\end{split}
\end{equation}
\end{theorem}

We remark that
in  \eqref{ID.VPB-E}, all the time derivatives are understood to be the limit as $t\to 0^+$ of those terms after iteratively replacing all the time differentiations in terms of the equations of $F(t,x,\xi)$ and $[\rho^r,u^r,\theta^r](t,x)$.
%
Moreover, whenever $\eps_0>0$ is suitably small,   \eqref{ID.VPB-E} also implies that there exists a constant $C>0$ such that
\begin{multline}\label{ID.VPB2}
\sum\limits_{|\al|\leq2}\left\|\pa^{\al}\left[\rh_0(x)-\rho^r(0,x),u_0(x)-u^r(0,x),\ta_0(x)-\ta^r(0,x)\right]\right\|^2\\
+\sum\limits_{|\al|+|\be|\leq2}\left\|\pa^{\al}\pa^\be\left(\FM_{[\rho_0(x),u_0(x),\ta_0(x)]}(\xi)
-\FM_{[\rho^r,u^r,\ta^r](0,x)}(\xi)\right)\right\|_{L_x^2\left(L^2_\xi\left(\frac{1}{\sqrt{\FM_\ast(\xi)}}\right)\right)}\\
+\sum\limits_{|\al|+|\be|\leq2}\left\|\pa^{\al}\pa^\be\FG_0(x,\xi)\right\|^2
_{L_x^2\left(L^2_\xi\left(\frac{1}{\sqrt{\FM_*(\xi)}}\right)\right)}
\leq C\eps_0^2,
\end{multline}
where $F_0(x,\xi)=\FM_{[\rho_0(x),u_0(x),\ta_0(x)]}(\xi)+\FG_0(x,\xi)$ is the macro-micro decomposition  of initial data $F_0(x,\xi)$.
Note that \eqref{ID.VPB2} will also be used in the proof of Theorem \ref{main.Res.}. For completeness, the proof of  \eqref{ID.VPB2} is given in the appendix.

\subsection{Literature}


We first present the main motivations of studying the system \eqref{VPB} that we have proposed at the beginning. In general, the motion of charged particles (e.g., electrons and ions) with slab symmetry is governed by the following two-species system
\begin{equation}
\label{2svpb}
\begin{split}
\pa_t F_i +\xi_1\pa_x F_i -\frac{1}{m_i}\pa_x\phi \pa_{\xi_1}F_i &=Q(F_i,F_i)+Q(F_i,F_e),\\
\pa_t F_e+\xi_1\pa_x F_e +\frac{1}{m_e}\pa_x\phi \pa_{\xi_1}F_e &=Q(F_e,F_i)+Q(F_e,F_e),
\end{split}
\end{equation}
coupling to
\begin{equation}
\label{2svpb-f}
-\pa_x^2\phi=\int_{\R^3}F_i\,d\xi-\int_{\R^3}F_e\,d\xi.
\end{equation}
Here $F_{i,e}(t,x,\xi)$ are the number density functions for the ions and electrons respectively, and $m_{i,e}$ are their masses. The Boltzmann collision terms on the right are defined in terms of \eqref{def.bo} with the relationship \eqref{re.ppc} replaced by
\begin{equation*}
\begin{split}
\xi'&=\xi-\frac{2m_1}{m_1+m_2}[(\xi-\xi_\ast)\cdot \omega]\,\omega,\\
\xi_\ast'&=\xi_\ast+\frac{2m_2}{m_1+m_2}[(\xi-\xi_\ast)\cdot \omega]\,\omega,
\end{split}
\end{equation*}
taking in account different masses $m_1,m_2\in\{m_i,m_e\}$. The study of system \eqref{2svpb}, \eqref{2svpb-f} has recently attracted many attentions. Among them, we mention series of works by Guo \cite{G,Guo3,G-IUMJ,Guo-VPL}, including the study of the more complex Vlasov-Maxwell-Boltzmann system and the case when the Boltzmann operator is replaced by the more physical Landau collision operator for plasmas. In those works, a robust energy method is developed to treat the global stability of global Maxwellians for the Cauchy problem in perturbation regime. The key point is to construct the delicate temporal energy functional and energy dissipation rate to control the nonlinear terms along the linearized dynamics. There exist many substantial extensions basing on the Guo's approach to further study the large time behavior of solutions on torus or in the whole space, particularly the issue of rates of convergence to the global Maxwellians, for instance, we would only mention Strain-Guo \cite{SG}, Duan-Strain \cite{DS}, Yang-Yu \cite{YY}, Duan-Yang-Zhao \cite{DYZ}, Duan-Liu \cite{DL-CMP}, Wang \cite{Wa}, Xiao-Xiong-Zhao \cite{XXZ}.  Recently, the spectral analysis is also carried out by Li-Yang-Zhong \cite{LYZ} for the VPB system in the same spirit of the classical works by Ellis-Pinsky \cite{EP} and Ukai \cite{U74}, see also Glassey-Strauss \cite{GS} for the early discussion on spectrum of a general kinetic evolution operator and its application to the VPB system. We emphasize that the appearance of the self-consistent force may be able to take an essential effect on the structure of systems under consideration and induce additional analytical difficulties in the application of both the energy method and the spectrum method.

Whenever the initial data $F_0(x,\xi)$ approaches distinct global Maxwellians at far fields, typically a phase transition occurring at initial time, we may not expect that the solution to the Cauchy problem on the Boltzmann equation converges to a constant equilibrium state in large time. Instead, the solution usually tends time-asymptotically toward the wave patterns of the Boltzmann equation, such as shock wave (cf., Caflisch-Nicolaenko \cite{CN}, Liu-Yu \cite{LY, LY-im}, Yu \cite{Yu2}), rarefaction wave (cf., Liu-Yang-Yu-Zhao \cite{LYYZ}, Xin-Yang-Yu \cite{XYY}), contact discontinuity (cf., Huang-Yang \cite{HY}, Huang-Xin-Yang \cite{HXY}), and their superposition. As far as either the rarefaction wave or the contact wave is concerned, the wave profile is in the form of a local Maxwellian with its macroscopic quantities formally determined by the conservation laws with the same far-field data. To treat the stability of such nontrivial time-asymptotic local Maxwellian, another kind of energy method is initiated by Liu-Yu \cite{LY}, developed by Liu-Yang-Yu \cite{LYY}, and later improved by Yang-Zhao \cite{YZ2}. Here, the main idea of the approach is to make use of the macro-micro decomposition \eqref{def.mmd} to rewrite the nonlinear kinetic Boltzmann equation as the form of the compressible Navier-Stokes-type system, so that the analysis in the context of the viscous conservation laws can be applied. Note that the kinetic part $\FG(t,x,\xi)$ is always dissipative due to the so-called $H$-theorem.  For applications of the approach to the VPB system, see Yang-Yu-Zhao \cite{YYZ} and Yang-Zhao \cite{YZ1}. At this moment we recall that the nonlinear stability of one-dimensional wave patterns regarding the classical fluid dynamic equations has been well established, for instance, Goodman \cite{Go}, Matsumura-Nishihara \cite{MN86,MN92,MN-S}, Liu-Xin \cite{LX}, Huang-Xin-Yang \cite{HXY}, see also the textbooks \cite{CF} and \cite{S} for the general theory.

There also exists a huge number of papers to apply the Liu-Yang-Yu's approach to study the fluid dynamic limit of the nonlinear Boltzmann equation as the Knudsen number $K\!n$ which is proportional to the mean free path goes to zero. In this direction, we mention the previous classical works by Nishida \cite{Ni}, Caflisch \cite{Caf}, Ukai-Asano \cite{UA}. Recently, Huang-Wang-Wang-Yang \cite{HWWY} has succeeded in justifying the convergence of the Boltzmann equation to the compressible Euler system as $K\!n\to 0^+$ in the setting of a Riemann solution that contains the generic superposition of shock, rarefaction wave, and contact discontinuity to the full compressible Euler system; see also  some previous relative works by Yu \cite{Yu1}, Huang-Wang-Yang \cite{HWY,HWY-cmp,HWY-arma}, and Xin-Zeng \cite{XZ} and its improvement Li \cite{Li}. On the other hand, Guo \cite{G06} also developed an energy method to deal with the diffusive limit of the Boltzmann equation, that is the limit to the incompressible Navier-Stokes equations as $K\!n\to 0^+$.  Interested readers may refer further to the book chapter by Golse \cite{Gol} and the book  by Saint-Raymond \cite{SR} for the detailed representations of the topic mainly in terms of the weak compactness method.

Even though there have been extensive studies of the time-asymptotics to the wave patterns for the Boltzmann equation and the relative hydrodynamical limits as $K\!n\to 0^+$, it seems that few results are devoted to the same issue in the case of appearance of a self-consistent force, for instance, the Vlasov-type system \eqref{2svpb}, \eqref{2svpb-f}. One of the main mathematical difficulties comes from the effect of the self-consistent force on the coupling system. We observe that the macroscopic system is in the form of the compressible Euler-Poisson system up to the zero-order and the compressible Navier-Stokes-Poisson system up to the first-order. Here we should mention the work by Guo-Jang \cite{GJ} for the study of the VPB system describing the dynamics of an electron gas in a constant ion background. By using the $L^2$-$L^\infty$ method introduced in \cite{G-bd}, they prove that any solution of the VPB system near a smooth local Maxwellian with a small irrotational velocity converges global in time to the corresponding solution to the Euler-Poisson system, as $K\!n\to 0^+$.

Back to the fluid level, Duan-Yang \cite{DY} recently proved the stability of rarefaction wave and boundary layer for outflow problem on the two-fluid Navier-Stokes-Poisson  equations. We point out that due to the techniques of the proof, it was assumed in \cite{DY} that all physical parameters in the model must be unit, particularly $m_i=m_e$ and $T_i=T_e$, which is obviously unrealistic since ions and electrons generally have different masses and temperatures.  One key point used in \cite{DY} is that the large-time behavior of the electric potential is trivial and hence the two fluids indeed have the same asymptotic profiles which are constructed from the Navier-Stokes equations without any force instead of the quasineutral system.

Motivated by \cite{DY}, we studied in \cite{DL} the time-asymptotic stability of rarefaction waves for the isentropic compressible two-fluid  Navier-Stokes-Poisson system or the corresponding one-fluid system for ions under the Boltzmann relation. One of important improvements is that all physical constants appearing in the model can be taken in a general way, and the large-time profile of the electric potential is constructed on the basis of the quasineutral assumption.
Compared to the classical Navier-Stokes system without any force, the main difficulty in the proof for the Navier-Stokes-Poisson system is to treat the estimates on those terms related to the potential function $\phi$. Since the large-time behavior of $\phi$ has a slow time-decay rate and the strength of rarefaction waves is not necessarily small, it is quite nontrivial to estimate the coupling term $\pa_x\phi$ in the momentum equation as in \eqref{BE-NS}. The key point to overcome the difficulty is to use the good dissipative property from  the Poisson equation. In the case of one-fluid, the technique that we used is to expand $\rho_e(\phi)$ around the asymptotic profile up to the third-order and then make use of some cancelation property in the energy estimate. In the two-fluid case, the situation is more complicated since the dissipation of the system  becomes much weaker than that in the case of one-fluid ions. We found that the trouble term turns out to be controlled by taking the difference of two momentum equations with different weights so as to balance the different masses of fluids, which is essentially due to the symmetry of the two-fluid model.

Therefore, we expect to combine the techniques employed in \cite{DL} at the fluid level with the developed energy method at the kinetic level to deal with the stability of rarefaction waves of the VPB system \eqref{2svpb}, \eqref{2svpb-f}. In order to figure out the most technical part of the analysis, for brevity we only consider in the paper the motion of one-species VPB system \eqref{VPB} for ions under the generalized Boltzmann relation satisfying the assumption $(\CA)$. Here, we remark that the Boltzmann relation $\rho_e=\rho_e(\phi)$ has been extensively used in the mathematical study of both the fluid dynamic equations, for instance, Guo-Pausader \cite{GP},  Suzuki \cite{Su}, Nishibata-Ohnawa-Suzuki \cite{NOS-h}, and the kinetic Vlasov-type equations, for instance, Han-Kwan \cite{HK}, Charles-Despr\'es-Perthame-Sentis \cite{CDPS}.

Several closely relative problems could arise from the current work, and we would list some of them for the future considerations. First of all, it is of course an interesting problem to justify the fluid dynamic limit of the VPB system \eqref{2svpb}, \eqref{2svpb-f} or the modelling system \eqref{VPB} to the two-fluid Euler-Poisson system or the one-fluid Euler-Poisson system for ions, respectively. The setting of function spaces associated with solutions to the fluid dynamic equations can be different, for instance, as used in \cite{Ni,Caf,Yu1, HWY-cmp,XZ}, analytical or smooth solutions, or solutions containing a single wave pattern.
In the mean time, motivated by \cite{DY} and \cite{Su},  we point out that it should be an even more interesting and challenging problem to study the proposed model \eqref{VPB}  on the half space, which is related to the justification of the kinetic Bohm criterion (cf.~\cite{Ch}). After all, in the context of plasma, collisions between particles are usually described by the Boltzmann operator for long-range potentials or more physically by the classical Landau operator for the Coulomb potential taking into account the grazing effect of plasma. Thus, it is a problem to extend the current result to those interesting cases.

\subsection{Key points of the proof}


In what follows we simply outline a few key points of the proof of Theorem \ref{main.Res.} which are distinct to some extent with the previous work Liu-Yang-Yu-Zhao \cite{LYYZ} concerning the stability of the rarefaction wave for the Boltzmann equation without any force:

\begin{itemize}
  \item We work in the Eulerian coordinate instead of the Lagrangian coordinate. It is not only because it is more convenient to treat the Poisson equation and the coupling term $\pa_x \phi \pa_{\xi_1}F$ in the  Eulerian coordinate, but also it seems necessary if one would consider the same issue for the two species VPB model \eqref{2svpb}. Note that the Eulerian coordinate has been also used in \cite{DL} to deal with the Navier-Stokes-Poisson system.
  \item We choose an appropriate entropy functional to treat the zero-order energy estimate. The relative entropy functional  takes the form of
 \begin{equation}
\label{def.entr}
\eta(v,u,\theta;v^r,u^r,\theta^r)=\frac{2}{3}\theta^r \Phi\left(\frac{v}{v^r}\right) +\frac{1}{2} |u-u^r|^2+\theta^r\Phi\left(\frac{\theta}{\theta^r}\right),
\end{equation}
where $\Phi(\tau)=\tau-\ln \tau -1$, $u^r=[u_1^r,0,0]$, and $v=1/\rho$, $v^r=1/\rho^r$. The form is indeed consistent with the one in the proof of the fluid dynamic limit of the Boltzmann equation to the rarefaction wave as $K\!n\to 0^+$, for instance \cite{XZ}. We recall that the quasineutral rarefaction wave is defined in \eqref{1-RW.def.}, where due to the assumption $(\CA)$, the induced pressure $P^\phi(\cdot)$ makes no essential effect on the energy estimates, see \eqref{mdef} for instance.
  \item We carry out the energy estimates on the inner product term
  \begin{equation*}
\int_\R \pa_x [\rho (u_1-u_1^r)] (\phi-\phi^r)\,dx,
\end{equation*}
through the Poisson equation after expanding the electron density function $\rho_e(\phi)$ to the third-order. The main reason for this technique is that as mentioned before, the potential profile $\phi^r$ has a slow time-decay. On the other hand, those contributions from the first-order and second expansions enjoy some cancelation property, see the estimate on \eqref{I2}.
\item We have to introduce the velocity derivatives in the solution space to take care the forcing term  $\pa_x \phi \pa_{\xi_1}\FG$ which dose not appear in \cite{LYYZ}.  The energy method for this part is due to \cite{G}. To control the terms involving $\pa_x \phi \pa_{\xi_1}\FG$, one has to split it into two parts:
$\pa_x \phi \pa_{\xi_1}\widetilde{\FG}$ and $\pa_x \phi \pa_{\xi_1}\overline{\FG}$, and then estimate each part respectively; this is different from the works \cite{G,YYZ,YZ1,YZ3}.

\item The assumptions $(\CA_2)$ and $(\CA_3)$ assure that the delicate term
  \begin{equation}
\label{J7}
\frac{1}{2}\left(\widetilde{\phi}^2, \rho_{e}''(\phi^r)\frac{d\phi^r}{d\rho^r}\rho^r\pa_x u_1^r\right)
-\frac{1}{2}\left(\widetilde{\phi}^2, \rho_{e}'(\phi^r)\pa_xu_1^r\right),
\end{equation}
coming from \eqref{rjadcintr},
is always non-positive, and we point  out that the classical Boltzmann relation $\rh_e(\phi)=e^{\frac{\phi}{A_e}}$ looks  critical in the sense that
it can make the above expression \eqref{J7} vanish; this phenomenon has been also observed in our previous work \cite{DL}.

\item 
The energy method around the local Maxwellian that we develop in the paper is a little different from the standard one used in the previous works, for instance \cite{LYY,LYYZ,YZ2}. We have to make some extra efforts to take care
the highest order energy of the fluid component and the dissipation of $\overline{\FG}$
and $\widetilde{\FG}$; see Section \ref{sec3.2} for the detailed discussion.

\end{itemize}

\medskip
The rest of the paper is arranged as follows. In Section 2, we present the construction of the quasineutral rarefaction waves as well as their properties. In the main part Section 3, we give the priori estimates on both the fluid part and the kinetic part. The proof of the local existence is sketched in Section 4, and the proof of Theorem \ref{main.Res.} is therefore concluded in Section 5. In the Appendix, we give full details that are left in the proofs of the previous sections for completeness of the paper.

\medskip
\noindent {\it Notations.} Throughout this paper,  $C$ denotes some generic positive (generally large) constant and $\la$ denotes some generic positive (generally small) constant, where both $C$ and $\la$ may take different values in different places. $D\lesssim E$ means that  there is a generic constant $C>0$
such that $D\leq CE$. $D\sim E$
means $D\lesssim E$ and $E\lesssim D$. $\|\cdot\|_{L^p}$ $(1\leq p\leq+\infty)$
stands for the $L_x^p-$norm. Sometimes, for convenience, we use $\|\cdot\|$ to denote $L_x^2-$norm, and use $(\cdot,\cdot)$ to denote the inner product in $L^2_x$ or $L^2_{x,\xi}$.
We also use $H^{k}$ $(k\geq0)$ to denote the usual Sobolev space with respect to $x$ variable.
If each
component of $\al'$ is not greater than that of $\al$, we denote the condition by $\al'\leq \al$.
We also define $\al'<\al$ if $\al'\leq \al$ and $|\al'|<|\al|$.
 For $\al'\leq \al$, we also use $C^{\al}_{\al'}$ to denote the usual binomial coefficient.
The same notations  also apply to $\be$ and $\be'$.

\section{Rarefaction waves of the quasineutral Euler system}

It is well-known that the Riemann problem on the Burgers' equation
\begin{eqnarray*}
\left\{\begin{array}{rll}
\begin{split}
w_t+ww_x&=0,\\
w(x,0)&=w_0=\left\{\begin{array}{rll}
w_-,\ x&<0,\\
w_+,\ x&>0,\end{array}\right.
\end{split}
\end{array}\right.
\end{eqnarray*}
for two constants $w_-<w_+$, 
admits a continuous weak solution $w^{R}(x/t)$ connecting $w_-$ and
$w_+$, 
in the form of
\begin{eqnarray*}
w^{R}(x/t)=\left\{\begin{array}{ll}
w_-&,\ \frac{x}{t}< w_-,\\[3mm]
\dis \frac{x}{t}&,\ w_-\leq \frac{x}{t}\leq w_+,\\[3mm]
w_+&,\ \frac{x}{t}> w_+.
\end{array}\right.
\end{eqnarray*}
The solution to the Burgers' equation becomes smooth whenever the Riemann data is replaced by a  smooth increasing function. Here we refer to the construction introduced in \cite{MN86}
with respect to initial data whose gradient is proportional to a
parameter $\eps>0$.
%
%
%
In fact, for given constants $w_-<w_+$, the rarefaction wave $w^R(x/t)$ can be approximated by a smooth function ${w}(t,x)$ satisfying
\begin{equation}\label{cl.Re.cons.}
\left\{\begin{array}{rll}
\begin{split}
\pa_tw+w\pa_xw&=0,\\
w(0,x)&=w_{0}(x)=\frac{1}{2}(w_{+}+w_{-})+\frac{1}{2}(w_{+}-w_{-})\tanh (\eps x).
\end{split}
\end{array}\right.
\end{equation}
We now list some basic properties for the smooth rarefaction wave $w(x,t)$ as follows.

\begin{lemma}\label{cl.Re.Re.}
Let $\overline{\delta}=w_+-w_->0$ be the wave strength. 
Then the problem \eqref{cl.Re.cons.} has a unique smooth solution
$w(t,x)$, satisfying

\noindent$(i)$ $w_-<w(t,x)<w_+$, $\pa_x {w}>0$ for all $x\in\R$ and $t\geq0$.

\noindent$(ii)$ For any $1\leq p\leq+\infty$, there exists a constant $C_p$ such that for $t>0$,
$$
\|\pa_xw\|_{L^p}\leq C_p\min\left\{\overline{\delta}\eps^{1-1/p}, \overline{\delta}^{1/p}t^{-1+1/p}\right\},
$$
$$
\|\pa^j_xw\|_{L^p}\leq C_p\min\left\{\overline{\delta}\eps^{j-1/p}, \eps^{j-1-1/p}t^{-1}\right\},\ \ j\geq2.
$$

\noindent$(iii)$ $\lim\limits_{t\rightarrow+\infty}\sup\limits_{x\in\R}\left|w(t,x)-w^R(x/t)\right|=0$.

\end{lemma}

It is also well-known that for the full Euler system \eqref{ME},  the $i$-th $(i=1, 3)$ rarefaction wave can be constructed along the corresponding  rarefaction wave  curve $R_i$
when the $i$-th characteristics satisfies the inviscid Burgers' equation with increasing data. We should point out that the existence of both \eqref{Org.RW.} and \eqref{1-RW.def.} can be directly verified due to the property \eqref{pr.pp} of $P^\phi(\rho)$.
Recall \eqref{chara.} and \eqref{def.r3}. For two constant states $[\rho_\pm,u_{1\pm},\theta_\pm]$ with $[\rho_+,u_{1+},\theta_+]\in R_3(\rho_-,u_{1-},\theta_-)$, we set $w_\pm=\la_3(\rho_\pm, {u_{1\pm}}, S_i)$.
One can see that
$[\rho^{r}, u_1^{r},\ta^r]=[\rho^{r}, u_1^{r},\ta^r](t,x)$ defined in \eqref{1-RW.def.} and \eqref{cl.Re.cons.in}
is the smooth approximation of $[\rho^R,u_1^R,\ta^R](x/t)$ constructed by \eqref{Org.RW.}. We also emphasize that
$[\rho^r,u_1^r]$ satisfies the isentropic Euler-type equations
\begin{eqnarray}\label{ME2}
\left\{
\begin{array}{clll}
\begin{split}
&\pa_t\rho^r+\pa_x(\rho^ru^r_1)=0,\\[2mm]
&\pa_t u^r_1+u^r_1\pa_xu^r_1+\frac{\pa_x \left[P^r+P^\phi(\rho^r)\right]}{\rho^r}=0,
\end{split}
\end{array}
\right.
\end{eqnarray}
with
\begin{equation}\label{pr}
P^r
=A_i(\rho^r)^{5/3},
\end{equation}
and $\theta^r$ is determined by
\begin{equation*}
\theta^r=\frac{3}{2}A_i(\rho^r)^{2/3}.
\end{equation*}
With Lemma \ref{cl.Re.Re.} in hand, one has the corresponding results concerning the smooth rarefaction wave $[\rho^{r}, u_1^{r},\ta^r]$ given by \eqref{1-RW.def.} and \eqref{cl.Re.cons.in}.

\begin{lemma}\label{cl.Re.Re2.}
It holds that

\noindent$(i)$ $\pa_xu_1^{r}(t,x)>0$ and $\rho_-<\rho^{r}(t,x)<\rho_+$,
$u_{1-}<u_1^{r}(t,x)<u_{1+}$ for $x\in\R$ and
$t\geq0$.

\noindent$(ii)$ For any $1\leq p\leq+\infty$, there exists a constant $C_p$ such that for $t>0$,
$$
\left\|\pa_x\left[\rho^{r}, u_1^{r},\ta^r\right]\right\|_{L^p}\leq C_p\min\left\{\delta_r\eps^{1-1/p}, \delta_r^{1/p}t^{-1+1/p}\right\},
$$
$$
\left\|\pa^j_x\left[\rho^{r}, u_1^{r},\ta^r\right]\right\|_{L^p}\leq C_p\min\left\{\delta_r\eps^{j-1/p}, \eps^{j-1-1/p}t^{-1}\right\},\ \ j\geq2.
$$
\noindent$(iii)$ $\lim\limits_{t\rightarrow+\infty}\sup\limits_{x\in\R}\left|\left[\rho^{r}, u_1^{r},\ta^r\right](t,x)-\left[\rho^R,u_1^R,\ta^R\right](x/t)\right|=0$.

\end{lemma}

\begin{proof}
We prove only $(ii)$ for brevity. Recalling \eqref{1-RW.def.} and \eqref{chara.}, one has
\begin{eqnarray*}
\left\{\begin{array}{rll}
\begin{split}
w=&u^r_1+\sqrt{(\pa_\rho P) (\rho^r,S_i) +\rho^r\left(\frac{d}{d\rho} (\rho_e^{-1})\right)(\rho^r)},\\
u^r_1&=u_{1-}
+\dis{\int_{\rho_-}^{\rho^r}}
\sqrt{\frac{5}{3} A_i \varrho^{-\frac{4}{3}}+\varrho^{-1}\left(\frac{d}{d\rho} (\rho_e^{-1})\right)(\varrho)}d\varrho,
\end{split}
\end{array}\right.
\end{eqnarray*}
from which as well as \eqref{1-RW.def.}, it follows that
\begin{eqnarray}\label{rut-w}
\left\{\begin{array}{rll}
\begin{split}
\pa_x \rho^r=&\frac{\pa_x w}{ \sqrt{\frac{5}{3} A_i (\rho^r)^{-\frac{4}{3}}+(\rho^r)^{-1}\left(\frac{d}{d\rho}(\rho_e^{-1})\right)(\rho^r)}
+\frac{(\pa^2_\rho P) (\rho^r,S_i)+\left(\frac{d}{d\rho} (\rho_e^{-1})\right)(\rho^r)+\rho^r\left(\frac{d^2}{d\rho^2} (\rho_e^{-1})\right)(\rho^r)}{2\sqrt{(\pa_\rho P) (\rho^r,S_i) +\rho^r\left(\frac{d}{d\rho} (\rho_e^{-1})\right)(\rho^r)}}}\\[2mm]
\pa_x u^r_1=&\sqrt{\frac{5}{3} A_i (\rho^r)^{-\frac{4}{3}}+(\rho^r)^{-1}\left(\frac{d}{d\rho} (\rho_e^{-1})\right)(\rho^r)}~\pa_x \rho^r\\[2mm]
\pa_x \ta^r=&A_i(\rho^r)^{-\frac{1}{3}}\pa_x\rho^r.
\end{split}
\end{array}\right.
\end{eqnarray}
On the other hand, by $(\CA_3)$, we see that
$$
\frac{d}{d\rho} (\rho_e^{-1})(\rho^r)+\rho^r\left(\frac{d^2}{d\rho^2} (\rho_e^{-1})\right)(\rho^r)\geq0.
$$
This together with the assumption $(\CA)$ and the definition $P(\rho,S_i)=k e^{S_i}\rho^{5/3}$ implies that the coefficient function of $\pa_x w$ on the right hand side of the first equation of \eqref{rut-w} is smooth in $\rho^r$ on the interval $[\rho_-,\rho_+]$ with $\rho_->0$. Thus, we can obtain
$$
\left\|\pa_x\rho^{r}\right\|_{L^p}\leq C\left\|\pa_x w\right\|_{L^p}\leq C_p\min\left\{\delta_r\eps^{1-1/p}, \delta_r^{1/p}t^{-1+1/p}\right\},
$$
and by an induction argument,
$$
\left\|\pa^j_x\rho^{r}\right\|_{L^p}\leq  C_p\min\left\{\delta_r\eps^{j-1/p}, \eps^{j-1-1/p}t^{-1}\right\},\ \ j\geq2.
$$
From the second and third equations of \eqref{rut-w}, the similar ones are true for $u_1^r$ and $\theta^r$.
Therefore $(ii)$ holds. This completes the proof of Lemma \ref{cl.Re.Re2.}.
\end{proof}

\section{The a priori estimates}
In this section, we will deduce the a priori energy estimates
for the Cauchy problem \eqref{VPB}, \eqref{BE.Idata1}, \eqref{VPB.b1}, \eqref{con.phi}, \eqref{VPB.b2}. First of all, let us define the macroscopic perturbation
$$
\left[\widetilde{\rho}, \widetilde{u}, \widetilde{\ta}, \widetilde{\phi}\right](t,x)=\left[\rho-\rho^{r},u-u^{r}, \ta-\ta^r,  \phi-\phi^{r}\right](t,x),
$$
as well as
$$
\widetilde{S}=S-S_i,
$$
where we recall that $[\rho^r,u^r,\theta^r,\phi^r]$ solving \eqref{MEt} is defined in \eqref{1-RW.def.} and \eqref{cl.Re.cons.in}, and $\widetilde{S}$ is given due to \eqref{BE.enp.} by
$$
\widetilde{S}=-\frac{2}{3} \ln \frac{\rho}{\rho^r}+\ln \frac{\theta}{\theta^r}.
$$
Then $\left[\widetilde{\rho}, \widetilde{u}, \widetilde{\ta},  \widetilde{\phi}\right](t,x)$ satisfies
\begin{eqnarray}
&&\pa_t\widetilde{\rho}+\pa_x(\rho u_1)-\pa_x(\rho^ru^r_1)=0,\label{trho0}\\
&&\pa_t \widetilde{u}_1+u_1\pa_x u_1-u^r_1\pa_x u^r_1+\frac{\pa_x P}{\rho}-\frac{\pa_xP^r}{\rho^r}+\pa_x\widetilde{\phi}
=\frac{3}{\rho}\pa_x\left(\mu(\theta)\pa_x u_1\right)
-\frac{1}{\rho}\int_{{\R}^3}\xi_1^2\pa_x\Theta\,d\xi,\label{tu10}\\
&&\pa_t \widetilde{u}_i+{u_1\pa_x\widetilde{u}_i}
=\frac{1}{\rho}\pa_x\left(\mu(\theta)\pa_x \widetilde{u}_i\right)-\frac{1}{\rho}\int_{{\R}^3}\xi_1\xi_i\pa_x\Theta\,d\xi,\
\ i=2,\ 3,\label{tui0}\\
&&\pa_t \widetilde{\theta}+u_1\pa_x\ta-u^r_1\pa_x\ta^r+\frac{P\pa_x u_1}{\rho}-\frac{P^r\pa_x u^r_1}{\rho^r}\notag\\&&
\quad=\frac{1}{\rho}\pa_x\left(\kappa(\theta)\pa_x \theta\right)+\frac{3}{\rho}\mu(\theta)(\pa_x u_1)^2+
\sum\limits_{i=2}^3\frac{1}{\rho}\mu(\theta)(\pa_x \widetilde{u}_i)^2
-\frac{1}{\rho}\int_{{\R}^3}\left(\frac{|\xi|^2}{2}-u\cdot\xi\right)\xi_1\pa_x\Theta\,d\xi,\label{tta0}\\
&&-\pa^2_x\widetilde{\phi}=\widetilde{\rho}+\rho_e(\phi^r)-\rho_e(\phi)
+\pa^2_x\phi^r,\label{tphy}\\
&&\left[\widetilde{\rho},\widetilde{u}, \widetilde{\ta}\right](0,x)=\left[\widetilde{\rho}_0,\widetilde{u}_0, \widetilde{\ta}_0\right](x)=\left[\rho_0(x)-\rho^r(0,x),u_0(x)-u^r(0,x),\ta_0(x)-\ta^r(0,x)\right],\label{tid}
\end{eqnarray}
as well as
\begin{eqnarray}
&&\pa_t \widetilde{S}+u_1\pa_x\widetilde{S}
=\frac{1}{\rho\theta}\pa_x\left(\kappa(\theta)\pa_x \theta\right)+\frac{3}{\rho\theta}\mu(\theta)(\pa_x u_1)^2+
\sum\limits_{i=2}^3\frac{1}{\rho\theta}\mu(\theta)(\pa_x u_i)^2\notag\\
&&\qquad\quad\qquad\qquad-\frac{1}{\rho\theta}\int_{{\R}^3}\left(\frac{|\xi|^2}{2}-u\cdot\xi\right)\xi_1\pa_x\Theta\,d\xi,\label{tS0}
\end{eqnarray}
where $P^r$ is defined as \eqref{pr} and $\Theta$ is given by \eqref{Ta.def}. We note that $\widetilde{u}_i=u_i$ for $i=2, 3$ and $\widetilde{\phi}(t,x)$ is determined by
the elliptic equation \eqref{tphy} under the boundary condition that $\widetilde{\phi}(t,x)\rightarrow0$ as $x\rightarrow\pm\infty$.
We also point out that the structural identity \eqref{tphy} will be of extremal importance for the later proof.

To the end  we use $\FM_i$ to denote $\FM_*$ or $\FM$ for brevity. Since
$$
\left\|\frac{\FG}{\sqrt{\FM_i}}\right\|^2_{L^2_{x,\xi}}
$$
is not integrable with respect to the time variable, it is necessary to consider the following perturbation
\begin{equation*}
\widetilde{\FG}=\FG-\overline{\FG},
\end{equation*}
where
\begin{equation}\label{def.ng}
\overline{\FG}=\frac{3 L_{\FM}^{-1}
\left\{{\bf P}_1^{\FM}\left[\xi_1\FM\left(\xi_1\pa_xu_1^r+\frac{|\xi-u|^2}{2\ta}\pa_x\ta^r\right)\right]\right\}}{2\ta}.
\end{equation}

To prove Theorem \ref{main.Res.}, the key point is to deduce the a priori energy estimates on the macroscopic part $\left[\widetilde{\rho}, \widetilde{u}, \widetilde{\ta}, \widetilde{\phi}\right]$ and the microscopic parts $\FG$ and $\widetilde{\FG}$ based on the following a priori assumption
\begin{equation}\label{aps}
\begin{split}
N^2(T)\equiv&\sup\limits_{0\leq t\leq T}
\left\|\left[\widetilde{\rho}, \widetilde{u}, \widetilde{\ta}\right](t)\right\|^2
+\sup\limits_{0\leq t\leq T}\sum\limits_{|\al|=1}
\left\|\pa^{\al}\left[\rho, u, \ta\right](t)\right\|^2
\\&+\sup\limits_{0\leq t\leq T} \sum\limits_{1\leq|\al|\leq
2}{\displaystyle
\int_{\R\times\R^3}} \frac{\left|\partial^{\al} F(t,x,\xi)\right|^2}{{\bf M}_*}dxd\xi
+\sup\limits_{0\leq t\leq T} {\displaystyle\int_{\R\times\R^3}}\frac{\left|\widetilde{\FG}(t,x,\xi)\right|^2}{{\bf M}_*}dxd\xi
\\&+\sup\limits_{0\leq t\leq T} \sum\limits_{|\al|+|\be|\leq
2\atop{|\be|\geq1}} {\displaystyle
\int_{\R\times\R^3}}
\frac{\left|\pa^{\al}\partial^\be\widetilde{\FG}(t,x,\xi)\right|^2}{{\bf M}_*}dxd\xi
+\eps\leq \eps_0^2,
\end{split}
\end{equation}
for an arbitrary positive time $T.$ Here we note that the above bound for $N(T)$ yields the following consequences. First, we have
\begin{equation}\label{aps2}
\begin{split}
\sup\limits_{0\leq t\leq T}&\sum\limits_{|\al|=2}
{\displaystyle\int_{\R\times\R^3}}\frac{\left|\partial^{\al}\FM(t,x,\xi)\right|^2}{{\bf M}}dxd\xi
+\sup\limits_{0\leq t\leq T} \sum\limits_{|\al|=
2} {\displaystyle\int_{\R\times\R^3}}\frac{\left|\partial^{\al}\FG(t,x,\xi)\right|^2}{{\bf M}}dxd\xi
 \\ \leq& 2\sup\limits_{0\leq t\leq T} \sum\limits_{|\al|=
2} \left|{\displaystyle\int_{\R\times\R^3}}\frac{\pa^{\al}\FM\partial^{\al}\FG(t,x,\xi)}{{\bf M}}dxd\xi \right|
+\sup\limits_{0\leq t\leq T} \sum\limits_{|\al|=
2} {\displaystyle\int_{\R\times\R^3}}\frac{\left|\partial^{\al} F(t,x,\xi)\right|^2}{{\bf M}_*}dxd\xi \\
\leq& C\sup\limits_{0\leq t\leq T} \sum\limits_{|\al'|=1,|\al|=2} \left(\int_{{\R}}\left|\pa^{\al'}[u,\ta](t)\right|^4dx\right)^{1/2}
\left(\int_{\R\times{\R}^3}\frac{|\partial^{\al}\FG(t,x,\xi)|^2}{\FM}dxd\xi \right)^{1/2}
\\&+\sup\limits_{0\leq t\leq T} \sum\limits_{|\al|=
2} {\displaystyle\int_{\R\times\R^3}}\frac{\left|\partial^{\al} F(t,x,\xi)\right|^2}{{\bf M}_*}dxd\xi
\\
\leq& {C\eps_0}\sum\limits_{1\leq|\al'|\leq
2} \sup\limits_{0\leq t\leq T}\left\|\pa^{\al'}[u,\ta](t)\right\|^2
+{C\eps_0}\sum\limits_{|\al|=2}\sup\limits_{0\leq t\leq T}\int_{{\R}\times{\R}^3}\frac{|\partial^{\al}\FG(t,x,\xi)|^2}{\FM}dxd\xi
\\&+\sup\limits_{0\leq t\leq T} \sum\limits_{|\al|=
2} {\displaystyle\int_{\R\times\R^3}}\frac{\left|\partial^{\al} F(t,x,\xi)\right|^2}{{\bf M}_*}dxd\xi .
\end{split}
\end{equation}
Moreover,
\begin{equation}\label{aps3}
\begin{split}
\sum\limits_{|\al|=
2}& \sup\limits_{0\leq t\leq T}\left\|\pa^{\al}[\rho,u,\ta](t)\right\|^2
\\ \leq& C\sup\limits_{0\leq t\leq T}\sum\limits_{|\al|=2}
{\displaystyle\int_{\R\times\R^3}}\frac{\left|\partial^{\al}\FM(t,x,\xi)\right|^2}{{\bf M}}dxd\xi
+C\sup\limits_{0\leq t\leq T}\sum\limits_{|\al'|=
1} \left\|\left|\pa^{\al'}\left[\rho,u,\ta\right](t)\right|^2\right\|^2\\
\leq& C\sup\limits_{0\leq t\leq T}\sum\limits_{|\al|=2}
{\displaystyle\int_{\R\times\R^3}}\frac{\left|\partial^{\al}\FM(t,x,\xi)\right|^2}{{\bf M}}dxd\xi
+C\eps_0^2\sup\limits_{0\leq t\leq T}\sum\limits_{1\leq|\al|\leq
2} \left\|\pa^{\al}[\rho,u,\ta](t)\right\|^2.
\end{split}
\end{equation}
Therefore \eqref{aps3} together with \eqref{aps2} imply
\begin{equation}\label{aps4}
\begin{split}
\sum\limits_{|\al|=
2} \sup\limits_{0\leq t\leq T}\left\|\pa^{\al}[\rho,u,\ta](t)\right\|^2
+\sup\limits_{0\leq t\leq T} \sum\limits_{|\al|=
2} {\displaystyle\int_{\R\times\R^3}}\frac{\left|\partial^{\al}\FG(t,x,\xi)\right|^2}{{\bf M}_\ast}dxd\xi
\leq C\eps_0^2.
\end{split}
\end{equation}
One can also see that \eqref{aps} and \eqref{aps4} lead to the following a priori estimate
\begin{equation}\label{aps.phy}
\sup\limits_{0\leq t\leq T}\left\{
\left\|\widetilde{\phi}(t)\right\|^2
+\sum\limits_{1\leq|\al|\leq2}
\left\|\pa^{\al}\phi(t)\right\|^2_{H^1}\right\}
\leq C\sup\limits_{0\leq t\leq T}\sum\limits_{ |\al|\leq
2}
\left\|\pa^\al\widetilde{\rho}(t)\right\|^2+C\eps\leq C\eps_0^2.
\end{equation}
In fact \eqref{aps.phy} follows from the standard elliptic estimates for the Poisson equation \eqref{tphy}.

\begin{remark}\label{rem.rjads}
Letting $\eps_0$ be small enough, by the a priori assumption \eqref{aps} and in view of \eqref{global.M.}, one sees that
\begin{eqnarray}\label{global.M2.}
\left\{\begin{array}{rll}
\frac{1}{2}\sup\limits_{(t,x)\in[0,T]\times{\R}}\ta(t,x)<\ta_*<\inf\limits_{(t,x)\in[0,T]\times{\R}}\ta(t,x),&\\[3mm]
\sup\limits_{(t,x)\in[0,T]\times{\R}}|\rho(t,x)-\rho_*|+|u(t,x)-u_*|+|\ta(t,x)-\ta_*|<\eta_0,&
\end{array}\right.
\end{eqnarray}
where $\eta_0$ is the constant given in \eqref{global.M.}.
We point out that \eqref{global.M2.} will be frequently used in the later energy estimates.
\end{remark}

The subsequent {two} subsections are devoted to deducing the desired energy type estimates based on the
a priori assumption (\ref{aps}) and the estimates on $[\rho^{r}, u_1^{r},\ta^r]$ in Lemma \ref{cl.Re.Re2.}.
The first one is concentrated on the energy estimates on the macroscopic part.

\subsection{Energy estimates on the macroscopic part}
In this subsection, we consider the energy estimates on
$\left[\widetilde{\rho}, \widetilde{u}, \widetilde{\ta}, \widetilde{\phi}\right](t,x)$. The main result is given as follows.

\begin{proposition}\label{mac.eng.lem.}
Assume that all the conditions in Theorem \ref{main.Res.} hold, and for $T>0$, $\Theta$ is given by \eqref{Ta.def}
with $\FG\in \widetilde{\CE}([0,T])$. Let $\left[\widetilde{\rho}, \widetilde{u}, \widetilde{\ta},  \widetilde{\phi}\right](t,x)$
be a smooth solution to the Cauchy problem
\eqref{trho0}, \eqref{tu10}, \eqref{tui0}, \eqref{tta0}, 
\eqref{tphy} and \eqref{tid} on $0\leq t\leq T$
and satisfy \eqref{aps}. Then there exist constants $0<\sigma_0<1/3$, {$\zeta_0>1$,  $\zeta_1>1$}, and an energy functional $\CE_1(\widetilde{\rho},\widetilde{u},\widetilde{\theta},\widetilde{\phi})$ with
\begin{equation*}
\CE_1(\widetilde{\rho},\widetilde{u},\widetilde{\theta},\widetilde{\phi})
\sim \sum\limits_{|\al|\leq1}\left\{\left\|\pa^{\al}\left[\widetilde{\rho}, \widetilde{u}, \widetilde{\ta}\right](t)\right\|^2
+\left\|\pa^{\al} \widetilde{\phi}(t)\right\|_{H^1}^2\right\},
\end{equation*}
such that the following energy estimate holds
\begin{equation}\label{macro.eng}
\begin{split}
\frac{d}{dt}&\CE_1(\widetilde{\rho},\widetilde{u},\widetilde{\theta},\widetilde{\phi})
-\ka_0\frac{d}{dt}\sum\limits_{|\al|=1}\left(\pa^{\al}\widetilde{u}_1,\pa^{\al}\pa_x\widetilde{\rho}\right)
\\&+\la\left\{\left\|\sqrt{\pa_xu_1^r}\left[ \widetilde{\rho}, \widetilde{u}_1, \widetilde{S}\right](t)\right\|^2
+\sum\limits_{1\leq|\al|\leq 2}\left\|\pa^{\al}\left[\widetilde{\rho},\widetilde{u},\widetilde{\theta}\right](t)\right\|^2
+\sum\limits_{1\leq|\al|\leq2}\left\|\pa^{\al}\widetilde{\phi}(t)\right\|_{H^1}^2\right\}\\
\lesssim&(1+t)^{-\zeta_1}\left\|\left[\widetilde{\rho},\widetilde{u},\widetilde{\ta},\widetilde{\phi}\right](t)\right\|^2
+\eps^{\sigma_0}(1+t)^{-\zeta_0}
+\sum\limits_{1\leq|\al|\leq2}\int_{\R\times{\R}^3}\frac{1+|\xi|}{\FM}|\pa^{\al}\FG|^2 dxd\xi
\\&+\eps_{0}\sum\limits_{|\al|\leq1}\int_{\R\times\R^3}\frac{|\pa_{\xi_1}\pa^{\al}\widetilde{\FG}|^2}{\FM}dxd\xi
+\eps_{0}\int_{{\R}\times{\R}^3}\frac{(1+|\xi|)\left|\widetilde{\FG}\right|^2}{\FM_{i}}dxd\xi ,
\end{split}
\end{equation}
where $\ka_0$ is a small positive constant.
\end{proposition}

\begin{proof} We divide it by the following three steps.

\medskip
\noindent{\bf Step 1.} {\it Zero-order energy estimates.}
It is known {(cf.~\cite{LYY}, for instance)} that the zero-order energy estimates for the Navier-Stokes type system \eqref{trho0}, \eqref{tu10}, {\eqref{tui0}} and \eqref{tta0}
can not be {directly} derived by the usual $L^2$ energy method. To overcome this difficulty, one way is to find and make use of a suitable entropy and entropy-flux. For this,
let us introduce an entropy
\begin{equation*}
\widetilde{\eta}=\widetilde{\eta}(\widetilde{\rho},\widetilde{u},\widetilde{\theta})=\widetilde{\ta}+\frac{1}{2}|\widetilde{u}|^2+\overline{P}^{\,r}\widetilde{v}-\theta^r\widetilde{S},
\end{equation*}
where $\widetilde{v}=v-v^r=\frac{1}{\rho}-\frac{1}{\rho^r}=-\frac{\widetilde{\rho}}{\rho\rho^r}$  and $\overline{P}^{\,r}=\overline{P}(v^r,S^r)=k e^{S^r}(v^r)^{-5/3}=P(\rho^r,S^r)$. Notice that the form of the relative entropy is consistent with \eqref{def.entr}, and the reason why we use the above equivalent form is that it seems more convenient for us to derive the consequent equations of $\widetilde{\eta}$.
One can see that
there is a constant $C_1>0$ such that
\begin{equation}
\label{en.iqv}
\frac{1}{C_1}\left\{\widetilde{\rho}^2+|\widetilde{u}|^2+\widetilde{\theta}^2\right\}\leq \widetilde{\eta}(\widetilde{\rho},\widetilde{u},\widetilde{\theta})\leq C_1\left\{\widetilde{\rho}^2+|\widetilde{u}|^2+\widetilde{\theta}^2\right\},
\end{equation}
according to the property of the pressure function $P$ and the a priori assumption \eqref{aps}.

In order to use $\widetilde{\eta}(\widetilde{\rho},\widetilde{u},\widetilde{\theta})$, we rewrite \eqref{trho0}, \eqref{tu10}, \eqref{tui0}, \eqref{tta0} and \eqref{tS0} respectively
as
\begin{eqnarray}
&&\pa_t\widetilde{v}-(v\pa_x u_1-v^r\pa_xu^r_1)+u_1\pa_xv-u^r_1\pa_xv^r=0,\label{trho}\\
&&\pa_t \widetilde{u}_1+u_1\pa_x u_1-u^r_1\pa_x u^r_1+v\pa_x \overline{P}-v^r\pa_x\overline{P}^{\,r}+\pa_x\widetilde{\phi}
=3v\pa_x\left(\mu(\theta)\pa_x u_1\right)
-v\int_{{\R}^3}\xi_1^2\pa_x\Theta \,d\xi,\label{tu1}\\
&&\pa_t \widetilde{u}_i+{u_1\pa_x\widetilde{u}_i}
=v\pa_x\left(\mu(\theta)\pa_x \widetilde{u}_i\right)-v\int_{{\R}^3}\xi_1\xi_i\pa_x\Theta \,d\xi,\
\ i=2,\ 3,\label{tui}\\
&&\pa_t \widetilde{\theta}+u_1\pa_x\ta-u^r_1\pa_x\ta^r+v\overline{P}\pa_x u_1-v^r\overline{P}^{\,r}\pa_x u^r_1\notag\\&&
\quad=v\pa_x\left(\kappa(\theta)\pa_x \theta\right)+3v\mu(\theta)(\pa_x u_1)^2+
\sum\limits_{i=2}^3v\mu(\theta)(\pa_x \widetilde{u}_i)^2
-v\int_{{\R}^3}\left(\frac{|\xi|^2}{2}-u\cdot\xi\right)\xi_1\pa_x\Theta \,d\xi,\label{tta}
\end{eqnarray}
and
\begin{eqnarray}
&&\pa_t \widetilde{S}+u_1\pa_x\widetilde{S}
=\frac{v}{\theta}\pa_x\left(\kappa(\theta)\pa_x \theta\right)+3\frac{v}{\theta}\mu(\theta)(\pa_x u_1)^2+
\sum\limits_{i=2}^3\frac{v}{\theta}\mu(\theta)(\pa_x u_i)^2\notag\\
&&\qquad\quad\qquad\qquad-\frac{v}{\theta}\int_{{\R}^3}\left(\frac{|\xi|^2}{2}-u\cdot\xi\right)\xi_1\pa_x\Theta \,d\xi.\label{tS}
\end{eqnarray}
Here $\overline{P}=\overline{P}(v,S)=k e^{S}v^{-5/3}=P(\rho,S)$.
In view of \eqref{trho}, \eqref{tu1}, \eqref{tui}, \eqref{tta}, \eqref{tS} and \eqref{tphy},
by a straightforward calculation, it follows that
\begin{equation}\label{b.entropy0}
\begin{split}
\pa_t&(\rho\widetilde{\eta})+\pa_x(\rho u_1\widetilde{\eta})
=\rho\pa_t\widetilde{\eta}+\rho u_1\pa_x\widetilde{\eta}
\\=&-\left(\overline{P}-\overline{P}^{\,r}-\pa_{v}\overline{P}(v^r,S^r)\widetilde{v}-\pa_{S}\overline{P}(v^r,S^r)\widetilde{S}\right)\pa_xu^r_{1}
-\rho\widetilde{u}^2_1\pa_xu^r_{1}-\rho\widetilde{u}_1\pa_x\widetilde{\phi}
\\&-\pa_x\left[\widetilde{u}_1(\overline{P}-\overline{P}^{\,r})\right]+\frac{\widetilde{v}^3}{vv^r}\pa_{v^r}\overline{P}^{\,r}\pa_xu^r_{1}
+\rho\widetilde{u}_1\overline{P}^{\,r}\pa_xv^r\widetilde{S}-\rho^r \widetilde{v}\overline{P}^{\,r}\pa_xu^r_{1}\widetilde{S}
-\widetilde{\rho}\widetilde{v}\overline{P}^{\,r}\pa_xu^r_{1}\widetilde{S}
\\&-\frac{\widetilde{v}^2}{v^r}\pa_{v^r}\overline{P}^{\,r}\pa_xu^r_{1}
+\frac{\widetilde{\ta}}{\theta}\pa_x\left(\kappa(\theta)\pa_x \theta\right)+3\frac{\widetilde{\ta}}{\theta}\mu(\theta)(\pa_x u_1)^2
+\sum\limits_{i=2}^3\frac{\widetilde{\ta}}{\theta}\mu(\theta)(\pa_x \widetilde{u}_i)^2
\\&-\frac{\widetilde{\ta}}{\theta}\int_{{\R}^3}\left(\frac{|\xi|^2}{2}-u\cdot\xi\right)\xi_1\pa_x\Theta \,d\xi
+3\widetilde{u}_1\pa_x\left(\mu(\theta)\pa_x u_1\right){+\sum\limits_{i=2}^3\widetilde{u}_i\pa_x\left(\mu(\theta)\pa_x \widetilde{u}_i\right)}
\\&-\widetilde{u}_1\int_{{\R}^3}\xi_1^2\pa_x\Theta \,d\xi{-\sum\limits_{i=2}^3\widetilde{u}_i\int_{{\R}^3}\xi_1\xi_i\pa_x\Theta \,d\xi},
\end{split}
\end{equation}
where we have also used the fact that $\pa_{v^r}\overline{P}^{\,r}=\pa_v\overline{P}(v^r,S^r)$, $\overline{P}^{\,r}=\pa_S\overline{P}(v^r,S^r)$,
$S^r=S_i=\textrm{constant}$, and $\pa_x\ta^r=-\overline{P}^{\,r}\pa_xv^r.$

Then \eqref{b.entropy0} and \eqref{tui}  imply that
\begin{equation}\label{b.entropy}
\begin{split}
\frac{d}{dt}&\int_{\R}(\rho\widetilde{\eta})(t,x)dx
+\underbrace{\int_{\R}\left(\overline{P}-\overline{P}^{\,r}-\pa_{v}\overline{P}(v^r,S^r)\widetilde{v}
-\pa_{S}\overline{P}(v^r,S^r)\widetilde{S}\right)\pa_xu^r_{1}dx}_{I_1}
+\int_{\R}(v^r)^{-1}\widetilde{u}^2_1\pa_xu^r_{1}dx
\\&
+3\int_{\R}\mu(\theta)\left(\pa_x \widetilde{u}_1\right)^2dx
+{\sum\limits_{i=2}^3\int_{\R}\mu(\theta)\left(\pa_x \widetilde{u}_i\right)^2dx}
+\int_{\R}\frac{\kappa(\theta)}{\ta}\left(\pa_x \widetilde{\theta}\right)^2dx
\\=&\underbrace{-\int_{\R}\rho\widetilde{u}_1\pa_x\widetilde{\phi}dx}_{I_2}
\underbrace{+3\int_{\R}\mu(\theta)\frac{\widetilde{\theta}}{\theta}(\pa_x u_1)^2dx
+\sum\limits_{i=2}^3\int_{\R}\mu(\theta)\frac{\widetilde{\theta}}{\theta}\left(\pa_x \widetilde{u}_i\right)^2dx}_{I_3}
\\&\underbrace{-\int_{\R}\int_{{\R}^3}\xi_1^2\widetilde{u}_1\pa_x\Theta dxd\xi
-{\sum\limits_{i=2}^3\int_{\R}\int_{{\R}^3}\xi_1\xi_i\widetilde{u}_i\pa_x\Theta dxd\xi}
-\int_{\R}\int_{{\R}^3}\frac{\widetilde{\ta}}{\ta}\left(\frac{|\xi|^2}{2}-u\cdot\xi\right)\xi_1\pa_x\Theta dxd\xi }_{I_{4}}
\\&\underbrace{+3\int_{\R}\widetilde{u}_1\pa_x\left(\mu(\theta)\pa_x u^r_1\right)dx
+\int_{\R}\frac{\widetilde{\theta}}{\theta}\pa_x\left(\kappa(\theta)\pa_x \theta^r\right)dx
+\int_{\R}\kappa(\theta)\frac{\widetilde{\theta}\pa_x\ta}{\theta^2}\pa_x \widetilde{\theta}dx}_{I_5}
\\&\underbrace{-\int_{\R}\rho^r \widetilde{v}\overline{P}^{\,r}\pa_xu^r_{1}\widetilde{S}dx
-\int_{\R}\frac{\widetilde{v}^2}{v^r}\pa_{v^r}\overline{P}^{\,r}\pa_xu^r_{1}dx}_{I_6}
\underbrace{+\int_{\R}\rho\widetilde{u}_1\overline{P}^{\,r}\pa_xv^r\widetilde{S}dx}_{I_7}
\\&\underbrace{-\int_{\R}\widetilde{\rho}\widetilde{v}\overline{P}^{\,r}\pa_xu^r_{1}\widetilde{S}dx
+\int_{\R}\frac{\widetilde{v}^3}{vv^r}\pa_{v^r}\overline{P}^{\,r}\pa_xu^r_{1}dx-\int_{\R}\widetilde{\rho}\widetilde{u}^2_1\pa_xu^r_{1}dx}_{I_{8}}.
\end{split}
\end{equation}
We now turn to compute $I_l$ $(1\leq l\leq 8)$ term by term. The procedure of the proof can be outlined as follows.
Since $I_6$, $I_7$ and $I_{8}$ depend on $I_1$ and the {third} term on the
left hand side of \eqref{b.entropy}, we first estimate $I_1$, $I_6$, $I_7$ and $I_{8}$ by putting them together and taking full advantage of the non-negativity of $I_1$. Then we compute $I_2$, which needs to be treated carefully too. The estimations for $I_3$, $I_4$ and $I_5$ will be much easier and thus left to the end of this step.


\begin{lemma}\label{lem3.1}
\begin{equation}\label{I1sum}
\int_{\R}(v^r)^{-1}\widetilde{u}^2_1\pa_xu^r_{1}dx+I_1-I_6-I_7-I_{8}
\geq \la \int_{\R}\left|\left[\widetilde{v},\widetilde{u}_1,\widetilde{S}\right]\right|^2\pa_xu^r_{1}dx.
\end{equation}
\end{lemma}

\begin{proof}
Noticing that $\overline{P}(v,S)=kv^{-5/3}e^{S}$ and $S^r=S_i=\textrm{constant}$, we have
\begin{equation*}
\begin{split}
I_1=&\frac{1}{2}\int_{\R} \left\{\pa_{v}^2\overline{P}(v^r,S^r)\widetilde{v}^2
+2\frac{\pa^2 \overline{P}}{\pa v\pa S}(v^r,S^r)\widetilde{v}\widetilde{S}+\pa_{S}^2\overline{P}(v^r,S^r)\widetilde{S}^2\right\}\pa_x u^r_1 dx+O(1)\int_{\R}\left|\left[\widetilde{v}^3,\widetilde{S}^3\right]\right|\pa_x u^r_1dx
\\=&A_i\underbrace{\int_{\R}\left\{\frac{20}{9}(v^r)^{-\frac{11}{3}}\widetilde{v}^2-\frac{5}{3}(v^r)^{-\frac{8}{3}}\widetilde{v}\widetilde{S}
+\frac{1}{2}(v^r)^{-\frac{5}{3}}\widetilde{S}^2\right\}\pa_x u^r_1 dx}_{I_0}+O(1)\underbrace{\int_{\R}\left|\left[\widetilde{v}^3,\widetilde{S}^3\right]\right|\pa_x u^r_1 dx}_{I_{9}}.
\end{split}
\end{equation*}
By a simple calculation, we obtain
\begin{equation*}
\begin{split}
I_6=&A_i\int_{\R}\left\{\frac{5}{3}(v^r)^{-\frac{11}{3}}\widetilde{v}^2-(v^r)^{-\frac{8}{3}}\widetilde{v}\widetilde{S}\right\}\pa_x u^r_1 dx.
\end{split}
\end{equation*}
On the other hand, by virtue of \eqref{1-RW.def.}, one can see that
$$
\pa_xv^r=-\frac{(v^r)^2}{\sqrt{\frac{5}{3} A_i (v^{r})^{\frac{4}{3}}+v^r\left(\frac{d}{d\rho}(\rho_e^{-1})\right)(1/v^r)}}\pa_xu_1^r,
$$
from which, it follows that
\begin{equation*}
\begin{split}
I_7=&-\int_{\R}\rho\widetilde{u}_1\overline{P}^{\,r}\pa_xu^r_1\widetilde{S}\frac{(v^r)^2}{\sqrt{\frac{5}{3} A_i (v^{r})^{\frac{4}{3}}
+v^r\left(\frac{d}{d\rho}(\rho_e^{-1})\right)(1/v^r)}}~dx
\\ =&\underbrace{-\int_{\R}\widetilde{u}_1\overline{P}^{\,r}\pa_xu^r_1\widetilde{S}\frac{v^r}{\sqrt{\frac{5}{3} A_i (v^{r})^{\frac{4}{3}}
+v^r\left(\frac{d}{d\rho}(\rho_e^{-1})\right)(1/v^r)}}~dx}_{I_{10}}
\\&\underbrace{-\int_{\R}\widetilde{\rho}\widetilde{u}_1\overline{P}^{\,r}\pa_xu^r_1\widetilde{S}\frac{(v^r)^2}{\sqrt{\frac{5}{3} A_i (v^{r})^{\frac{4}{3}}+v^r\left(\frac{d}{d\rho}(\rho_e^{-1})\right)(1/v^r)}}~dx}_{I_{11}}.
\end{split}
\end{equation*}
Consequently,
\begin{equation}\label{mdef}
\begin{split}
\int_{\R}&(v^{r})^{-1}\widetilde{u}^2_1\pa_xu^r_{1}dx+I_0-I_6-I_{10}
=\int_\R \left[\widetilde{v},\widetilde{u}_1,\widetilde{S}\right] M \left[\widetilde{v},\widetilde{u}_1,\widetilde{S}\right]^T\pa_xu_1^r\,dx,
\end{split}
\end{equation}
with the real symmetric matrix $M$ given by
\begin{equation*}
\begin{pmatrix}
\   \frac{5}{9} A_i (v^r)^{-\frac{11}{3}} &\ \ \ \  -\frac{1}{3} A_i (v^r)^{-\frac{8}{3}} &\ \  0\\[6mm]
\   * & \frac{1}{2} A_i (v^r)^{-\frac{5}{3}} &\ \ \  \frac{\frac{1}{2} A_i (v^r)^{-\frac{2}{3}}}{\sqrt{\frac{5}{3} A_i (v^r)^{\frac{4}{3}}
+v^r\left(\frac{d}{d\rho}(\rho_e^{-1})\right)(1/v^r)}}\\[6mm]
\   0 & *  & (v^r)^{-1}
\end{pmatrix}.
\end{equation*}
It is straightforward to check that $M$ is positive-definite, since its all leading principal minors are strictly positive, i.e.
\begin{eqnarray*}
&&\De_{11}>0,\\
&&\De_{22}=\frac{1}{6}A_i^2 (v^r)^{-\frac{16}{3}}>0,\\
&&\De_{33}=\frac{1}{6}A_i^2 (v^r)^{-\frac{19}{3}} -\frac{5}{36} A_i^3 \frac{(v^r)^{-\frac{15}{3}}}{ \frac{5}{3} A_i (v^r)^{\frac{4}{3}}
+ v^r\left(\frac{d}{d\rho}(\rho_e^{-1})\right)(1/v^r)}>0.
\end{eqnarray*}
Recalling Sobolev's inequality,
\begin{equation}\label{sob.ine.}
\|f\|_{L^\infty}\leq\sqrt{2}\|f\|^{1/2}\|\pa_xf\|^{1/2} \ \  \textrm{for any}\ \ f\in H^1,
\end{equation}
we see that $I_{9}$ can be controlled by
$$
C\eps_0\left\|\sqrt{\pa_x u^r_1} \left[\widetilde{v},\widetilde{S}\right]\right\|^2,
$$
according to \eqref{aps}.

Similarly, for $I_{8}$ and $I_{11}$, one has
$$
|I_{8}|+|I_{11}|\lesssim A_i\eps_0\int_{\R}\left\{(v^r)^{-\frac{11}{3}}\widetilde{v}^2
+(v^r)^{-\frac{5}{3}}\widetilde{S}^2+\widetilde{u}_1^2\right\}\pa_x u^r_1 dx.
$$
Combing the above estimates on $I_1$, $I_6$, $I_7$, $I_{8}$, $I_{9}$, $I_{10}$ and $I_{11}$, we thus arrive at \eqref{I1sum}. The proof of Lemma \ref{lem3.1} is complete.
\end{proof}

Let us now consider the most delicate term $I_2$.
The key technique to handle $I_2$ is to use the good dissipative property of
the Poisson equation by expanding $\rho_e(\phi)$ around the asymptotic profile up to the third-order.
Only in this way, can we observe some new cancelations and obtain the higher order nonlinear terms.

\begin{lemma}\label{lem3.2}
\begin{multline}\label{I2sum}
\left|I_2+\frac{d}{dt}\left[\frac{1}{2}\left(\pa_x\widetilde{\phi},\pa_x \widetilde{\phi}\right)+\frac{1}{2}\left(\widetilde{\phi}^2,\rho_{e}'(\phi^r)\right)+\frac{1}{3}\left(\widetilde{\phi}^3,\rho_{e}''(\phi^r)\right)\right]\right.\\
\left.-\frac{1}{2}\left(\widetilde{\phi}^2, \left( \rho_{e}''(\phi^r)\frac{d\phi^r}{d\rho^r}\rho^r-\rho_{e}'(\phi^r)\right)\pa_xu_1^r\right)\right|
\\
\lesssim (\eta+\eps_0C_\eta +\eps_0^2) {\left\|\pa_x \left[\widetilde{\rho}, \widetilde{u}_1,\widetilde{\phi}\right]\right\|^2}
+C_\eta(1+t)^{-2}\|[\widetilde{\rho},\widetilde{u}_1]\|^2\\
+C_\eta(1+t)^{-4/3}\left\|\widetilde{\phi}\right\|^2+\eps C_\eta (1+t)^{-2}.
\end{multline}
\end{lemma}

\begin{proof}
In light of \eqref{trho0} and \eqref{tphy} and by integration by parts, one has
\begin{equation}\label{I2}
\begin{split}
I_2=\int_{\R}\pa_x(\rho\widetilde{u}_1)\widetilde{\phi}\,dx
=\left(\widetilde{\phi},\pa_t\pa_x^2\widetilde{\phi}^2\right)
\underbrace{+\left(\widetilde{\phi},\pa_t\left(\rho_e(\phi^r)-\rho_e(\phi)\right)\right)}_{I_{2,1}}
+\left(\widetilde{\phi},\pa_t\pa_x^2\phi^r\right)
\underbrace{-\left(\widetilde{\phi},\pa_x(\widetilde{\rho}u^r_1)\right)}_{I_{2,2}}.
\end{split}
\end{equation}
We now turn to compute the right hand side of \eqref{I2} term by term. It is straightforward to see that
\begin{equation*}
\left(\widetilde{\phi}, \pa_t\pa_x^2\widetilde{\phi}\right)=-\frac{1}{2}\frac{d}{dt}\left(\pa_x\widetilde{\phi}, \pa_x\widetilde{\phi}\right).
\end{equation*}
For the third term on the right hand side of \eqref{I2}, by integration by parts and employing Lemma \ref{cl.Re.Re2.} and Cauchy-Schwarz's inequality with $0<\eta<1$, we obtain
\begin{equation*}
\left|\left(\widetilde{\phi},\pa_t\pa_x^2\phi^r\right)\right|\leq \eta\left\|\pa_x\widetilde{\phi}\right\|^2+{C_\eta}\eps(1+t)^{-2}.
\end{equation*}

\medskip
\noindent{\it Estimates on $I_{2,1}$.}
For $I_{2,1}$, we first get from the Taylor's formula with an integral remainder
that
\begin{equation}\label{taylor}
\rho_{e}(\phi^r)-\rho_{e}(\phi)=-\rho'_{e}(\phi^r)\widetilde{\phi}-\frac{1}{2}\rho_{e}''(\phi^r)\widetilde{\phi}^2
\underbrace{-\int_{\phi^r}^{\phi}\frac{(\varrho-\phi)^2}{2}\rho_{e}'''(\varrho) d\varrho}_{J_1}.
\end{equation}
Then it follows that
\begin{equation}\label{ztu.ip3}
\begin{split}
I_{2,1}=-\left(\widetilde{\phi}, \pa_t\left(\rho_{e}'(\phi^r)\widetilde{\phi}\right)\right)
-\frac{1}{2}\left(\widetilde{\phi}, \pa_t\left(\rho_{e}''(\phi^r)\widetilde{\phi}^2\right)\right)
+\left(\widetilde{\phi}, \pa_tJ_1\right).
\end{split}
\end{equation}
To compute the right hand side of \eqref{ztu.ip3}, we first consider $\left(\widetilde{\phi}, \pa_tJ_1\right)$. Note that
\begin{multline}\label{tJ1}
J_1\sim \widetilde{\phi}^3,\ \ \pa_t J_1=\pa_t\phi\int_{\phi^r}^{\phi}(\varrho-\phi)\rho_{e}'''(\varrho) d\varrho
+\frac{1}{2}\widetilde{\phi}^2\pa_t\phi^r\int_{\phi^r}^{\phi}\rho_{e}'''(\varrho) d\varrho
\\ \sim \pa_t\phi \widetilde{\phi}^2+\pa_t\phi^r\widetilde{\phi}^3= \pa_t\widetilde{\phi} \widetilde{\phi}^2
+\pa_t\phi^r\widetilde{\phi}^2+\pa_t\phi^r\widetilde{\phi}^3.
\end{multline}
In addition, it follows from \eqref{tphy} that
\begin{multline*}
-\left(\pa_t\pa^2_x\widetilde{\phi}
,\pa_t\widetilde{\phi}\right)+
\left(\rho_{e}'(\phi^r)\pa_t\widetilde{\phi},\pa_t\widetilde{\phi}\right)\\
=\left(\pa_t\widetilde{\rho},\pa_t\widetilde{\phi}\right)-\left(\pa_t\left(\rho_{e}'(\phi^r)\right)\widetilde{\phi},\pa_t\widetilde{\phi}\right)
-\frac{1}{2}\left(\pa_t\left(\rho_{e}''(\phi^r)\widetilde{\phi}^2\right),\pa_t\widetilde{\phi}\right)
+\left(\pa_t J_1,\pa_t\widetilde{\phi}\right)
+\left(\pa_t\pa^2_x\phi^r,\pa_t\widetilde{\phi}\right),
\end{multline*}
which implies
\begin{equation}\label{ptphyL2}
\left\|\sqrt{\rho_{e}'(\phi^r)}\pa_t\widetilde{\phi}\right\|^2+\left\|\pa_t\pa_x\widetilde{\phi}\right\|^2\leq
C\left\|\pa_x\left[\widetilde{\rho},\widetilde{u}_1\right]\right\|^2
+C(1+t)^{-2}{\left\|\widetilde{\phi}\right\|^2}
+C\eps(1+t)^{-2},
\end{equation}
according to $\eqref{trho0}$, \eqref{sob.ine.}, \eqref{tJ1}, Lemma \ref{cl.Re.Re2.} and Cauchy-Schwarz's inequality. 

With  \eqref{tJ1} and \eqref{ptphyL2} in hand, we get from Lemma \ref{cl.Re.Re2.} and H\"older's inequality as well as the Cauchy-Schwarz  inequality with $0<\eta<1$ that
\begin{equation*}
\begin{split}
\left|\left(\widetilde{\phi}, \pa_tJ_1\right)\right|\lesssim& C_\eta\left\|\widetilde{\phi}^3\right\|^2+\left\|\widetilde{\phi}\pa_x[\rho^r,u_1^r]\right\|^2+\eta\|\pa_t\widetilde{\phi}\|^2
+C\left|\left(\widetilde{\phi}^3, \pa_x[\rho^r,u_1^r]\right)\right|\\
\lesssim& \max\{\eps_0C_\eta,\eta\}\left\|\pa_x\left[\widetilde{\rho},\widetilde{u}_1,\widetilde{\phi}\right]\right\|^2
+(1+t)^{-4/3}\left\|\widetilde{\phi}\right\|^2
+\eps(1+t)^{-2},
\end{split}
\end{equation*}
where Sobolev's inequality \eqref{sob.ine.} has been also used to obtain the bounds:
\begin{equation}\label{J1.p1}
\left\|\widetilde{\phi}^3\right\|^2\lesssim \left\|\pa_x\widetilde{\phi}\right\|^2\left\|\widetilde{\phi}\right\|^4\lesssim \eps_0^4\left\|\pa_x\widetilde{\phi}\right\|^2,
\end{equation}
and 
\begin{equation}\label{J1.p2}
\begin{split}
\left|\left(\widetilde{\phi}^3, \pa_x[\rho^r,u_1^r]\right)\right|\lesssim& \left\|\pa_x[\rho^r,u_1^r]\right\|_{L^\infty}\left\|\widetilde{\phi}\right\|^{1/2}\left\|\pa_x\widetilde{\phi}\right\|^{1/2}\left\|\widetilde{\phi}\right\|^2\\
\lesssim & \left\|\widetilde{\phi}\right\|^{2}\left\|\pa_x\widetilde{\phi}\right\|^{2}
+\|\pa_x[\rho^r,u_1^r]\|^{4/3}_{L^\infty}\left\|\widetilde{\phi}\right\|^{8/3}
\\ \lesssim&\eps_0^2\left\|\pa_x\widetilde{\phi}\right\|^{2}+\eps_0^{2/3}(1+t)^{-4/3}\left\|\widetilde{\phi}\right\|^2.
\end{split}
\end{equation}
As to the first two terms on the right hand side of \eqref{ztu.ip3},
invoking the first equation of \eqref{ME2},
we obtain
\begin{equation*}
\begin{split}
-\left(\widetilde{\phi}, \pa_t\left(\rho_{e}'(\phi^r)\widetilde{\phi}\right)\right)=&-\frac{1}{2}\frac{d}{dt}\left(\widetilde{\phi}^2, \rho_{e}'(\phi^r)\right)-\frac{1}{2}\left(\widetilde{\phi}^2, \pa_t\left(\rho_{e}'(\phi^r)\right)\right)
\\=&-\frac{1}{2}\frac{d}{dt}\left(\widetilde{\phi}^2, \rho_{e}'(\phi^r)\right)\underbrace{+\frac{1}{2}\left(\widetilde{\phi}^2, \rho_{e}''(\phi^r)\frac{d\phi^r}{d\rho^r}\pa_x(\rho^ru_1^r)\right)}_{J_2},
\end{split}
\end{equation*}
and
\begin{equation*}
\begin{split}
-\frac{1}{2}\left(\widetilde{\phi}, \pa_t\left(\rho_{e}''(\phi^r)\widetilde{\phi}^2\right)\right)=&-\frac{1}{3}\frac{d}{dt}\left(\widetilde{\phi}^3, \rho_{e}''(\phi^r)\right)-\frac{1}{6}\left(\widetilde{\phi}^3,
\pa_t (\rho_e''(\phi^r))
\right)
\\=&-\frac{1}{3}\frac{d}{dt}\left(\widetilde{\phi}^3, \rho_{e}''(\phi^r)\right)+\underbrace{\frac{1}{6}\left(\widetilde{\phi}^3, \rho_{e}'''(\phi^r)\frac{d\phi^r}{d\rho^r}\pa_x(\rho^ru_1^r)\right)}_{J_3}.
\end{split}
\end{equation*}
Here $J_3$ can be treated as in \eqref{J1.p1}.
It is worthwhile pointing out that $J_2$ can not be directly controlled for the time being, and its estimate should be postponed to
the subsequent estimates on $I_{2,2}$  by an exact cancellation with other terms.

\medskip
\noindent {\it Estimates on $I_{2,2}$.}
As to $I_{2,2}$, we have from \eqref{tphy} that
\begin{equation*}
\begin{split}
-\left(\widetilde{\phi}, \pa_x(\widetilde{\rho}u_1^r)\right)
&=-\left(\widetilde{\phi}, \pa_x\widetilde{\rho}u_1^r\right)-\left(\widetilde{\phi}, \widetilde{\rho}\pa_xu_1^r\right)\\
&=\left(\widetilde{\phi}, \left(\pa_x^3\widetilde{\phi}
+\pa_x\left(\rho_{e}(\phi^r)-\rho_{e}(\phi)\right)+\pa_x^3\phi^r\right)u_1^r\right)\\
&\quad +\left(\widetilde{\phi}, \left(\pa_x^2\widetilde{\phi}
+\rho_{e}(\phi^r)-\rho_{e}(\phi)+\pa_x^2\phi^r\right)\pa_xu_1^r\right)\\
&=\underbrace{\frac{1}{2}\left(\pa_xu_1^r, \left(\pa_x\widetilde{\phi}\right)^2\right)-\left(\pa_x\widetilde{\phi},\pa_x^2\phi^ru_1^r\right)}_{J_4}
\underbrace{+\left(\widetilde{\phi},
\pa_x\left(\rho_{e}(\phi^r)-\rho_{e}(\phi)\right)u_1^r\right)}_{J_5}\\
&\quad \underbrace{+\left(\widetilde{\phi}, \left(\rho_{e}(\phi^r)-\rho_{e}(\phi)\right)\pa_xu_1^r\right)}_{J_6},
\end{split}
\end{equation*}
where the last identity holds true due to the following identities:
$$
\left(\widetilde{\phi}, \pa_x^3\widetilde{\phi}
u_1^r\right)+\left(\widetilde{\phi}, \pa_x^2\widetilde{\phi}
\pa_x u_1^r\right)=\frac{1}{2}\left(\pa_xu_1^r, \left(\pa_x\widetilde{\phi}\right)^2\right),
$$
and
$$
\left(\widetilde{\phi}, \pa_x^3\phi^ru_1^r\right)
+\left(\widetilde{\phi}, \pa_x^2\phi^r\pa_xu_1^r\right)=-\left(\pa_x\widetilde{\phi},\pa_x^2\phi^ru_1^r\right).
$$
Notice that $|J_4|$ is dominated by
$$
(\eta+{C\eps})\left\|\pa_x\widetilde{\phi}\right\|^2+C_\eta\eps(1+t)^{-2},
$$
according to {the} Cauchy-Schwarz  inequality with $0<\eta<1$ and Lemma \ref{cl.Re.Re2.}.

We now use \eqref{taylor} to expand $J_5$ and $J_6$ respectively as 
\begin{equation*}
\begin{split}
J_5
=&-\left(\widetilde{\phi}, \pa_x\left(\rho_{e}'(\phi^r)\widetilde{\phi}\right)u_1^r\right)
-\frac{1}{2}\left(\widetilde{\phi}, \pa_x\left(\rho_{e}''(\phi^r)\widetilde{\phi}^2\right)u_1^r\right)
+\left(\widetilde{\phi}, \pa_xJ_1u_1^r\right)\\
=&-\frac{1}{2}\left(\widetilde{\phi}^2, \rho_{e}''(\phi^r)\frac{d\phi^r}{d\rho^r}\pa_x\rho^ru_1^r\right)
+\frac{1}{2}\left(\widetilde{\phi}^2, \rho_{e}'(\phi^r)\pa_xu_1^r\right)
\\&-\frac{1}{6}\left(\widetilde{\phi}^3, \rho_{e}'''(\phi^r)\frac{d\phi^r}{d\rho^r}\pa_x\rho^ru_1^r\right)+\frac{1}{3}\left(\widetilde{\phi}^3, \rho_{e}''(\phi^r)\pa_xu_1^r\right)
+\left(\widetilde{\phi}, \pa_xJ_1u_1^r\right),
\end{split}
\end{equation*}
and
\begin{equation*}
\begin{split}
J_6=-\left(\widetilde{\phi}^2, \rho_{e}'(\phi^r)\pa_xu_1^r\right)
-\frac{1}{2}\left(\widetilde{\phi}^3, \rho_{e}''(\phi^r)\pa_xu_1^r\right)
+\left(\widetilde{\phi}, J_1\pa_xu_1^r\right).
\end{split}
\end{equation*}
Owing to these, we find
\begin{equation}\label{rjadcintr}
\begin{split}
J_2+J_5+J_6=&\underbrace{\frac{1}{2}\left(\widetilde{\phi}^2, \rho_{e}''(\phi^r)\frac{d\phi^r}{d\rho^r}\rho^r\pa_x u_1^r\right)
-\frac{1}{2}\left(\widetilde{\phi}^2, \rho_{e}'(\phi^r)\pa_xu_1^r\right)}_{J_7}
\\&+\underbrace{\left(\widetilde{\phi}, \pa_xJ_1u_1^r\right)
+\left(\widetilde{\phi}, J_1\pa_xu_1^r\right)}_{J_8}
\underbrace{-\frac{1}{6}\left(\widetilde{\phi}^3, \rho_{e}'''(\phi^r)\frac{d\phi^r}{d\rho^r}\pa_x\rho^ru_1^r\right)-\frac{1}{6}\left(\widetilde{\phi}^3, \rho_{e}''(\phi^r)\pa_xu_1^r\right)}_{J_9}.
\end{split}
\end{equation}
Due to the assumptions $(\CA_2)$, $(\CA_3)$, and $\pa_x u_1^r>0$, one can see that
$$
J_7=\frac{1}{2}\left(\widetilde{\phi}^2\pa_x u_1^r, \frac{\rho_{e}''(\phi^r)\rho_e(\phi^r)-[\rho_e'(\phi^r)]^2}{\rho_e'(\phi^r)}\right)\leq 0.
$$
For $J_8$, from \eqref{J1.p1}, it follows that
$$
|J_8|=\left|-\left(\pa_x\widetilde{\phi}, J_1u_1^r\right)\right|
\leq C \eps_0^2\left\|\pa_x\widetilde{\phi}\right\|^2.
$$
Finally, $J_9$ can be handled in the same way as in \eqref{J1.p2}.

Recalling \eqref{I2} and collecting all estimates above, we thereby complete the estimate on the term $I_2$ in the way of \eqref{I2sum}. The proof of Lemma \ref{lem3.2} is complete.
\end{proof}

Now we turn to estimate $I_3$, $I_4$ and $I_5$.
Noticing that $\mu(\theta)$ is a smooth function of $\ta$, we see by using Cauchy-Schwarz's inequality with $\eta$, Lemma \ref{cl.Re.Re2.}, \eqref{sob.ine.} and
\eqref{aps} that for $1<\sigma_1< 5/3$,
\begin{equation}\label{I3}
\begin{split}
|I_3|\leq& C\left\|\widetilde{\theta}\right\|_{L^\infty}\|\pa_x\widetilde{u}\|^2+C\left\|\widetilde{\theta}\right\|_{L^\infty}\|\pa_x u_1^r\|^2\\
\leq&C\left\|\widetilde{\theta}\right\|^{1/2}\left\|\pa_x\widetilde{\theta}\right\|^{1/2}\|\pa_x\widetilde{u}\|^2
+C\left\|\widetilde{\theta}\right\|^{1/2}\left\|\pa_x\widetilde{\theta}\right\|^{1/2}\|\pa_x u_1^r\|^2\\
\leq&C\left\|\widetilde{\theta}\right\|_{H^1}\|\pa_x\widetilde{u}\|^2+\eta\left\|\pa_x\widetilde{\theta}\right\|^{2}
+C_\eta\left\|\widetilde{\theta}\right\|^{2}\|\pa_x u_1^r\|^{\sigma_1+1}+C_\eta\|\pa_x u_1^r\|^{\frac{7-\sigma_1}{2}}\\
\leq&C\left\|\widetilde{\theta}\right\|_{H^1}\|\pa_x\widetilde{u}\|^2+\eta\left\|\pa_x\widetilde{\theta}\right\|^{2}
+C_\eta\left\|\widetilde{\theta}\right\|^{2}\|\pa_x u_1^r\|^{\sigma_1+1}+C_\eta\|\pa_x u_1^r\|_{L^\infty}^{\frac{7-\sigma_1}{4}}\|\pa_x u_1^r\|_{L^1}^{\frac{7-\sigma_1}{4}}\\
\leq&C\left\|\widetilde{\theta}\right\|_{H^1}\|\pa_x\widetilde{u}\|^2+\eta\left\|\pa_x\widetilde{\theta}\right\|^{2}
+C_\eta\left\|\widetilde{\theta}\right\|^{2}\|\pa_x u_1^r\|^{\sigma_1+1}
+C_\eta\|\pa_x u_1^r\|^{\frac{7-\sigma_1}{8}}\|\pa^2_x u_1^r\|^{\frac{7-\sigma_1}{8}}\|\pa_x u_1^r\|_{L^1}^{\frac{7-\sigma_1}{4}}\\
\leq&C\eps_0\|\pa_x\widetilde{u}\|^2+\eta\left\|\pa_x\widetilde{\theta}\right\|^{2}
+C_\eta(1+t)^{-\frac{\sigma_1+1}{2}}\left\|\widetilde{\theta}\right\|^{2}
+C_\eta\eps^{\frac{7-\sigma_1}{16}}(1+t)^{-\frac{21-3\sigma_1}{16}}.
\end{split}
\end{equation}
Recalling the definition \eqref{Ta.def}, by integration by parts and utilizing Lemma \ref{est.nonop},
Corollary \ref{inv.L.}, Cauchy-Schwarz's inequality and Sobolev's inequality, one has
\begin{equation*}
\begin{split}
|I_4|\leq& C_\eta\int_{\R\times\R^3}\frac{\Theta^2}{\FM}dxd\xi
+{\eta\sum\limits_{i=1}^3\int_{\R\times\R^3}|\xi|^4\FM|\pa_x\widetilde{u}_i|^2dxd\xi}
+\eta\int_{\R\times\R^3}|\xi|^6\FM\left|\pa_x\left(\frac{\widetilde{\theta}}{\theta}\right)\right|^2dxd\xi
\\&+\eta\sum\limits_{i=1}^3\int_{\R\times\R^3}|\xi|^4\FM\left|\pa_x\left(\frac{\widetilde{\theta}}{\theta}\widetilde{u}_i\right)\right|^2dxd\xi
+\eta\sum\limits_{i=1}^3\int_{\R\times\R^3}|\xi|^4\FM\left|\pa_x\left(\frac{\widetilde{\theta}}{\theta}u^r_i\right)\right|^2dxd\xi
\\
\leq&(\eta+\eps_0)\left\|\pa_x\left[\widetilde{u},\widetilde{\theta}\right]\right\|^2
+C(1+t)^{-2}\left\|\left[\widetilde{u},\widetilde{\theta}\right]\right\|^2
+C_\eta\sum\limits_{|\al|=1}\int_{\R\times\R^3}\frac{(1+|\xi|)|\pa^{\al}\FG|^2}{\FM}dxd\xi
\\&+C_\eta\int_{\R\times\R^3}\frac{|\pa_x\phi|^2|\pa_{\xi_1}(\widetilde{\FG}+\overline{\FG})|^2}{\FM}dxd\xi
+C_\eta\int_{\R}\left(\int_{\R^3}\frac{(1+|\xi|)|\FG|^2}{\FM_i}d\xi\right) \left(\int_{\R^3}\frac{|\FG|^2}{\FM_i}d\xi\right) dx\\
\leq&(\eta+\eps_0)\left\|\pa_x\left[\widetilde{u},\widetilde{\theta}\right]\right\|^2
+C(1+t)^{-2}\left\|\left[\widetilde{u},\widetilde{\theta}\right]\right\|^2
+C_\eta\sum\limits_{|\al|=1}\int_{\R\times\R^3}\frac{(1+|\xi|)|\pa^{\al}\FG|^2}{\FM}dxd\xi
\\&+C_\eta\int_{\R}\left(\int_{\R^3}\frac{(1+|\xi|)|\FG|^2}{\FM_i}d\xi\right) \left(\int_{\R^3}\frac{|\FG|^2}{\FM_i}d\xi\right) dx
+C_\eta\|\pa_x\phi\|_{L^\infty}^2
\int_{\R\times\R^3}\frac{|\pa_{\xi_1}\widetilde{\FG}|^2}{\FM}dxd\xi
\\&+C_\eta\int_{\R\times\R^3}|\pa_x\phi|^2|\pa_{x}[u^r,\ta^r]|^2dxd\xi
\\
\leq&(\eta+\eps_0)\left\|\pa_x\left[\widetilde{u},\widetilde{\theta}\right]\right\|^2+C_\eta\eps_0\left\|\pa_x\widetilde{\phi}\right\|^2_{H^1}
+C(1+t)^{-2}\left\|\left[\widetilde{u},\widetilde{\theta}\right]\right\|^2+C_\eta\eps(1+t)^{-2}
\\&+C_\eta\eps_0\int_{\R\times\R^3}\frac{|\pa_{\xi_1}\widetilde{\FG}|^2}{\FM}dxd\xi
+C_\eta\sum\limits_{|\al|=1}\int_{\R\times\R^3}\frac{(1+|\xi|)|\pa^{\al}\FG|^2}{\FM}dxd\xi
+C_\eta\eps_{0}\int_{{\R}\times{\R}^3}\frac{(1+|\xi|)\left|\widetilde{\FG}\right|^2}{\FM_{i}}dxd\xi.
\end{split}
\end{equation*}
Here the following crucial estimate has been used:
\begin{equation}\label{db.G}
\begin{split}
\int_{\R}\left(\int_{\R^3}\frac{(1+|\xi|)|\FG|^2}{\FM_i}d\xi\right) \left(\int_{\R^3}\frac{|\FG|^2}{\FM_i}d\xi\right) dx
\leq&
C\int_{{\R}}\left(\int_{{\R}^3}\frac{(1+|\xi|)\left|\widetilde{\FG}+\overline{\FG}\right|^2}{\FM_{i}}d\xi\right)
\left(\int_{{\R}^3}\frac{\left|\widetilde{\FG}+\overline{\FG}\right|^2}{\FM_{i}}d\xi\right) dx\\
\leq& \eps_{0}\int_{{\R}\times{\R}^3}\frac{(1+|\xi|)\left|\widetilde{\FG}\right|^2}{\FM_{i}}dxd\xi
+C\int_{{\R}}\left|\pa_x[u^r,\ta^r]\right|^4 dx\\
\leq& \eps_{0}\int_{{\R}\times{\R}^3}\frac{(1+|\xi|)\left|\widetilde{\FG}\right|^2}{\FM_{i}}dxd\xi
+C\eps(1+t)^{-2}.
\end{split}
\end{equation}
Moreover, we also have used the formula
\begin{multline}\label{d.iL}
\pa^{\al}\pa^\be\left\{L_{\FM}^{-1}h\right\}=L_{\FM}^{-1}(\pa^{\al} \pa^\be h)\\
-\sum\limits_{j=0}^{|\al|+|\be|-1}\sum\limits_{|\al'|+|\be'|=j}C^{\al,\be}_{\al',\be'}
L_{\FM}^{-1}\Big\{Q\left(\pa^{\al'}\pa^{\be'}\left(L_{\FM}^{-1}h\right),\pa^{\al-\al'}\pa^{\be-\be'}\FM\right)\\
+Q\left(\pa^{\al-\al'}\pa^{\be-\be'}\FM,\pa^{\al'}\pa^{\be'}\left(L_{\FM}^{-1}h\right)\right)\Big\},
\end{multline}
as well as Corollary \ref{inv.L.} and Lemma \ref{est.nonop}
to deduce that
\begin{equation}\label{d.bG}
\begin{split}
\int_{\R^3}\frac{|\pa_{\xi_1}\overline{\FG}|^2}{\FM_i}d\xi
\leq&
C\sum_{|\be|\leq 1}\int_{{\R}^3}\frac{(1+|\xi|)\left|L_{\FM}^{-1}\pa^\be \left\{{\bf P}_1^{\FM}\left[\xi_1\FM\left(\xi_1\pa_xu_1^r+\frac{|\xi-u|^2}{2\ta}\pa_x\ta^r\right)\right]\right\} \right|^2}{\FM_{i}}d\xi
\\
\leq & C\sum_{|\be|\leq 1}\int_{{\R}^3}\frac{(1+|\xi|)^{-1}\left|\pa^\be \left\{{\bf P}_1^{\FM}\left[\xi_1\FM\left(\xi_1\pa_xu_1^r+\frac{|\xi-u|^2}{2\ta}\pa_x\ta^r\right)\right]\right\} \right|^2}{\FM_{i}}d\xi
\\ \leq& C|\pa_x[u_1^r,\ta^r]|^2.
\end{split}
\end{equation}
It should be also noted that
$$
\dis{\int_{\R\times\R^3}}\frac{|\pa_x\phi|^2|\pa_{\xi_1}\FG|^2}{\FM}dxd\xi
$$
can not be directly controlled.
One has to use the splitting $\FG=\overline{\FG}+\widetilde{\FG}$
and estimate $\overline{\FG}$ and $\widetilde{\FG}$ respectively. This is different from the previous works \cite{YYZ, YZ1}, where
$$
\dis{\int_{\R\times\R^3}}\frac{|\pa_{\xi_1}\FG|^2}{\FM}dxd\xi
$$
is integrable with respect to time.

We now undertake to estimate $I_5$, by performing the similar calculations as those for obtaining  \eqref{I3}.
Since $\mu(\ta)$ and $\kappa(\ta)$ are smooth functions of $\ta$, it follows that
\begin{equation*}
\begin{split}
|I_5|\lesssim& \int_{\R}\left|\pa_x^2[u_1^r,\ta^r]\left[\widetilde{u}_1,\widetilde{\ta}\right]\right|dx
+\int_{\R}\left|\widetilde{\ta}\left(\pa_x\widetilde{\ta}\right)^2\right|dx
+\int_{\R}\left|\left[\widetilde{u}_1,\widetilde{\ta}\right]\pa_x\left[\widetilde{\ta},\ta^r\right]\pa_x[u_1^r,\ta^r]\right|dx.
\end{split}
\end{equation*}
Next, letting $0<\sigma_2<1$, by applying H\"older's inequality, Young's inequality, Lemma \ref{cl.Re.Re2.} and the Sobolev  inequality \eqref{sob.ine.},
we have
\begin{equation}\label{I51}
\begin{split}
\int_{\R}\left|\pa_x^2[u_1^r,\ta^r]\left[\widetilde{u}_1,\widetilde{\ta}\right]\right|dx
\lesssim&\|\pa_x^2[u^r_1,\ta^r]\|_{L^{1+\sigma_2}}\left\|\left[\widetilde{u}_1,\widetilde{\ta}\right]\right\|_{L^{\frac{1+\sigma_2}{\sigma_2}}}\\
\lesssim& \eps^{\frac{\sigma_2}{1+\sigma_2}}(1+t)^{-1}\left\|\pa_x\left[\widetilde{u}_1,\widetilde{\ta}\right]\right\|^{\frac{1-\sigma_2}{1+\sigma_2}}_{L^{\infty}}
\left\|\left[\widetilde{u}_1,\widetilde{\ta}\right]\right\|^{\frac{2\sigma_2}{1+\sigma_2}}\\
\lesssim&\eps^{\frac{\sigma_2}{1+\sigma_2}}(1+t)^{-1}\left\|\pa_x\left[\widetilde{u}_1,\widetilde{\ta}\right]\right\|^{\frac{1-\sigma_2}{2(1+\sigma_2)}}
\left\|\left[\widetilde{u}_1,\widetilde{\ta}\right]\right\|^{\frac{1+3\sigma_2}{2(1+\sigma_2)}}
\\
\lesssim&\eps^{\frac{\sigma_2}{1+\sigma_2}}\left\{\left\|\pa_x\left[\widetilde{u}_1,\widetilde{\ta}\right]\right\|^2+(1+t)^{-1-\frac{1-\sigma_2}{3+5\sigma_2}}
\left\|\left[\widetilde{u}_1,\widetilde{\ta}\right]\right\|^{\frac{2(1+3\sigma_2)}{3+5\sigma}}\right\}\\
\lesssim&\eps^{\frac{\sigma_2}{1+\sigma_2}}\left\{\left\|\pa_x\left[\widetilde{u}_1,\widetilde{\ta}\right]\right\|^2+(1+t)^{-\frac{5+3\sigma_2}{4+4\sigma_2}}
+(1+t)^{-\frac{3+5\sigma_2}{2(1+3\sigma_2)}}\left\|\left[\widetilde{u}_1,\widetilde{\ta}\right]\right\|^{2}\right\},
\end{split}
\end{equation}
\begin{equation}\label{I52}
\begin{split}
\int_{\R}\left|\left[\widetilde{u},\widetilde{\ta}\right]\pa_x\widetilde{\ta}\pa_x[v^r,\ta^r]\right|dx\leq
\eta\left\|\pa_x\widetilde{\ta}\right\|^{2}
+C_\eta(1+t)^{-2}\left\|\left[\widetilde{u},\widetilde{\ta}\right]\right\|^{2},
\end{split}
\end{equation}
\begin{equation}\label{I53}
\begin{split}
\int_{\R}\left|\widetilde{\ta}\left(\pa_x\widetilde{\ta}\right)^2\right|dx\lesssim
\eps_0\left\|\pa_x\widetilde{\ta}\right\|^{2},
\end{split}
\end{equation}
and from \eqref{I3}, it follows that
\begin{equation}\label{I54}
\begin{split}
\int_{\R}\left|\left[\widetilde{u}_1,\widetilde{\ta}\right]\pa_x\ta^r\pa_x\left[u_1^r,\ta^r\right]\right|dx
\leq& C\left\|\left[\widetilde{u}_1,\widetilde{\theta}\right]\right\|_{L^\infty}\|\pa_x[u_1^r,\ta^r]\|^2\\
\leq&\eta\left\|\pa_x\left[\widetilde{u}_1,\widetilde{\ta}\right]\right\|^{2}
+C_\eta(1+t)^{-\frac{\sigma_1+1}{2}}\left\|\left[\widetilde{u}_1,\widetilde{\ta}\right]\right\|^{2}
+C_\eta\eps^{\frac{7-\sigma_1}{16}}(1+t)^{-\frac{21-3\sigma_1}{16}}.
\end{split}
\end{equation}
Therefore \eqref{I51}, \eqref{I52}, \eqref{I53} and \eqref{I54} give
\begin{equation*}
\begin{split}
|I_5|\lesssim& \eps^{\frac{\sigma_2}{1+\sigma_2}}\left\{\left\|\pa_x\left[\widetilde{u},\widetilde{\ta}\right]\right\|^2+(1+t)^{-\frac{5+3\sigma_2}{4+4\sigma_2}}
+(1+t)^{-\frac{3+5\sigma_2}{2(1+3\sigma_2)}}\left\|\left[\widetilde{u},\widetilde{\ta}\right]\right\|^{2}\right\}
+\eta\left\|\pa_x\left[\widetilde{v},\widetilde{u},\widetilde{\ta}\right]\right\|^{2}
\\&+C_\eta(1+t)^{-2}\left\|\left[\widetilde{u},\widetilde{\ta}\right]\right\|^{2}
+C_\eta(1+t)^{-\frac{\sigma_1+1}{2}}\left\|\left[\widetilde{u},\widetilde{\ta}\right]\right\|^{2}
+C_\eta\eps^{\frac{7-\sigma_1}{16}}(1+t)^{-\frac{21-3\sigma_1}{16}}.
\end{split}
\end{equation*}

Let us now define
\begin{eqnarray}\label{dec.co}
\left\{\begin{array}{rll}
\begin{split}
\zeta_0=&\min\left\{2, \frac{21-3\sigma_1}{16}, \frac{5+3\sigma_2}{4+4\sigma_2}\right\},\\
\zeta_1=&\min\left\{\frac{4}{3}, \frac{\sigma_1+1}{2}, \frac{3+5\sigma_2}{2(1+3\sigma_2)}\right\},\\
\sigma_0=&\min\left\{1, \frac{7-\sigma_1}{16}, \frac{\sigma_2}{1+\sigma_2}\right\},
\end{split}
\end{array}
\right.
\end{eqnarray}
where $1<\sigma_1<5/3$ and $0<\sigma_2<1$. Note that $0<\sigma_0<1/3$, $\zeta_0>1$, and  $\zeta_1>1$.
Due to \eqref{aps.phy} as well as the assumption $(\CA)$, we also observe that there exists $C_2>0$ such that
\begin{equation}\label{eqv.eg0}
\begin{split}
\frac{1}{C_2}&\left\{
\frac{1}{2}\left(\pa_x\widetilde{\phi},\pa_x \widetilde{\phi}\right)+\frac{1}{2}\left(\widetilde{\phi}^2,\rho_{e}'(\phi^r)\right)+\frac{1}{3}\left(\widetilde{\phi}^3,\rho_{e}''(\phi^r)\right)
\right\}
\\ \leq&\left\|\widetilde{\phi}\right\|^2_{H^1}\leq C_2\left\{\frac{1}{2}\left(\pa_x\widetilde{\phi},\pa_x \widetilde{\phi}\right)+\frac{1}{2}\left(\widetilde{\phi}^2,\rho_{e}'(\phi^r)\right)+\frac{1}{3}\left(\widetilde{\phi}^3,\rho_{e}''(\phi^r)\right)\right\}.
\end{split}
\end{equation}
Defining
\begin{eqnarray*}
\CE_1(\widetilde{\rho},\widetilde{u},\widetilde{\theta},\widetilde{\phi})=\rho \widetilde{\eta}(\widetilde{\rho},\widetilde{u},\widetilde{\theta})
+\frac{1}{2}\left(\pa_x\widetilde{\phi},\pa_x \widetilde{\phi}\right)+\frac{1}{2}\left(\widetilde{\phi}^2,\rho_{e}'(\phi^r)\right)+\frac{1}{3}\left(\widetilde{\phi}^3,\rho_{e}''(\phi^r)\right),
\end{eqnarray*}
with \eqref{dec.co} and \eqref{eqv.eg0} in hand, we now can conclude from \eqref{en.iqv}, \eqref{b.entropy}, \eqref{I1sum}, \eqref{I2sum} and the above estimates
on $I_3$, $I_4$ and $I_5$ that
\begin{equation}\label{zeng.p3}
\begin{split}
\frac{d}{dt}&\CE_1(\widetilde{\rho},\widetilde{u},\widetilde{\theta},\widetilde{\phi})
+\la\left\{\int_{\R}\left|\left[\widetilde{v},\widetilde{u}_1,\widetilde{S}\right]\right|^2\pa_xu^r_{1}dx
+\left\|\pa_x\left[\widetilde{u},\widetilde{\theta}\right]\right\|^2\right\}\\
\lesssim&(\eps_0+\eta)\left\|\left[\pa_x\widetilde{v},
\pa_x\widetilde{\phi},\pa_x^2\widetilde{\phi}\right]\right\|^2
+C_\eta(1+t)^{-\zeta_1}\left\|\left[\widetilde{v},\widetilde{u},\widetilde{\ta},\widetilde{\phi}\right]\right\|^2+C_\eta\eps^{\sigma_0}(1+t)^{-\zeta_0}\\
&+\sum\limits_{|\al|=1}C_\eta\int_{\R\times\R^3}\frac{(1+|\xi|)|\pa^{\al}\FG|^2}{\FM}dxd\xi
+C_\eta\eps_0\int_{\R\times\R^3}\frac{|\pa_{\xi_1}\widetilde{\FG}|^2}{\FM}dxd\xi
+C_\eta\eps_0\int_{{\R}\times{\R}^3}\frac{(1+|\xi|)\left|\widetilde{\FG}\right|^2}{\FM_{i}}dxd\xi.
\end{split}
\end{equation}

\noindent{\bf Step 2.} {\it Dissipation of $\pa_x\left[\widetilde{\rho},\widetilde{\phi},\pa_x\widetilde{\phi}\right]$
and $\pa_t\left[\widetilde{\rho},\widetilde{u},\widetilde{\ta},\widetilde{\phi},\pa_x\widetilde{\phi}\right]$.}

We first differentiate \eqref{tphy} and \eqref{trho0} with respect to $x$, respectively, to obtain
\begin{equation}\label{d.phi}
-\pa^3_x\widetilde{\phi}=\pa_x\widetilde{\rho}+\pa_x(\rho_{e}(\phi^r)-\rho_{e}(\phi))
+\pa^3_x\phi^r,
\end{equation}
and
\begin{equation}\label{d.rho}
\pa_t\pa_x\widetilde{\rho}+\pa_x u_1\pa_x\widetilde{\rho}+u_1\pa^2_x\widetilde{\rho}+\pa_x\rho\pa_x\widetilde{u}_1
+\rho\pa^2_x\widetilde{u}_1+\widetilde{\rho}\pa^2_xu_1^r
+\pa_x\widetilde{\rho}\pa_xu_1^r+\widetilde{u}_1\pa^2_x\rho^r
+\pa_x\rho^r\pa_x\widetilde{u}_1=0.
\end{equation}
Then taking the inner products of \eqref{d.phi}, \eqref{d.rho} and \eqref{tu10} with $\pa_x\widetilde{\phi}$, $3\mu(\theta)\frac{\pa_x\widetilde{\rho}}{\rho^2}$
and $\pa_x\widetilde{\rho}$ with  respect to $x$ over $\R$, respectively, one has
\begin{equation}\label{d.phi.ip}
\begin{split}
&\left(\pa^2_x\widetilde{\phi},\pa^2_x\widetilde{\phi}\right)
-\left(\pa_x(\rho_{e}(\phi^r)-\rho_{e}(\phi)),\pa_x\widetilde{\phi}\right)
=
\left(\pa_x\widetilde{\rho},\pa_x\widetilde{\phi}\right)
+\left(\pa^2_x\phi^r,\pa_x\widetilde{\phi}\right),
\end{split}
\end{equation}
\begin{equation}\label{d.rho.ip}
\begin{split}
&\left(\pa_t\pa_x\widetilde{\rho},\frac{3\mu(\theta)}{\rho^2}\pa_x\widetilde{\rho}\right)
+\left(\frac{3\mu(\theta)}{\rho}\pa^2_x\widetilde{u}_1,\pa_x\widetilde{\rho}\right)
\\&\quad+\left(\pa_x u_1\pa_x\widetilde{\rho}+u_1\pa^2_x\widetilde{\rho}+\pa_x\rho\pa_x\widetilde{u}_1
+\widetilde{\rho}\pa^2_xu_1^r
+\pa_x\widetilde{\rho}\pa_xu_1^r+\widetilde{u}_1\pa^2_x\rho^r
+\pa_x\rho^r\pa_x\widetilde{u}_1,\pa_x\widetilde{\rho}\right)=0,
\end{split}
\end{equation}
and
\begin{equation}\label{tu1.ip}
\begin{split}
&\left(\pa_t \widetilde{u}_1,\pa_x\widetilde{\rho}\right)+(u_1\pa_x u_1-u_1^r\pa_x u_1^r, \pa_x\widetilde{\rho})+\left(\frac{\pa_x P-\pa_xP^r}{\rho},\pa_x\widetilde{\rho}\right)+\left(\pa_xP^r\left(\frac{1}{\rho}-\frac{1}{\rho^r}\right),\pa_x\widetilde{\rho}\right)
+\left(\pa_x\widetilde{\phi},\pa_x\widetilde{\rho}\right)
\\
&\quad=\left(\frac{3\mu(\theta)}{\rho}\pa^2_x \widetilde{u}_1,\pa_x\widetilde{\rho}\right)
+\left(3\mu(\theta)\pa^2_x u^r_1,\pa_x\widetilde{\rho}\right)
+\left(\frac{3}{\rho}\pa_x(\mu(\theta))\pa_x u_1,\pa_x\widetilde{\rho}\right)
-\left(\frac{1}{\rho}\int_{{\R}^3}\xi_1^2\pa_x\Theta d\xi,\pa_x\widetilde{\rho}\right).
\end{split}
\end{equation}
To obtain the dissipation of $\pa_x\widetilde{\phi}$, we now expand the second term on the left hand side of \eqref{d.phi}
as
\begin{equation}\label{taylor2}
\rho_{e}(\phi^r)-\rho_{e}(\phi)=-\rho_{e}'(\phi^r)\widetilde{\phi}
\underbrace{-\int_{\phi^r}^{\phi}(\varrho-\phi)\rho_{e}''(\varrho) d\varrho}_{J_{10}}.
\end{equation}
Similar to \eqref{tJ1}, one has
\begin{multline}\label{xJ10}
J_{10}\sim \widetilde{\phi}^2,\ \ \pa_x J_{10}=\pa_x\phi\int_{\phi^r}^{\phi}\rho_{e}''(\varrho) d\varrho
+\widetilde{\phi}\pa_x\phi^r\int_{\phi^r}^{\phi}\rho_{e}''(\varrho) d\varrho
\\ \sim \pa_x\phi \widetilde{\phi}+\pa_x\phi^r\widetilde{\phi}^2= \pa_x\widetilde{\phi} \widetilde{\phi}
+\pa_x\phi^r\widetilde{\phi}+\pa_x\phi^r\widetilde{\phi}^2.
\end{multline}
Applying \eqref{taylor2} and noticing the cancellations in \eqref{d.phi.ip}, \eqref{d.rho.ip} and \eqref{tu1.ip},
we further have
\begin{multline}\label{sum.d1}
\frac{d}{dt}\left(\widetilde{u}_1,\pa_x\widetilde{\rho}\right)
+\frac{3}{2}\frac{d}{dt}\left(\pa_x\widetilde{\rho},\frac{\mu(\theta)}{\rho^2}\pa_x\widetilde{\rho}\right)
+\frac{2}{3}\left(\frac{\ta^r}{\rho},(\pa_x\widetilde{\rho})^2\right)+\left(\pa^2_x\widetilde{\phi},\pa^2_x\widetilde{\phi}\right)
+\left(\rho_{e}'(\phi^r)\pa_x\widetilde{\phi},\pa_x\widetilde{\phi}\right)\\
=\underbrace{-\left(\pa_x(\rho_{e}'(\phi^r))\widetilde{\phi},\pa_x\widetilde{\phi}\right)+\left(\pa_x J_{10},\pa_x\widetilde{\phi}\right)
+\left(\pa^2_x\phi^r,\pa_x\widetilde{\phi}\right)}_{J_{11}}
\\
\underbrace{-\left(\pa_x u_1\pa_x\widetilde{\rho}+u_1\pa^2_x\widetilde{\rho}+\pa_x\rho\pa_x\widetilde{u}_1
+\widetilde{\rho}\pa^2_xu_1^r
+\pa_x\widetilde{\rho}\pa_xu_1^r+\widetilde{u}_1\pa^2_x\rho^r
+\pa_x\rho^r\pa_x\widetilde{u}_1,\pa_x\widetilde{\rho}\right)}_{J_{12}}
\\+\underbrace{\frac{3}{2}\left(\pa_t\left(\frac{\mu(\theta)}{\rho^2}\right)\pa_x\widetilde{\rho},\pa_x\widetilde{\rho}\right)}_{J_{13}}
+\underbrace{\left(\widetilde{u}_1,\pa_t \pa_x\widetilde{\rho}\right)-
(u_1\pa_x u_1-u_1^r\pa_x u_1^r, \pa_x\widetilde{\rho})}_{J_{14}}
\\
\underbrace{-\frac{2}{3}\left(\frac{\pa_x (\widetilde{\theta}\widetilde{\rho})}{\rho},\pa_x\widetilde{\rho}\right)
-\frac{2}{3}\left(\frac{\pa_x (\widetilde{\theta}\rho^r)}{\rho},\pa_x\widetilde{\rho}\right)
-\frac{2}{3}\left(\frac{\pa_x \theta^r\widetilde{\rho}}{\rho},\pa_x\widetilde{\rho}\right)-\left(\pa_xP^r\left(\frac{1}{\rho}-\frac{1}{\rho^r}\right),\pa_x\widetilde{\rho}\right)}_{J_{15}}
\\
\underbrace{+\left(3\mu(\theta)\pa^2_x u^r_1,\pa_x\widetilde{\rho}\right)
+\left(\frac{3}{\rho}\pa_x(\mu(\theta))\pa_x u_1,\pa_x\widetilde{\rho}\right)}_{J_{16}}
\underbrace{-\left(\frac{1}{\rho}\int_{{\R}^3}\xi_1^2\pa_x\Theta d\xi,\pa_x\widetilde{\rho}\right)}_{J_{17}}.
\end{multline}
We now turn to estimate
$J_l$ $(11\leq l\leq 17)$ term by term. We first present the calculations for $J_{11}$, $J_{12}$, $J_{14}$ and $J_{17}$, since the other terms are similar and easier.
For $J_{11}$, in light of Lemma \ref{cl.Re.Re2.}, \eqref{aps3}, \eqref{xJ10} and by Cauchy-Schwarz's inequality with $0<\eta<1$, one has
\begin{equation*}
|J_{11}|
\lesssim (\eta+\eps_0)\left\|\pa_x\widetilde{\phi}\right\|^2+C_\eta(1+t)^{-2}\left\|\widetilde{\phi}\right\|^2+C_\eta\eps(1+t)^{-2}.
\end{equation*}
The first two terms in $J_{12}$ is equal to
$$
-\frac{1}{2}(\pa_x u_1\pa_x\widetilde{\rho},\pa_x\widetilde{\rho})
$$
by integration by parts. We thus obtain
\begin{equation*}
|J_{12}|
\lesssim (\eta+\eps_0+\eps)\|\pa_x\widetilde{\rho}\|^2+(\eps_0+\eps)\|\pa_x\widetilde{\rho}\|^2
+C_\eta(1+t)^{-2}\|[\widetilde{\rho},\widetilde{u}_1]\|^2+C_\eta\eps(1+t)^{-2},
\end{equation*}
by further performing the similar calculations as $J_{11}$.

For $J_{14}$,  by integration by parts and applying \eqref{trho0}, one has
\begin{eqnarray*}
J_{14} &=&(\pa_x \widetilde{u}_1,-\pa_t \widetilde{\rho})-\big((\widetilde{u}+u^r)\pa_x (\widetilde{u}_1+u_1^r) -u^r\pa_x u^r,\pa_x \widetilde{\rho}\big)\\
&=&\big(\pa_x \widetilde{u}_1, \pa_x (\rho u)-\pa_x (\rho^r u_1^r)\big)- (\widetilde{u} \pa_x \widetilde{u}_1 +\widetilde{u}_1 \pa_x u_1^r
+u_1^r \pa_x\widetilde{u}_1,\pa_x \widetilde{\rho})\\
&=&((\pa_x \widetilde{u}_1)^2, \widetilde{\rho})+(\pa_x \widetilde{u}_1,\widetilde{\rho}\pa_x u^r) +(\pa_x \widetilde{u}_1,\pa_x \rho^r \widetilde{u}_1) +(\pa_x \widetilde{u}_1, \rho^r \pa_x \widetilde{u}_1)-(\widetilde{u}_1\pa_x u_1^r,\pa_x \widetilde{\rho}),
\end{eqnarray*}
and therefore
\begin{equation*}
|J_{14}|
\leq (C \eps_0+\eta) \|\pa_x \widetilde{u}_1\|^2+\rho_+\|\pa_x \widetilde{u}_1\|^2 +\eta\|\pa_x\widetilde{\rho}\|^2+C_\eta(1+t)^{-2}\|[\widetilde{\rho},\widetilde{u}_1]\|^2.
\end{equation*}
As to $J_{17}$,
from Cauchy-Schwarz's inequality with $\eta$, it follows that
\begin{equation*}
\begin{split}
|J_{17}|\leq&\eta\|\pa_x\widetilde{\rho}\|^2+C_{\eta}\int_{\R\times{\R}^3}\frac{1+|\xi|}{\FM}|\pa_x\Theta|^2 dxd\xi.
\end{split}
\end{equation*}
To compute the above integral, by applying \eqref{d.iL}, Corollary \ref{inv.L.}, Lemma \ref{est.nonop} and \eqref{db.G}, one further obtains
\begin{equation}\label{TTa.es.}
\begin{split}
&\int_{\R\times{\R}^3}\frac{1+|\xi|}{\FM}|\pa_x\Theta|^2 dxd\xi \\
\lesssim& \sum\limits_{|\al|=2}\int_{\R\times{\R}^3}\frac{1+|\xi|}{\FM}|\pa^{\al}\FG|^2 dxd\xi
+\sum\limits_{|\al|=1}\int_{\R\times{\R}^3}\frac{1+|\xi|}{\FM}|\pa^{\al}\FG|^2 |\pa_x[\rho,u,\ta]|^2 dxd\xi \\
&+\eps_{0}\sum\limits_{|\al|\leq1}\int_{\R\times\R^3}\frac{|\pa_{\xi_1}\pa^{\al}\widetilde{\FG}|^2}{\FM}dxd\xi
+\sum\limits_{|\al|\leq1}\int_{\R\times\R^3}\left|\pa^{\al}\left(\pa_x\phi\pa_{\xi_1}\overline{\FG}\right)\right|^2dxd\xi
\\&+\int_{\R}\left(\int_{\R^3}\frac{(1+|\xi|)|\pa_x\FG|^2}{\FM_i}d\xi\right) \left(\int_{\R^3}\frac{|\FG|^2}{\FM_i}d\xi\right)dx
+\int_{\R}\left(\int_{\R^3}\frac{(1+|\xi|)|\FG|^2}{\FM_i}d\xi\right) \left(\int_{\R^3}\frac{|\pa_x\FG|^2}{\FM_i}d\xi\right)dx\\
&+\int_{\R}\left(\int_{\R^3}\frac{(1+|\xi|)|\FG|^2}{\FM_i}d\xi\right) \left(\int_{\R^3}\frac{|\FG|^2}{\FM_i}d\xi\right)|\pa_x[\rho,u,\ta]|^2dx
\\
\lesssim& \sum\limits_{|\al|=2}\int_{\R\times{\R}^3}\frac{1+|\xi|}{\FM}|\pa^{\al}\FG|^2 dxd\xi
+\eps_{0}\sum\limits_{|\al|=1}\int_{\R\times{\R}^3}\frac{1+|\xi|}{\FM}|\pa^{\al}\FG|^2 dxd\xi
\\&+\eps_{0}\sum\limits_{|\al|\leq1}\int_{\R\times\R^3}\frac{|\pa_{\xi_1}\pa^{\al}\widetilde{\FG}|^2}{\FM}dxd\xi
+\eps_0\sum\limits_{|\al|\leq1}\left\|\pa_x\pa^\al\widetilde{\phi}\right\|^2
+\eps_{0}\int_{{\R}\times{\R}^3}\frac{(1+|\xi|)\left|\widetilde{\FG}\right|^2}{\FM_{i}}dxd\xi
+C\eps(1+t)^{-2}.
\end{split}
\end{equation}
Thus,
\begin{equation*}
\begin{split}
|J_{17}|\leq&\eta\|\pa_x\widetilde{\rho}\|^2+C_\eta\sum\limits_{|\al|=2}\int_{\R\times{\R}^3}\frac{1+|\xi|}{\FM}|\pa^{\al}\FG|^2 dxd\xi
+C_\eta\eps_{0}\sum\limits_{|\al|=1}\int_{\R\times{\R}^3}\frac{1+|\xi|}{\FM}|\pa^{\al}\FG|^2 dxd\xi
\\&+C_\eta\eps_{0}\sum\limits_{|\al|\leq1}\int_{\R\times\R^3}\frac{|\pa_{\xi_1}\pa^{\al}\widetilde{\FG}|^2}{\FM}dxd\xi
+C_\eta\eps_{0}\int_{{\R}\times{\R}^3}\frac{(1+|\xi|)\left|\widetilde{\FG}\right|^2}{\FM_{i}}dxd\xi
\\&+C_\eta\eps_0\sum\limits_{|\al|\leq1}\left\|\pa_x\pa^\al\widetilde{\phi}\right\|^2+C_\eta\eps(1+t)^{-2}.
\end{split}
\end{equation*}
The estimations for the remaining terms will be much easier. For brevity, we directly give the following
computations:
\begin{equation*}
|J_{13}|\leq C\|\pa_t[\rho,\ta]\|_{L^\infty}\|\pa_x\widetilde{\rho}\|^2
\leq C\eps_0\|\pa_x\widetilde{\rho}\|^2,
\end{equation*}
\begin{equation*}
|J_{15}|
\leq (C \eps_0+\eta) \|\pa_x \widetilde{\rho}\|^2+C_\eta\left\|\pa_x \widetilde{\ta}\right\|^2 +\eta\left\|\pa_x\widetilde{\rho}\right\|^2+C(1+t)^{-2}\left\|\left[\widetilde{\rho},\widetilde{\ta}\right]\right\|^2,
\end{equation*}
\begin{equation*}
|J_{16}|
\leq \eta\|\pa_x \widetilde{\rho}\|^2+\eps_0\left\|\pa_x\left[\widetilde{u}_1,\widetilde{\ta}\right]\right\|^2+\eps C_\eta(1+t)^{-2}.
\end{equation*}
We insert the above estimations for $J_l$ $(11\leq l\leq 17)$ into \eqref{sum.d1} and then choose $\eps$, $\eps_0$ and $\eta$ suitably small such
that
\begin{equation}\label{sum.d12}
\begin{split}
\frac{d}{dt}&\left(\widetilde{u}_1,\pa_x\widetilde{\rho}\right)
+\frac{3}{2}\frac{d}{dt}\left(\pa_x\widetilde{\rho},\mu(\theta)\pa_x\widetilde{\rho}\right)
+\la\left(\pa_x\widetilde{\rho},\pa_x\widetilde{\rho}\right)
+\la\left(\pa^2_x\widetilde{\phi},\pa^2_x\widetilde{\phi}\right)
+\la\left(\pa_x\widetilde{\phi},\pa_x\widetilde{\phi}\right)\\
\lesssim&\left\|\pa_x\widetilde{\ta}\right\|^2
+\max\{1,\rho_+\}\left\|\pa_x\widetilde{u}_1\right\|^2
+(1+t)^{-2}\left\|\left[\widetilde{\rho},\widetilde{u}_1,\widetilde{\ta},\widetilde{\phi}\right]\right\|^2+\eps(1+t)^{-2}\\
&+\sum\limits_{|\al|=2}\int_{\R\times{\R}^3}\frac{1+|\xi|}{\FM}|\pa^{\al}\FG|^2 dxd\xi
+\eps_{0}\sum\limits_{|\al|=1}\int_{\R\times{\R}^3}\frac{1+|\xi|}{\FM}|\pa^{\al}\FG|^2 dxd\xi
\\&+\eps_{0}\sum\limits_{|\al|\leq1}\int_{\R\times\R^3}\frac{|\pa_{\xi_1}\pa^{\al}\widetilde{\FG}|^2}{\FM}dxd\xi
+\eps_{0}\int_{{\R}\times{\R}^3}\frac{(1+|\xi|)\left|\widetilde{\FG}\right|^2}{\FM_{i}}dxd\xi .
\end{split}
\end{equation}
Having obtained \eqref{sum.d12}, one can see that $\pa_t\left[\widetilde{\rho},\widetilde{u},\widetilde{\ta}\right]$
also enjoys the dissipation property. To see this, we get from {\eqref{cons.law.} and \eqref{MEt}} that
\begin{eqnarray}\label{pb.con.}
\left\{\begin{array}{rlll}
\begin{split}
&\pa_t\widetilde{\rho}+\pa_x(\rho u_1-\rho^ru_1^r)=0,\\
&\pa_t \widetilde{u}_1+u_1\pa_x u_1-u_1^r\pa_x u_1^r+\frac{\pa_x P}{\rho}-\frac{\pa_xP^r}{\rho^r}+\pa_x\widetilde{\phi}=
-\frac{1}{\rho}\int_{{\R}^3}\xi_1^2\pa_x\FG d\xi,\\
&\pa_t \widetilde{u}_i+{u_1\pa_x\widetilde{u}_i}=-\frac{1}{\rho}\int_{{\R}^3}\xi_1\xi_i\pa_x\FG d\xi,\
\ i=2,\ 3,\\
&\pa_t \widetilde{\theta}+u_1\pa_x\ta-u^r_1\pa_x\ta^r+\frac{P\pa_x u_1}{\rho}-\frac{P^r\pa_x u^r_1}{\rho^r}=-\frac{1}{\rho}\int_{{\R}^3}\left(\frac{|\xi|^2}{2}-u\cdot\xi\right)\xi_1\pa_x\FG d\xi,
\end{split}
\end{array}\right.
\end{eqnarray}
which yields
\begin{equation}\label{sum.d13}
\begin{split}
\left\|\pa_t\left[\widetilde{\rho},\widetilde{u},\widetilde{\ta}\right]\right\|^2
\lesssim&\left\|\pa_x\left[\widetilde{\rho},\widetilde{u},\widetilde{\ta}, \widetilde{\phi}\right]\right\|^2
+(1+t)^{-2}\left\|\left[\widetilde{\rho},\widetilde{u}_1,\widetilde{\ta}\right]\right\|^2+\int_{\R\times{\R}^3}\frac{1+|\xi|}{\FM}|\pa_x\FG|^2 dxd\xi .
\end{split}
\end{equation}
Letting $1\gg\ka_1\gg\ka_2>0$, taking the summation
of \eqref{zeng.p3}, $\eqref{sum.d12}\times\ka_1$, $\eqref{sum.d13}\times\ka_2$ and $\eqref{ptphyL2}\times \ka_2$, we have for sufficiently small $\eps_0>0$, $\eps>0$ and $\eta>0$ that
\begin{equation}\label{ma.dis1}
\begin{split}
\frac{d}{dt}&\CE_1(\widetilde{\rho},\widetilde{u},\widetilde{\theta},\widetilde{\phi})
+\ka_1\frac{d}{dt}\left(\widetilde{u}_1,\pa_x\widetilde{\rho}\right)
+\frac{3\ka_1}{2}\frac{d}{dt}\left(\pa_x\widetilde{\rho},\pa_x\widetilde{\rho}\right)
\\[2mm]
&+\la\int_{\R}\left|\left[\widetilde{\rho},\widetilde{u}_1,\widetilde{S}\right]\right|^2\pa_xu^r_{1}dx
+\la\left\|\pa_x\left[\widetilde{u},\widetilde{\theta}\right]\right\|^2
+\la\left(\left\|\pa_t\widetilde{\phi}\right\|^2+\left\|\pa_t\pa_x\widetilde{\phi}\right\|^2\right)
\\
&+\la\left\|\pa_t\left[\widetilde{\rho},\widetilde{u},\widetilde{\ta}\right]\right\|^2+\la\left(\pa_x\widetilde{\rho},\pa_x\widetilde{\rho}\right)
+\la\left(\pa^2_x\widetilde{\phi},\pa^2_x\widetilde{\phi}\right)
+\la\left(\pa_x\widetilde{\phi},\pa_x\widetilde{\phi}\right)
\\
\lesssim&(1+t)^{-\zeta_1}\left\|\left[\widetilde{\rho},\widetilde{u},\widetilde{\ta},\widetilde{\phi}\right]\right\|^2
+\eps^{\sigma_0}(1+t)^{-\zeta_0}+\eps_{0}\int_{{\R}\times{\R}^3}\frac{(1+|\xi|)\left|\widetilde{\FG}\right|^2}{\FM_{i}}dxd\xi
\\
&+\sum\limits_{1\leq|\al|\leq2}\int_{\R\times\R^3}\frac{(1+|\xi|)|\pa^{\al}\FG|^2}{\FM}dxd\xi
+\eps_{0}\sum\limits_{|\al|\leq1}\int_{\R\times\R^3}\frac{|\pa_{\xi_1}\pa^{\al}\widetilde{\FG}|^2}{\FM}dxd\xi,
\end{split}
\end{equation}
where we also used the fact that $\widetilde{v}\sim\widetilde{\rho}$.

\noindent{\bf Step 3.} {\it The first order energy estimates.}

For $|\al|=1$, taking the inner product of $\pa^{\al}\eqref{trho0}$, $\pa^{\al}\eqref{tu10}$, $\pa^{\al}\eqref{tui0}$ and $\pa^{\al}\eqref{tta0}$ with $\pa^{\al}\widetilde{\rho}$, $\pa^{\al}\widetilde{u}_1$,
$\pa^{\al}\widetilde{u}_i$ and $\pa^{\al}\widetilde{\ta}$, respectively, and then taking the summation of the resulting equations,
one has
\begin{equation}\label{p2u1}
\begin{split}
\frac{1}{2}&\frac{d}{dt}\sum\limits_{|\al|=1}\left\|\pa^{\al}\left[\widetilde{\rho},\widetilde{u},\widetilde{\ta}\right]\right\|^2
+\sum\limits_{|\al|=1}\left(\frac{3\mu(\theta)}{\rho}\pa_x \pa^{\al}\widetilde{u}_1,\pa_x\pa^{\al}\widetilde{u}_1\right)
+\sum\limits_{i=2}^3\sum\limits_{|\al|=1}\left(\frac{\mu(\theta)}{\rho}\pa_x \pa^{\al}\widetilde{u}_i,\pa_x\pa^{\al}\widetilde{u}_i\right)
\\
&+\sum\limits_{|\al|=1}\left(\frac{\kappa(\theta)}{\rho}\pa_x \pa^{\al}\widetilde{\theta},\pa_x\pa^{\al}\widetilde{\ta}\right)
=\sum\limits_{l=1}^{9}\CI_l,
\end{split}
\end{equation}
where
\begin{eqnarray*}
\left\{\begin{array}{rll}
\begin{split}
&\CI_1=-\left(\pa^{\al}\pa_x (\rho u_1-\rho^r u^r_1),\pa^{\al}\widetilde{\rho}\right),\
\CI_2=-\left(\pa^{\al}(u_1\pa_x u_1-u_1^r\pa_x u_1^r), \pa^{\al}\widetilde{u}_1\right)-{\sum\limits_{i=2}^3\left(\pa^{\al}(u_1\pa_x \widetilde{u}_i), \pa^{\al}\widetilde{u}_i\right)},\\
&\CI_3=-\left(\pa^{\al}\pa_x\widetilde{\phi}, \pa^{\al}\widetilde{u}_1\right),\
\CI_4=-\left(\pa^{\al}\left(\frac{\pa_x P-\pa_xP^r}{\rho}\right),\pa^\al\widetilde{u}_1\right)
-\left(\pa^\al\left(\pa_xP^r\left(\frac{1}{\rho}-\frac{1}{\rho^r}\right)\right),\pa^{\al}\widetilde{u}_1\right),\\
&\CI_{5}=-\left(\frac{3\mu(\theta)}{\rho}\pa^{\al}\pa_x u^r_1,\pa^{\al}\pa_x\widetilde{u}_1\right)
-\left(\frac{\ka(\theta)}{\rho}\pa^{\al}\pa_x \theta^r,\pa^{\al}\pa_x\widetilde{\theta}\right),\\
&\CI_{6}=-\left(\frac{3\pa^{\al}\mu(\theta)}{\rho}\pa_x u_1,\pa^{\al}\pa_x\widetilde{u}_1\right)
-\sum\limits_{i=2}^3\left(\frac{\pa^{\al}\mu(\theta)}{\rho}\pa_x \widetilde{u}_i,\pa^{\al}\pa_x\widetilde{u}_i\right)
-\left(\frac{\pa^{\al}\ka(\theta)}{\rho}\pa_x \theta,\pa^{\al}\pa_x\widetilde{\theta}\right),\\
&\CI_{7}=-3\left(\pa_x\left(\frac{1}{\rho}\right)\pa^{\al}(\mu(\theta)\pa_x  u_1),\pa^{\al}\widetilde{u}_1\right)-\sum\limits_{i=2}^3\left(\pa_x\left(\frac{1}{\rho}\right)\pa^{\al}(\mu(\theta)\pa_x \widetilde{u}_i),\pa^{\al}\widetilde{u}_i\right)
\\&\qquad-\left(\pa_x\left(\frac{1}{\rho}\right)\pa^{\al}(\ka(\theta)\pa_x\theta),\pa^{\al}\widetilde{\theta}\right)
+3\left(\pa^{\al}\left(\frac{1}{\rho}\right)\pa_x(\mu(\theta)\pa_x  u_1),\pa^{\al}\widetilde{u}_1\right)\\&\qquad+\sum\limits_{i=2}^3\left(\pa^{\al}\left(\frac{1}{\rho}\right)\pa_x(\mu(\theta)\pa_x \widetilde{u}_i),\pa^{\al}\widetilde{u}_i\right)
+\left(\pa^{\al}\left(\frac{1}{\rho}\right)\pa_{x}(\ka(\theta)\pa_x\theta),\pa^{\al}\widetilde{\theta}\right),
\end{split}
\end{array}\right.
\end{eqnarray*}
and
\begin{eqnarray*}
\left\{\begin{array}{rll}
\begin{split}
&\CI_{8}=3\left(\pa^{\al}\left(\frac{\mu(\theta)}{\rho}(\pa_x u_1)^2\right),\pa^{\al}\widetilde{\ta}\right)+
\sum\limits_{i=2}^3\left(\pa^{\al}\left(\frac{\mu(\theta)}{\rho}(\pa_x \widetilde{u}_i)^2\right),\pa^{\al}\widetilde{\ta}\right),\\
&\CI_{9}=\sum\limits_{i=1}^3\left(\int_{{\R}^3}\xi_1\xi_i\pa^{\al}\Theta d\xi,\pa_x\left(\frac{\pa^{\al}\widetilde{u}_i}{\rho}\right)\right)
+\left(\int_{{\R}^3}\left(\frac{|\xi|^2}{2}-u\cdot\xi\right)\xi_1\pa^{\al}\Theta d\xi,\pa_x\left(\frac{\pa^{\al}\widetilde{\ta}}{\rho}\right)\right)
\\&\qquad-\sum\limits_{i=1}^3\left(\int_{{\R}^3}\pa_x u_i\xi_i\xi_1\pa^{\al}\Theta d\xi,\frac{\pa^{\al}\widetilde{\ta}}{\rho}\right)
+\sum\limits_{i=1}^3\left(\int_{{\R}^3}\pa^{\al}\left(\frac{u_i}{\rho}\right)\xi_i\xi_1\pa_x\Theta d\xi,\pa^{\al}\widetilde{\ta}\right)
\\&\qquad-\sum\limits_{i=1}^3\left(\pa^{\al}\left(\frac{1}{\rho}\right)\int_{{\R}^3}\xi_1\xi_i\pa_x\Theta d\xi,\pa^{\al}\widetilde{u}_i\right).
\end{split}
\end{array}\right.
\end{eqnarray*}
We now turn to estimate $\CI_l$ $(1\leq l\leq 9)$ term by term.
By integration by parts and applying Cauchy-Schwarz's inequality with $0<\eta<1$, Sobolev's inequality \eqref{sob.ine.}, a priori assumption \eqref{aps}, the estimates \eqref{TTa.es.}, as well as Lemma \ref{cl.Re.Re2.}, one can see that
\begin{equation*}
\begin{split}
|\CI_1|\leq &C_\eta\sum\limits_{|\al|=1}\left\|\pa^\al[\widetilde{\rho},\widetilde{u}_1]\right\|^2+(\eta+\eps_0)\sum\limits_{|\al|=1}\left\|\pa_x\pa^\al\widetilde{u}_1\right\|^2
+C_\eta(1+t)^{-2}\|[\widetilde{\rho},\widetilde{u}_1]\|^2,
\end{split}
\end{equation*}
\begin{equation*}
\begin{split}
|\CI_2|\leq \eps_0\sum\limits_{|\al|=1}\|\pa^\al\widetilde{u}\|^2
+C(1+t)^{-2}\|\widetilde{u}_1\|^2,\ |\CI_3|\leq \eta\sum\limits_{|\al|=1}\|\pa_x\pa^\al\widetilde{u}_1\|^2
+C_\eta\sum\limits_{|\al|=1}\left\|\pa^\al\widetilde{\phi}\right\|^2,
\end{split}
\end{equation*}
\begin{equation*}
\begin{split}
|\CI_4|\leq &\eta\sum\limits_{|\al|=1}\|\pa_x\pa^\al\widetilde{u}_1\|^2+(\eps_0+\eta)\sum\limits_{|\al|=1}\|\pa^\al\widetilde{u}_1\|^2
+C_\eta\sum\limits_{|\al|=1}\left\|\pa^\al\left[\widetilde{\rho},\widetilde{\ta}\right]\right\|^2
+C_\eta(1+t)^{-2}\left\|\left[\widetilde{\rho},\widetilde{\ta}\right]\right\|^2,
\end{split}
\end{equation*}
\begin{equation*}
\begin{split}
|\CI_5|\leq \eta\sum\limits_{|\al|=1}\left\|\pa^\al\pa_x\left[\widetilde{u}_1,\widetilde{\ta}\right]\right\|^2
+C_\eta\eps(1+t)^{-2},
\end{split}
\end{equation*}
\begin{equation*}
\begin{split}
|\CI_6|\leq \eta\sum\limits_{|\al|=1}\left\|\pa_x\pa^\al\left[\widetilde{u},\widetilde{\ta}\right]\right\|^2
+\eps_0\sum\limits_{|\al|=1}\left\|\pa^\al\left[\widetilde{u},\widetilde{\ta}\right]\right\|^2
+C_\eta\eps(1+t)^{-2},
\end{split}
\end{equation*}
\begin{equation*}
\begin{split}
|\CI_7|+|\CI_8|\leq \eps_0\sum\limits_{|\al|=1}\left\|\pa_x\pa^\al\left[\widetilde{u},\widetilde{\ta}\right]\right\|^2
+\eps_0\sum\limits_{|\al|=1}\left\|\pa^\al\left[\widetilde{u},\widetilde{\ta}\right]\right\|^2
+C\eps(1+t)^{-2},
\end{split}
\end{equation*}
\begin{equation*}
\begin{split}
|\CI_{9}|
\lesssim&\eta\sum\limits_{|\al|=1}\left\|\pa_x\pa^{\al}\left[\widetilde{u},\widetilde{\ta}\right]\right\|^2
+\eps_0\sum\limits_{|\al|=1}\left\|\pa^\al\left[\widetilde{u},\widetilde{\ta}\right]\right\|^2
+C_\eta\eps_0\sum\limits_{|\al|\leq1}\left\|\pa_x\pa^\al\widetilde{\phi}\right\|^2+C_\eta\eps(1+t)^{-2}
\\&+ C_\eta\sum\limits_{|\al|=2}\int_{\R\times{\R}^3}\frac{1+|\xi|}{\FM}|\pa^{\al}\FG|^2 dxd\xi
+C_\eta\eps_{0}\sum\limits_{|\al|=1}\int_{\R\times{\R}^3}\frac{1+|\xi|}{\FM}|\pa^{\al}\FG|^2 dxd\xi
\\&+C_\eta\eps_{0}\sum\limits_{|\al|\leq1}\int_{\R\times\R^3}\frac{|\pa_{\xi_1}\pa^{\al}\widetilde{\FG}|^2}{\FM}dxd\xi
+C_\eta\eps_{0}\int_{{\R}\times{\R}^3}\frac{(1+|\xi|)\left|\widetilde{\FG}\right|^2}{\FM_{i}}dxd\xi .
\end{split}
\end{equation*}
Plugging the previous computations for $\CI_l$ $(1\leq l\leq9)$ into \eqref{p2u1}, we thus arrive at
\begin{equation}\label{p2u2}
\begin{split}
\frac{d}{dt}&\sum\limits_{|\al|=1}\left\|\pa^{\al}\left[\widetilde{\rho},\widetilde{u},\widetilde{\ta}\right]\right\|^2
+\la\sum\limits_{|\al|=1}\left\|\pa_x \pa^{\al}\left[\widetilde{u},\widetilde{\theta}\right]\right\|^2
\\ \lesssim&\sum\limits_{|\al|=1}\left\|\pa^{\al}\left[\widetilde{\rho},\widetilde{u},\widetilde{\ta},\widetilde{\phi}\right]\right\|^2
+\eps_0\sum\limits_{|\al|=1}\left\|\pa_x\pa^\al\widetilde{\phi}\right\|^2+(1+t)^{-2}\left\|\left[\widetilde{\rho},\widetilde{u},\widetilde{\ta}\right]\right\|^2+\eps(1+t)^{-2}
\\&+ \sum\limits_{|\al|=2}\int_{\R\times{\R}^3}\frac{1+|\xi|}{\FM}|\pa^{\al}\FG|^2 dxd\xi
+\eps_{0}\sum\limits_{|\al|=1}\int_{\R\times{\R}^3}\frac{1+|\xi|}{\FM}|\pa^{\al}\FG|^2 dxd\xi
\\&+\eps_{0}\sum\limits_{|\al|\leq1}\int_{\R\times\R^3}\frac{|\pa_{\xi_1}\pa^{\al}\widetilde{\FG}|^2}{\FM}dxd\xi
+\eps_{0}\int_{{\R}\times{\R}^3}\frac{(1+|\xi|)\left|\widetilde{\FG}\right|^2}{\FM_{i}}dxd\xi .
\end{split}
\end{equation}
Let us now deduce the second order dissipation of $\widetilde{\rho}$. For this, letting $|\al|=1$, taking the inner product of $\pa^{\al}\eqref{pb.con.}_2$ with $\pa^{\al}\pa_x\widetilde{\rho}$, we obtain
\begin{equation}\label{hveg}
\begin{split}
\frac{d}{dt}&\left(\pa^{\al}\widetilde{u}_1,\pa^{\al}\pa_x\widetilde{\rho}\right)
+\frac{2}{3}\left(\frac{\ta^r}{\rho}\pa^{\al} \pa_x\widetilde{\rho},\pa^{\al}\pa_x\widetilde{\rho}\right)
\\
=&\left(\pa^{\al}\widetilde{u}_1,\pa_t\pa^{\al}\pa_x\widetilde{\rho}\right)-\left(\pa^{\al}(u_1\pa_x u_1-u_1^r\pa_x u_1^r), \pa^{\al}\widetilde{\rho}\right)
-\frac{2}{3}\left(\pa^\al\left(\frac{\pa_x (\widetilde{\theta}\widetilde{\rho})}{\rho}\right),\pa^{\al}\pa_x\widetilde{\rho}\right)
\\&-\frac{2}{3}\left(\pa^\al\left(\frac{\pa_x (\widetilde{\theta}\rho^r)}{\rho}\right),\pa^{\al}\pa_x\widetilde{\rho}\right)
-\frac{2}{3}\left(\pa^\al\left(\frac{\pa_x \theta^r\widetilde{\rho}}{\rho}\right),\pa^{\al}\pa_x\widetilde{\rho}\right)
-\left(\pa^\al\left(\pa_xP^r\left(\frac{1}{\rho}-\frac{1}{\rho^r}\right)\right),\pa^{\al}\pa_x\widetilde{\rho}\right)
\\&-\left(\pa^{\al}\pa_x\widetilde{\phi}, \pa^{\al}\pa_x\widetilde{\rho}\right)
-\left(\int_{{\R}^3}\xi_1^2\pa^{\al}\pa_x\FG d\xi,\pa^{\al}\pa_x\widetilde{\rho}\right).
\end{split}
\end{equation}
By integration by parts and in view of the first equation of \eqref{pb.con.}, the first term on the right hand side of \eqref{hveg}
can be rewritten as
$$
-\left(\pa^{\al}\pa_x\widetilde{u}_1,\pa_t\pa^{\al}\widetilde{\rho}\right)=\left(\pa^{\al}\pa_x\widetilde{u}_1,\pa^{\al}(\pa_x(\rho u_1-\rho^r u^r_1))\right),
$$
which is further bounded by
$$
C_\eta\sum\limits_{|\al|=1}\|\pa^{\al}\pa_x\widetilde{u}_1\|^2+(\eta+\eps_0)\|\pa^{\al}\pa_x\widetilde{\rho}\|^2
+\eps_0\sum\limits_{|\al|=1}\|\pa^{\al}[\widetilde{\rho},\widetilde{u}_1]\|^2
+(1+t)^{-2}\|[\widetilde{\rho},\widetilde{u}_1]\|^2.
$$
The remaining terms on the right hand side of \eqref{hveg} are dominated by
\begin{equation*}
\begin{split}
(\eps_0+\eta)\sum\limits_{|\al|=1}\left\|\pa_x\pa^{\al}\widetilde{\rho}\right\|^2
&+C_\eta\sum\limits_{|\al|=1}\left\|\pa^{\al}\pa_x\left[\widetilde{u}_1,\widetilde{\ta},\widetilde{\phi}\right]\right\|^2
+\eps_{0}\sum\limits_{|\al|=1}\left\|\pa^{\al}\left[\widetilde{\rho},\widetilde{u}_1,\widetilde{\ta}\right]\right\|^2
\\&+C_\eta(1+t)^{-2}\left\|\left[\widetilde{\rho},\widetilde{\ta}\right]\right\|^2+\eps(1+t)^{-2}
+C_\eta\sum\limits_{|\al|=2}\int_{\R\times{\R}^3}\frac{1+|\xi|}{\FM}|\pa^{\al}\FG|^2 dxd\xi .
\end{split}
\end{equation*}
We next get from substituting the above estimates into \eqref{hveg} that
\begin{equation}\label{hveg2}
\begin{split}
\frac{d}{dt}\sum\limits_{|\al|=1}&\left(\pa^{\al}\widetilde{u}_1,\pa^{\al}\pa_x\widetilde{\rho}\right)
+\la\sum\limits_{|\al|=1}\left(\pa^{\al} \pa_x\widetilde{\rho},\pa^{\al}\pa_x\widetilde{\rho}\right)
\\
\lesssim&\sum\limits_{|\al|=1}\left\|\pa^{\al}\pa_x\left[\widetilde{u}_1,\widetilde{\ta},\widetilde{\phi}\right]\right\|^2
+\eps_{0}\sum\limits_{|\al|=1}\left\|\pa^{\al}\left[\widetilde{\rho},\widetilde{u}_1,\widetilde{\ta}\right]\right\|^2
\\&+(1+t)^{-2}\left\|\left[\widetilde{\rho},\widetilde{u}_1,\widetilde{\ta}\right]\right\|^2+\eps(1+t)^{-2}
+\sum\limits_{|\al|=2}\int_{\R\times{\R}^3}\frac{1+|\xi|}{\FM}|\pa^{\al}\FG|^2 dxd\xi ,
\end{split}
\end{equation}
provided $\eta>0$ and $\eps>0$ suitably small.

As to the second order $t-$derivative of $\left[\widetilde{\rho},\widetilde{u},\widetilde{\ta}\right]$, by \eqref{pb.con.}, one has
\begin{equation}\label{sed.t}
\begin{split}
\left\|\pa^2_t\left[\widetilde{\rho},\widetilde{u},\widetilde{\ta}\right]\right\|^2
\lesssim&\sum\limits_{|\al|=1}\left\|\pa^{\al}\pa_x\left[\widetilde{\rho},\widetilde{u},\widetilde{\ta},\widetilde{\phi}\right]\right\|^2
+\eps_0\sum\limits_{|\al|=1}\left\|\pa^{\al}\left[\widetilde{\rho},\widetilde{u},\widetilde{\ta},\widetilde{\phi}\right]\right\|^2
+(1+t)^{-2}\left\|\left[\widetilde{\rho},\widetilde{u}_1,\widetilde{\ta}\right]\right\|^2
\\&+\sum\limits_{|\al|=2}\int_{\R\times{\R}^3}\frac{1+|\xi|}{\FM}|\pa^{\al}\FG|^2 dxd\xi
+\eps_{0}\sum\limits_{|\al|=1}\int_{\R\times{\R}^3}\frac{1+|\xi|}{\FM}|\pa^{\al}\FG|^2 dxd\xi .
\end{split}
\end{equation}
In addition, it follows from \eqref{tphy}, \eqref{taylor2} and \eqref{xJ10} that
\begin{equation}\label{ptphyL22}
\begin{split}
\sum\limits_{|\al|=2}\left\{\left\|\pa^{\al}\widetilde{\phi}\right\|^2+\left\|\pa^{\al}\pa_x\widetilde{\phi}\right\|^2\right\}\leq&
C\sum\limits_{|\al|=2}\|\pa^{\al}\widetilde{\rho}\|^2+\eps_{0}\sum\limits_{|\al|=1}\left\|\pa^{\al}\widetilde{\phi}\right\|_{H^1}^2
+C(1+t)^{-2}\|\widetilde{\phi}\|^2
+C\eps(1+t)^{-2}.
\end{split}
\end{equation}
Finally, letting $\ka_2\gg \ka_3\gg \ka_4\gg \ka_5\gg\ka_6>0$, we get from
$\eqref{ma.dis1}+\eqref{p2u2}\times\ka_3
+\eqref{hveg2}\times\ka_4+\eqref{sed.t}\times\ka_5+\eqref{ptphyL22}\times\ka_6$ that
\begin{equation}\label{sum.eng1}
\begin{split}
\frac{d}{dt}&\CE_1(\widetilde{\rho},\widetilde{u},\widetilde{\theta},\widetilde{\phi})
+\ka_1\frac{d}{dt}\left(\widetilde{u}_1,\pa_x\widetilde{\rho}\right)
+\frac{3\ka_1}{2}\frac{d}{dt}\left(\pa_x\widetilde{\rho},\pa_x\widetilde{\rho}\right)
+\ka_3\sum\limits_{|\al|=1}\frac{d}{dt}\left\|\pa^{\al}\left[\widetilde{\rho},\widetilde{u},\widetilde{\ta}\right]\right\|^2\\
&+\ka_4\sum\limits_{|\al|=1}\frac{d}{dt}\left(\pa^{\al}\widetilde{u}_1,\pa^{\al}\pa_x\widetilde{\rho}\right)
+\la\int_{\R}\left|\left[\widetilde{\rho},\widetilde{u}_1,\widetilde{S}\right]\right|^2\pa_xu^r_{1}dx
\\&+\la\sum\limits_{1\leq|\al|\leq 2}\left\|\pa^{\al}\left[\widetilde{\rho},\widetilde{u},\widetilde{\theta}\right]\right\|^2
+\la\sum\limits_{1\leq|\al|\leq2}\left\|\pa^{\al}\widetilde{\phi}\right\|_{H^1}^2\\
\lesssim&(1+t)^{-\al_1}\left\|\left[\widetilde{\rho},\widetilde{u},\widetilde{\ta},\widetilde{\phi}\right]\right\|^2
+\eps^{\sigma_0}(1+t)^{-\al_0}
+\sum\limits_{1\leq|\al|\leq2}\int_{\R\times{\R}^3}\frac{1+|\xi|}{\FM}|\pa^{\al}\FG|^2 dxd\xi
\\&+\eps_{0}\sum\limits_{|\al|\leq1}\int_{\R\times\R^3}\frac{|\pa_{\xi_1}\pa^{\al}\widetilde{\FG}|^2}{\FM}dxd\xi
+\eps_{0}\int_{{\R}\times{\R}^3}\frac{(1+|\xi|)\left|\widetilde{\FG}\right|^2}{\FM_{i}}dxd\xi .
\end{split}
\end{equation}
Then \eqref{macro.eng} follows from \eqref{sum.eng1}. This concludes the proof of Proposition \ref{mac.eng.lem.}.

\begin{remark}
Note that the above estimates do not include the second order energy of $\left[\widetilde{\rho},\widetilde{u},\widetilde{\theta}\right]$,
and they will be left to the next subsection, where the dissipation of the microscopic part will be mainly addressed. This special treatment coincides with
the energy method developed in \cite{G06}.
\end{remark}

\end{proof}

\subsection{Energy estimates on the microscopic part}\label{sec3.2}
Now we turn to deduce the energy estimates on the microscopic part $\FG$.
The trick of
deducing the desired energy estimates can be outlined as follows. As mentioned in the previous subsection,
$$
\left\|\frac{\FG}{\sqrt{\FM_i}}\right\|^2_{L^2_{x,\xi}}
$$
is not integrable with respect to the time variable,
we first perform the zeroth order energy estimate on $\widetilde{\FG}$.
Then we directly present the higher order energy
estimates on $F$.
At last, we deduce the mixed derivative energy estimates on
$\pa^{\al}\pa^\be\widetilde{\FG}$ for $|\al|+|\be|\leq2$ and $|\be|\geq1$. It is shown that the above energy estimates only
 with respect to the global Maxwellian
$\FM_\ast$ or the local Maxwellian $\FM$ can not be closed. To overcome this difficulty,
one has to use the interplay of these two kinds of weighted energy estimates.

The main result of this subsection is given in the following

\begin{proposition}\label{g.eng.lem.}
Under the conditions listed in Proposition \ref{mac.eng.lem.}, it holds that
\begin{equation}\label{g.eng.}
\begin{split}
\sum\limits_{|\al|\leq1}&\left\|\pa^{\al}\left[\widetilde{\rho}, \widetilde{u}, \widetilde{\ta}\right](t)\right\|^2
+\sum\limits_{|\al|\leq2}
\left\|\pa^{\al} \widetilde{\phi}(t)\right\|_{H^1}^2
+\int_{{\R}\times{\R}^3}\frac{\left|\widetilde{\FG}\right|^2}{{\FM_{*}}}dxd\xi
+\sum\limits_{1\leq|\al|\leq2}\int_{{\R}\times{\R}^3}\frac{\left|\pa^{\al} F\right|^2}{\FM_*}dxd\xi
\\&+\sum\limits_{|\al|+|\be|\leq 2\atop{|\be|\geq1}}
\int_{{\R}\times{\R}^3}\frac{\left|\pa^{\al}\pa^\be \widetilde{\FG}\right|^2}{\FM_{*}}dxd\xi
+\sum\limits_{1\leq|\al|\leq2}\int_{0}^{T}\int_{\R\times{\R}^3}\frac{(1+|\xi|)\left|\pa^{\al} \FG\right|^2}{\FM_*}dxd\xi dt
\\&+\int_{0}^{T}\int_{{\R}\times{\R}^3}\frac{(1+|\xi|)\left|\widetilde{\FG}\right|^2}{\FM_{*}}dxd\xi dt
+\sum_{|\al|+|\be|\leq 2\atop{|\be|\geq1}}\int_{0}^{T}\int_{{\R}\times{\R}^3}
\frac{(1+|\xi|)\left|\pa^{\al}\pa^\be\widetilde{\FG}\right|^2}{\FM_{*}}dxd\xi dt
\\&+\int_{0}^{T}\int_{\R}\left|\left[\widetilde{\rho},\widetilde{u}_1,\widetilde{S},\widetilde{\phi}\right]\right|^2\pa_xu^r_{1}dxdt
+\sum\limits_{1\leq|\al|\leq 2}\int_{0}^{T}\left\|\pa^{\al}\left[\widetilde{\rho},\widetilde{u},\widetilde{\theta}\right]\right\|^2dt
+\sum\limits_{1\leq|\al|\leq2}\int_{0}^{T}\left\|\pa^{\al}\widetilde{\phi}\right\|_{H^1}^2dt
\\ \leq& C_0N^2(0)+C_0\eps^{\sigma_0},
\end{split}
\end{equation}
for $0\leq t \leq T$.
\end{proposition}
\begin{proof}
We divide the proof by the following three steps.

\noindent{\bf Step 1.} {\it Zeroth order energy estimates for $\frac{\widetilde{\FG}}{\sqrt{\FM_*}}$.}

Notice that $\widetilde{\FG}$ solves
\begin{equation}\label{g.eq1.}
\begin{split}
\pa_t\widetilde{\FG}-L_{\FM}\widetilde{\FG}=
-\frac{3}{2\ta}{\FP_1^{\FM}}\left[\xi_1\FM\left(\xi\cdot\pa_x\widetilde{u}+\frac{|\xi-u|^2}{2\ta}\pa_x\widetilde{\ta}\right)\right]
-\FP_1^{\FM}\left(\xi_1\pa_x\FG\right)+\pa_x\phi\pa_\xi\FG+Q(\FG,\FG)-\pa_t\overline{\FG},
\end{split}
\end{equation}
where we have used the fact that
$$
\FP_1^{\FM}\left(\xi_1\pa_x\FM\right)-L_{\FM}\overline{\FG}
=\frac{3}{2\ta}{\FP_1^{\FM}}\left[\xi_1\FM\left(\xi\cdot\pa_x\widetilde{u}+\frac{|\xi-u|^2}{2\ta}\pa_x\widetilde{\ta}\right)\right].
$$

\begin{remark}
It is worth pointing out that $\overline{\FG}$ defined in \eqref{def.ng} is designed to deal with the
linear term ${\bf P}_1^{\FM}\left(\xi_1\pa_x\FM\right)$ which can not be directly controlled.
\end{remark}

Taking the inner product of \eqref{g.eq1.} with $\frac{\widetilde{\FG}}{\FM_{*}}$ over ${\R}\times{\R}^3$, one has
\begin{equation}\label{zero.g.eng.}
\begin{split}
\frac{1}{2}\frac{d}{dt}&\int_{{\R}\times{\R}^3}\frac{\left|\widetilde{\FG}\right|^2}{{\FM_{*}}}dxd\xi
\underbrace{-{\displaystyle\int_{{\R}\times{\R}^3}}\frac{\widetilde{\FG}L_{\FM}\widetilde{\FG}}{\FM_{*}}dxd\xi }_{\CJ_1}\\
=&\underbrace{-\left(\frac{3}{2\ta}{\FP_1^{\FM}}
\left[\xi_1\FM\left(\xi\cdot\pa_x\widetilde{u}+\frac{|\xi-u|^2}{2\ta}\pa_x\widetilde{\ta}\right)\right],\frac{\widetilde{\FG}}{\FM_{*}}\right)}_{\CJ_2}
\underbrace{-\left(\FP_1^{\FM}\left(\xi_1\pa_x\FG\right),\frac{\widetilde{\FG}}{\FM_{*}}\right)}_{\CJ_3}
\\
&\underbrace{+\left(\pa_x\phi\pa_{\xi_1}\FG,\frac{\widetilde{\FG}}{\FM_{*}}\right)}_{\CJ_4}
+\underbrace{\left(Q(\FG,\FG),\frac{\widetilde{\FG}}{\FM_{*}}\right)}_{\CJ_5}
\underbrace{-\left(\pa_t\overline{\FG},\frac{\widetilde{\FG}}{\FM_{*}}\right)}_{\CJ_6}.
\end{split}
\end{equation}
From Lemma \ref{co.est.}, we see that
\begin{eqnarray*}
\begin{array}{rl}
\CJ_1\geq \de{\displaystyle
\int_{{\R}\times{\R}^3}}\frac{(1+|\xi|)\left|\widetilde{\FG}\right|^2}{\FM_{*}}dxd\xi .
\end{array}
\end{eqnarray*}
Moreover, we get from Cauchy-Schwarz's inequality with $0<\eta<1$ and Remark \ref{rem.rjads}
that
\begin{eqnarray*}
\begin{array}{rl}
|\CJ_2|\leq& \eta{\displaystyle
\int_{{\R}\times{\R}^3}}\frac{(1+|\xi|)\left|\widetilde{\FG}\right|^2}{\FM_{*}}dxd\xi +C_\eta\left\|\pa_x\left[\widetilde{u},\widetilde{\ta}\right]\right\|^2.
\end{array}
\end{eqnarray*}
By integration by parts and applying Lemma \ref{cl.Re.Re2.}, the a priori assumption \eqref{aps} and \eqref{d.bG}, one can see that
$\CJ_3$, $\CJ_4$  and $\CJ_6$ can be bounded as follows
\begin{equation*}
\begin{split}
|\CJ_3|\leq& \left|\left(\xi_1\pa_x\FG,\frac{\widetilde{\FG}}{\FM_{*}}\right)\right|+
\left|\left(\FP_0^{\FM}\left(\xi_1\pa_x\FG\right),\frac{\widetilde{\FG}}{\FM_{*}}\right)\right|\\
\leq& \eta\int_{{\R}\times{\R}^3}\frac{(1+|\xi|)\left|\widetilde{\FG}\right|^2}{\FM_{*}}dxd\xi
+C_\eta\int_{{\R}\times{\R}^3}\frac{(1+|\xi|)\left|\pa_x\FG\right|^2}{\FM_{*}}dxd\xi ,
\end{split}
\end{equation*}
\begin{equation*}
\begin{split}
|\CJ_4|\leq& \left|\left(\pa_x\phi\pa_{\xi_1}\widetilde{\FG},\frac{\widetilde{\FG}}{\FM_{*}}\right)\right|+
\left|\left(\pa_x\phi\pa_{\xi_1}\overline{\FG},\frac{\widetilde{\FG}}{\FM_{*}}\right)\right|
\\ \leq& C\left\|\pa_x\phi\right\|_{L^\infty}\int_{{\R}\times{\R}^3}\frac{(1+|\xi|)\left|\widetilde{\FG}\right|^2}{\FM_{*}}dxd\xi
+\eta\int_{{\R}\times{\R}^3}\frac{(1+|\xi|)\left|\widetilde{\FG}\right|^2}{\FM_{*}}dxd\xi
\\
&+C_\eta \left\|\pa_x\phi\pa_x\left[u_1^r,\ta^r\right]\right\|^2
\\
\leq& (\eps_{0}+\eta)\int_{{\R}\times{\R}^3}\frac{(1+|\xi|)\left|\widetilde{\FG}\right|^2}{\FM_{*}}dxd\xi
+C_\eta\eps_0\left\|\pa_x\widetilde{\phi}\right\|^2
+C_\eta\eps(1+t)^{-2},
\end{split}
\end{equation*}
\begin{equation*}
\begin{split}
|\CJ_6|
\leq \eta\int_{{\R}\times{\R}^3}\frac{(1+|\xi|)\left|\widetilde{\FG}\right|^2}{\FM_{*}}dxd\xi
+C_\eta\eps\sum\limits_{{|\al|=1}}\left\|\pa^\al\left[\widetilde{\rho},\widetilde{u},\widetilde{\ta}\right]\right\|^2+C_\eta\eps(1+t)^{-2}.
\end{split}
\end{equation*}
For $\CJ_5$, it follows from \eqref{db.G} and Cauchy-Schwarz's inequality with $\eta$ that
\begin{equation*}
\begin{split}
|\CJ_5|\leq& \eta\int_{{\R}\times{\R}^3}\frac{(1+|\xi|)\left|\widetilde{\FG}\right|^2}{\FM_{*}}dxd\xi
+C_\eta\int_{\R}\left(\int_{\R^3}\frac{(1+|\xi|)|\FG|^2}{\FM_*}d\xi\right) \left(\int_{\R^3}\frac{|\FG|^2}{\FM_*}d\xi\right) dx\\
\leq& (\eps_{0}+\eta)\int_{{\R}\times{\R}^3}\frac{(1+|\xi|)\left|\widetilde{\FG}\right|^2}{\FM_{*}}dxd\xi
+C_\eta \eps(1+t)^{-2}.
\end{split}
\end{equation*}
Now substituting the above estimates into \eqref{zero.g.eng.}, we arrive at
\begin{equation}\label{zero.g.eng1.}
\begin{split}
\frac{d}{dt}&\int_{{\R}\times{\R}^3}\frac{\left|\widetilde{\FG}\right|^2}{{\FM_{*}}}dxd\xi
+ \la{\displaystyle
\int_{{\R}\times{\R}^3}}\frac{(1+|\xi|)\left|\widetilde{\FG}\right|^2}{\FM_{*}}dxd\xi \\
\leq &C\int_{{\R}\times{\R}^3}\frac{(1+|\xi|)\left|\pa_x\FG\right|^2}{\FM_{*}}dxd\xi
+C{\sum\limits_{|\al|=1}\left\{ \eps \left\|\pa^{\al}\widetilde{\rho}\right\|^2+\left\|\pa^{\al}\left[\widetilde{u},\widetilde{\ta}\right]\right\|^2\right\}
+C{\eps_0}\left\|\pa_x\widetilde{\phi}\right\|^2}
+C\eps(1+t)^{-2}.
\end{split}
\end{equation}

\noindent{\bf Step 2.} {\it Higher order dissipation.} Let us now deduce the higher order dissipation of $\FG$.
Note that even for $|\al|\geq1$, one can not directly obtain the dissipation of $\pa^{\al} \FG/\sqrt{\FM}$ with the aid of \eqref{micBE},
since the linear term $\left(\pa^{\al} {\bf P}^{{\bf M}}_1\left(\xi_1\pa_x{\bf
M}\right),\frac{\pa^{\al} \FG}{\FM}\right)$
makes a big trouble. To overcome this difficulty, we first deduce the energy estimates on $\pa^{\al} F$ by using the original equation $\eqref{VPB}_1$ with respect to the local Maxwellian $\FM$,
in this case, the corresponding term becomes $\left(\pa^{\al}\left(\xi_1\pa_xF\right),\frac{\pa^{\al} F}{\FM}\right)$, which can be smoothly controlled.
Then we turn to obtain another estimates based on the global Maxwellian $\FM_*$. The desired estimates will be derived by the interplay of these two kinds of weighted energy estimates.

Let $1\leq|\al|\leq2$. Taking the inner product of $\pa^{\al} \eqref{VPB}_1$ with $\frac{R\ta\pa^{\al} F}{\FM}$ with respect to $x$ and $\xi$ over $\R\times\R^3$, one has
\begin{equation}\label{2d.F}
\begin{split}
\frac{1}{2}\frac{d}{dt}&\int_{{\R}\times{\R}^3}\frac{R\ta\left|\pa^{\al} F\right|^2}{\FM}dxd\xi
\underbrace{-\left(L_{\FM}\pa^{\al}\FG,\frac{R\ta\pa^{\al} \FG}{\FM}\right)}_{\CJ_{7}}
\underbrace{-\left(\pa^{\al}\pa_x\phi\pa_{\xi_1}\FM,\frac{R\ta\pa^{\al} \FM}{\FM}\right)}_{\CJ_{8}}\\
=&\underbrace{ -\frac{1}{2}\left(R\ta(\pa^{\al} F)^2,\FM^{-2}\pa_t\FM\right)+\frac{1}{2}\left((\pa^{\al} F)^2,\frac{R\pa_t\ta}{\FM}\right)}_{\CJ_{9}}
\\&+\underbrace{\sum\limits_{{0<\al'\leq \al}}C_{\al'}^{\al}
\left(Q(\pa^{\al'} \FM, \pa^{\al-\al'}\FG)+Q(\pa^{\al-\al'}\FG,\pa^{\al'} \FM),\frac{R\ta\pa^{\al} F}{\FM}\right)}_{\CJ_{10}}\\&+\underbrace{\left(L_{\FM}\pa^{\al}\FG, \FP_1^{\FM}\left(\frac{R\ta\pa^{\al} \FM}{\FM}\right)\right)}_{\CJ_{11}}
\underbrace{-\left(\xi_1\pa^{\al}\pa_xF,\frac{R\ta\pa^{\al} F}{\FM}\right)}_{\CJ_{12}}
\\&\underbrace{+\sum\limits_{{0<\al'<\al}}C_{\al'}^{\al}\left(\pa^{\al-\al'} \pa_x\phi \pa^{\al'} \pa_{\xi_1}F,\frac{R\ta\pa^{\al} F}{\FM}\right)}_{\CJ_{13}}
\\&\underbrace{+\left(\pa_x\phi\pa^{\al}\pa_{\xi_1}F,\frac{R\ta\pa^{\al} F}{\FM}\right)}_{\CJ_{14}}
\underbrace{+\left(\pa^{\al}\pa_x\phi\pa_{\xi_1}\FG,\frac{R\ta\pa^{\al} F}{\FM}\right)}_{\CJ_{15}}
+\underbrace{\left(\pa^{\al} Q(\FG,\FG),\frac{R\ta\pa^{\al} F}{\FM}\right)}_{\CJ_{16}},
\end{split}
\end{equation}
where we have used the fact that
$$
\left(\pa^{\al}\pa_x\phi\pa_{\xi_1}\FM,\frac{R\ta\pa^{\al} \FG}{\FM}\right)=0.
$$
We note that $\CJ_{13}$ is just $\CJ_{14}$ when $|\al|=1$.

Lemma \ref{co.est.} implies that
\begin{equation*}
\begin{split}
\CJ_{7}\geq \de\int_{{\R}\times{\R}^3}\frac{(1+|\xi|)\left|\pa^{\al}\FG\right|^2}{\FM}dxd\xi .
\end{split}
\end{equation*}
For $\CJ_{8}$, 
we have
\begin{equation}\label{CJ8.c}
\begin{split}
&{\sum\limits_{1\leq|\al|\leq 2}\left|\CJ_{8}-\frac{1}{2}\frac{d}{dt}\left\{\left\|\pa^{\al}\pa_x\widetilde{\phi}\right\|^2
+\left\|\sqrt{\rho_{e}'(\phi^r)}\pa^{\al}\widetilde{\phi}\right\|^2\right\}
+\frac{1}{2}\frac{d}{dt}\left(\pa^\al\widetilde{\phi}\int_{\phi^r}^{\phi}\rho_{e}''(\varrho) d\varrho
,\pa^{\al}\widetilde{\phi}\right)\right|}
\\& \qquad  \lesssim
(\eps_0+\eta)\sum\limits_{1\leq|\al|\leq2}\left\|\pa^\al\widetilde{\phi}\right\|_{H^1}^2
+(\eps_0+\eta)\sum\limits_{1\leq|\al|\leq2}\left\|\pa^\al\left[\widetilde{\rho},\widetilde{u}_1\right]\right\|^2
\\& \qquad\quad+{C_\eta}(1+t)^{-2}\left\|\left[\widetilde{\rho},\widetilde{u}_1,\widetilde{\phi}\right]\right\|^2
+{C_\eta}\eps^{3/4}(1+t)^{-5/4},
\end{split}
\end{equation}
whose proof is given in the appendix.
For the remaining terms in \eqref{2d.F},
we only present in what follows the estimations in the case of $|\al|=1$. When $|\al|=2$, since $\left\|\pa^{\al}\left[\rh^r,u^r,\ta^r,\phi^r\right]\right\|_{L^p}$ $(p\geq1)$ decays much faster, the corresponding estimates are similar to those for the case $|\al|=1$ and are much easier to obtain. Hence the details for the case $|\al|=2$ are omitted for brevity.
Now, by applying Lemma \ref{cl.Re.Re2.}, Sobolev's inequality and Cauchy-Schwarz's inequality, we have
\begin{equation*}
\begin{split}
\sum\limits_{|\al|=1}|\CJ_{9}|\lesssim& \sum\limits_{|\al|=1}
\int_{\R\times{\R}^3}\frac{(1+|\xi|)|\pa_t[\rho,u,\ta]|(\left|\pa^{\al} \FM\right|^2+\left|\pa^{\al} \FG\right|^2)}{\FM_*}dxd\xi
\\
\lesssim& \sum\limits_{|\al|=1}\bigg\{\int_{\R}\left|\pa^{\al}\left[\widetilde{\rho},\widetilde{u},\widetilde{\ta}\right]\right|^2
|\pa_t\left[\rho,u,\ta \right]|dx+
\int_{\R}|\pa^{\al}[\rho^r,u^r,\ta^r]|^2
|\pa_t[\rho^r,u^r,\ta^r]|dx
\\&+\int_{\R}|\pa^{\al}[\rho^r,u^r,\ta^r]|^2
\left|\pa_t\left[\widetilde{\rho},\widetilde{u},\widetilde{\ta}\right]\right|dx\bigg\}
+\eps_{0}\sum\limits_{|\al|=1}\int_{\R\times{\R}^3}\frac{(1+|\xi|)\left|\pa^{\al}\FG\right|^2}{\FM_*}dxd\xi \\
\lesssim& \sum\limits_{|\al|=1}\|\pa_t[\rho^r,u^r,\ta^r]\|^{1/2}\left\|\pa_x\pa_t\left[\rho^r,u^r,\ta^r\right]\right\|^{1/2}
\|\pa^{\al}[\rho^r,u^r,\ta^r]\|^2
\\&+\sum\limits_{|\al|=1}{C_\eta}\|\pa^{\al}[\rho^r,u^r,\ta^r]\|^2\|\pa^{\al}[\rho^r,u^r,\ta^r]\|\|\pa_x\pa^{\al}[\rho^r,u^r,\ta^r]\|
\\&+\sum\limits_{|\al|=1}(\eps_0+\eta)\left\|\pa^{\al}\left[\widetilde{\rho},\widetilde{u},\widetilde{\ta}\right]\right\|^2
+\eps_{0}\sum\limits_{|\al|=1}\int_{\R\times{\R}^3}\frac{(1+|\xi|)\left|\pa^{\al}\FG\right|^2}{\FM_*}dxd\xi
\\\lesssim &\sum\limits_{|\al|=1}(\eps_0+\eta)\left\|\pa^{\al}\left[\widetilde{\rho},\widetilde{u},\widetilde{\ta}\right]\right\|^2
+\eps_{0}\sum\limits_{|\al|=1}\int_{\R\times{\R}^3}\frac{(1+|\xi|)\left|\pa^{\al}\FG\right|^2}{\FM_*}dxd\xi
+{C_\eta}\eps^{3/4}(1+t)^{-5/4},
\end{split}
\end{equation*}
where we also used the trivial inequality $\eps(1+t)^{-2}<\eps^{3/4}(1+t)^{-5/4}$ for $0<\eps<1$.

When $|\al|=1$, one sees that $\CJ_{10}$ becomes
$$
\sum\limits_{|\al|=1}\left(Q(\pa^{\al} \FM, \FG)+Q(\FG,\pa^{\al} \FM),\frac{R\ta\pa^{\al} \FG}{\FM}\right),
$$
and then it follows that
\begin{equation*}
\begin{split}
 \sum\limits_{|\al|=1}|\CJ_{10}|\lesssim& \eta\int_{\R\times{\R}^3}\frac{(1+|\xi|)\left|\pa^{\al} \FG\right|^2}{\FM}dxd\xi
+C_\eta\int_{\R}\left(\int_{\R^3}\frac{(1+|\xi|)|\pa^{\al}\FM|^2}{\FM}d\xi\right)
\left(\int_{\R^3}\frac{|\widetilde{\FG}+\overline{\FG}|^2}{\FM}d\xi\right) dx
\\&+C_\eta\int_{\R}\left(\int_{\R^3}\frac{|\pa^{\al}\FM|^2}{\FM}d\xi\right)
\left(\int_{\R^3}\frac{(1+|\xi|)|\widetilde{\FG}+\overline{\FG}|^2}{\FM}d\xi\right) dx
\\ \lesssim& \eta\int_{\R\times{\R}^3}\frac{(1+|\xi|)\left|\pa^{\al} \FG\right|^2}{\FM}dxd\xi +{C_\eta}\eps_{0}\int_{{\R}\times{\R}^3}\frac{(1+|\xi|)\left|\widetilde{\FG}\right|^2}{\FM_{*}}dxd\xi
\\&+{C_\eta}\eps_{0}\sum\limits_{|\al|=1}\left\|\pa^{\al}\left[\widetilde{\rho},\widetilde{u},\widetilde{\ta}\right]\right\|^2
+{C_\eta}\eps(1+t)^{-2}.
\end{split}
\end{equation*}
Note that $\CJ_{11}$ vanishes for $|\al|=1$. By integration by parts and performing the similar calculations as for obtaining $\CJ_{9}$, one sees that
$|\CJ_{12}|$ is bounded by
\begin{equation*}
\begin{split}
 \sum\limits_{|\al|=1}&
\int_{\R\times{\R}^3}\frac{(1+|\xi|)|\pa^{\al}[\rho,u,\ta]|(\left|\pa^{\al} \FM\right|^2+\left|\pa^{\al} \FG\right|^2)}{\FM_*}dxd\xi
\\ \lesssim &(\eps_0+\eta)\sum\limits_{|\al|=1}\left\|\pa^{\al}\left[\widetilde{\rho},\widetilde{u},\widetilde{\ta}\right]\right\|^2
+\eps_{0}\sum\limits_{|\al|=1}\int_{\R\times{\R}^3}\frac{(1+|\xi|)\left|\pa^{\al}\FG\right|^2}{\FM_*}dxd\xi
+{C_\eta}\eps^{3/4}(1+t)^{-5/4}.
\end{split}
\end{equation*}
For $\CJ_{13}$ and $\CJ_{14}$ with $|\al|=1$, we use $F=\FM+\FG$ again, to obtain
\begin{equation*}
\begin{split}
|\CJ_{13}|,\ |\CJ_{14}|=&\Bigg|\left(\pa_x\phi\pa^{\al}\pa_{\xi_1}\FG,\frac{R\ta\pa^{\al}\FG}{\FM}\right)
+\left(\pa_x\phi\pa^{\al}\pa_{\xi_1}\FM,\frac{R\ta\pa^{\al}\FM}{\FM}\right)
+\left(\pa_x\phi\pa^{\al}\pa_{\xi_1}\FM,\frac{R\ta\pa^{\al}\FG}{\FM}\right)\Bigg|\\
\lesssim &{\sum\limits_{|\al|=1}\int_{\R}|\pa_x\phi||\pa^{\al} [\rho,u,\ta]|^2 dx
+\sum\limits_{|\al|=1}\eps_0\left\|\pa^{\al}\left[\widetilde{\rho},\widetilde{u},\widetilde{\ta},\widetilde{\phi}\right]\right\|^2}
+\eps_0\sum\limits_{|\al|=1}\int_{\R\times{\R}^3}\frac{|\pa^{\al}\pa_{\xi_1}\overline{\FG}|^2}{\FM}dxd\xi
\\&+\eps_{0}\sum\limits_{|\al|=1}\int_{\R\times{\R}^3}\frac{(1+|\xi|)
\left|\pa^{\al}\pa_{\xi_1}\widetilde{\FG}\right|^2}{\FM}dxd\xi
+C_\eta\sum\limits_{|\al|=1}\left\|\pa^\al\phi^r\pa^\al\left[\rh^r,u^r,\ta^r\right]\right\|^2
\\&+(\eps_{0}+\eta)\sum\limits_{|\al|=1}\int_{\R\times{\R}^3}\frac{(1+|\xi|)\left|\pa^{\al}\FG\right|^2}{\FM}dxd\xi
\\ \lesssim &\sum\limits_{|\al|=1}\eps_0\left\|\pa^{\al}\left[\widetilde{\rho},\widetilde{u},\widetilde{\ta},\widetilde{\phi}\right]\right\|^2
+\eps_{0}\sum\limits_{|\al|=1}\int_{\R\times{\R}^3}\frac{(1+|\xi|)
\left|\pa^{\al}\pa_{\xi_1}\widetilde{\FG}\right|^2}{\FM}dxd\xi
\\&+(\eps_{0}+\eta)\sum\limits_{|\al|=1}\int_{\R\times{\R}^3}\frac{(1+|\xi|)\left|\pa^{\al}\FG\right|^2}{\FM}dxd\xi
+{C_\eta}\eps^{3/4}(1+t)^{-5/4},
\end{split}
\end{equation*}
and
\begin{equation*}
\begin{split}
\sum\limits_{|\al|=1}|\CJ_{15}|=&\sum\limits_{|\al|=1}\left|\left(\pa^{\al}\pa_x\phi\pa_{\xi_1}\FG,\frac{R\ta\pa^{\al}\FG}{\FM}\right)\right|\\
\lesssim &\sum\limits_{|\al|=1}{C_\eta}\int_{\R\times{\R}^3}|\pa^{\al}\pa_x\phi|^2\frac{|\pa_{\xi_1}\overline{\FG}|^2}{\FM}dxd\xi
+\eps_{0}\int_{\R\times{\R}^3}\frac{(1+|\xi|)
\left|\pa_{\xi_1}\widetilde{\FG}\right|^2}{\FM}dxd\xi
\\&+(\eps_{0}+\eta)\sum\limits_{|\al|=1}\int_{\R\times{\R}^3}\frac{(1+|\xi|)\left|\pa^{\al}\FG\right|^2}{\FM}dxd\xi
\\ \lesssim &\sum\limits_{|\al|=1}{C_\eta}\eps_0\left\|\pa_x\pa^{\al}\widetilde{\phi}\right\|^2
+\eps_{0}\int_{\R\times{\R}^3}\frac{(1+|\xi|)
\left|\pa_{\xi_1}\widetilde{\FG}\right|^2}{\FM}dxd\xi
\\&+(\eps_{0}+\eta)\sum\limits_{|\al|=1}\int_{\R\times{\R}^3}\frac{(1+|\xi|)\left|\pa^{\al}\FG\right|^2}{\FM}dxd\xi
+{C_\eta}\eps(1+t)^{-2}.
\end{split}
\end{equation*}
As to the last term $\CJ_{16}$, we get from Lemma \ref{est.nonop} and Cauchy-Schwarz's inequality that
\begin{equation*}
\begin{split}
\sum\limits_{|\al|=1}|\CJ_{16}|=&\sum\limits_{|\al|=1}\left|\left(\pa^{\al} Q(\FG,\FG),\frac{R\ta\pa^{\al} \FG}{\FM}\right)\right|
\\ \lesssim& \eta\sum\limits_{|\al|=1}\int_{\R\times{\R}^3}\frac{(1+|\xi|)\left|\pa^{\al} \FG\right|^2}{\FM}dxd\xi
+\sum\limits_{|\al|=1}{C_\eta}\int_{\R}\left(\int_{\R^3}\frac{(1+|\xi|)\left|\pa^{\al} \FG\right|^2}{\FM}d\xi\right)
\left(\int_{\R^3}\frac{\left|\FG\right|^2}{\FM}d\xi\right) dx
\\&+\sum\limits_{|\al|=1}{C_\eta}\int_{\R}\left(\int_{\R^3}\frac{\left|\pa^{\al} \FG\right|^2}{\FM}d\xi\right)
\left(\int_{\R^3}\frac{(1+|\xi|)\left|\FG\right|^2}{\FM}d\xi\right) dx
\\ \lesssim&(\eta+\eps_0)\sum\limits_{|\al|=1}\int_{\R\times{\R}^3}\frac{(1+|\xi|)\left|\pa^{\al} \FG\right|^2}{\FM}dxd\xi
+{C_\eta}\eps_0\int_{\R\times{\R}^3}\frac{(1+|\xi|)\left| \widetilde{\FG}\right|^2}{\FM}dxd\xi .
\end{split}
\end{equation*}
It should be noted that when $|\al|=2$,
$$
\CJ_{16}=\left(\pa^{\al} Q(\FG,\FG),\frac{R\ta\pa^{\al} \FM}{\FM}\right)+\left(\pa^{\al} Q(\FG,\FG),\frac{R\ta\pa^{\al} \FG}{\FM}\right),
$$
with $\left(\pa^{\al} Q(\FG,\FG),\frac{R\ta\pa^{\al} \FM}{\FM}\right)$ being non-zero, but in case $|\al|=2$, the term $\frac{R\ta\pa^{\al} \FM}{\FM}$ becomes quadratic, which can be
handled as in \eqref{aps2}.

Substituting the above estimates for $\CJ_{l}$ $(7\leq l\leq 16)$ into \eqref{2d.F}, we see that
\begin{equation}\label{2d.F.sum}
\begin{split}
\frac{d}{dt}&\left\{\sum\limits_{1\leq|\al|\leq2}\int_{{\R}\times{\R}^3}\frac{R\ta\left|\pa^{\al} F\right|^2}{\FM}dxd\xi
+\CE_2(\widetilde{\phi})
\right\}
+\la\sum\limits_{1\leq|\al|\leq2}\int_{\R\times{\R}^3}\frac{(1+|\xi|)\left|\pa^{\al} \FG\right|^2}{\FM}dxd\xi \\
\lesssim&{C_\eta}\eps_{0}\sum\limits_{1\leq|\al|\leq2}\int_{\R\times{\R}^3}\frac{(1+|\xi|)\left|\pa^{\al}\FG\right|^2}{\FM_*}dxd\xi
+{C_\eta}\eps_{0}\sum\limits_{|\al|\leq1}\int_{\R\times{\R}^3}\frac{(1+|\xi|)\left|\pa^{\al}\pa_{\xi_1} \widetilde{\FG}\right|^2}{\FM}dxd\xi
\\&+{C_\eta}\eps_{0}\int_{{\R}\times{\R}^3}\frac{(1+|\xi|)\left|\widetilde{\FG}\right|^2}{\FM}dxd\xi
+(\eps_{0}+\eta)\sum\limits_{1\leq|\al|\leq2}\left\{\left\|\pa^{\al}\left[\widetilde{\rho},\widetilde{u},\widetilde{\ta}\right]\right\|^2
+\left\|\pa^{\al}\widetilde{\phi}\right\|^2_{H^1}\right\}\\
&
+{C_\eta}\eps^{3/4}(1+t)^{-5/4},
\end{split}
\end{equation}
where we have set
\begin{equation*}
\CE_2(\widetilde{\phi})=\sum\limits_{1\leq|\al|\leq2}\left\{\left\|\pa^{\al}\pa_x\widetilde{\phi}\right\|^2
+\left\|\sqrt{\rho_{e}'(\phi^r)}\pa^{\al}\widetilde{\phi}\right\|^2\right\}
-\sum\limits_{1\leq|\al|\leq2}\left(\pa^\al\widetilde{\phi}\int_{\phi^r}^{\phi}\rho_{e}''(\varrho) d\varrho
,\pa^{\al}\widetilde{\phi}\right)
\sim \left\|\pa_x\widetilde{\phi}\right\|_{{H^2}}^2.
\end{equation*}
Similarly, one can obtain the following energy estimates for $\pa^{\al} F$ $(1\leq|\al|\leq2)$ with respect to the global Maxwellian $\FM_*$: 
\begin{equation}\label{2d.F.sum2}
\begin{split}
\frac{d}{dt}&\sum\limits_{1\leq|\al|\leq2}\int_{{\R}\times{\R}^3}\frac{\left|\pa^{\al} F\right|^2}{\FM_*}dxd\xi
+\la\sum\limits_{1\leq|\al|\leq2}\int_{\R\times{\R}^3}\frac{(1+|\xi|)\left|\pa^{\al} \FG\right|^2}{\FM_*}dxd\xi \\
\lesssim&
(\eps_{0}+\eta)\sum\limits_{|\al|\leq1}\int_{\R\times{\R}^3}\frac{(1+|\xi|)\left|\pa^{\al}\pa_{\xi_1} \widetilde{\FG}\right|^2}{\FM_*}dxd\xi +(\eps_{0}+\eta)\int_{{\R}\times{\R}^3}\frac{(1+|\xi|)\left|\widetilde{\FG}\right|^2}{\FM_{*}}dxd\xi
\\&+{C_\eta}\sum\limits_{1\leq|\al|\leq2}\left\|\pa^{\al}\left[\widetilde{\rho},\widetilde{u},\widetilde{\ta}\right]\right\|^2
+{C_\eta}\sum\limits_{1\leq|\al|\leq2}
\left\|\pa^{\al}\widetilde{\phi}\right\|^2_{H^1}
+{C_\eta}\eps^{3/4}(1+t)^{-5/4},
\end{split}
\end{equation}
whose proof is also given in the appendix.
With \eqref{2d.F.sum} in hand, letting $1\gg\ka_7>0$, we get from the summation of \eqref{2d.F.sum} and $\eqref{macro.eng}\times\ka_7$ that
\begin{equation}\label{mi.diss2}
\begin{split}
\ka_7\frac{d}{dt}&\CE_1(\widetilde{\rho},\widetilde{u},\widetilde{\theta},\widetilde{\phi})
-\ka_7\ka_0\frac{d}{dt}\sum\limits_{|\al|=1}\left(\pa^{\al}\widetilde{u}_1,\pa^{\al}\pa_x\widetilde{\rho}\right)
\\
&+\frac{d}{dt}\sum\limits_{1\leq|\al|\leq2}\left\{\int_{{\R}\times{\R}^3}\frac{R\ta\left|\pa^{\al} F\right|^2}{\FM}dxd\xi
+\CE_2(\widetilde{\phi})
\right\}
+\la\int_{\R}\left|\left[\widetilde{\rho},\widetilde{u}_1,\widetilde{S},\widetilde{\phi}\right]\right|^2\pa_xu^r_{1}dx
\\&+\la\sum\limits_{1\leq|\al|\leq 2}\left\|\pa^{\al}\left[\widetilde{\rho},\widetilde{u},\widetilde{\theta}\right]\right\|^2
+\la\sum\limits_{1\leq|\al|\leq2}\left\|\pa^{\al}\widetilde{\phi}\right\|_{H^1}^2
+\la \sum\limits_{1\leq|\al|\leq2}\int_{{\R}\times{\R}^3}\frac{(1+|\xi|)\left|\pa^{\al}\FG\right|^2}{\FM}dxd\xi
\\
\lesssim&(1+t)^{-\zeta_1}\left\|\left[\widetilde{\rho},\widetilde{u},\widetilde{\ta},\widetilde{\phi}\right]\right\|^2
+\eps^{\sigma_0}(1+t)^{-\zeta_0}
+\eps_{0}\sum\limits_{|\al|\leq1}\int_{\R\times\R^3}\frac{|\pa_{\xi_1}\pa^{\al}\widetilde{\FG}|^2}{\FM}dxd\xi
\\&+\eps_{0}\sum\limits_{1\leq|\al|\leq2}\int_{\R\times{\R}^3}\frac{(1+|\xi|)\left|\pa^{\al}\FG\right|^2}{\FM_*}dxd\xi
+\eps_{0}\int_{{\R}\times{\R}^3}\frac{(1+|\xi|)\left|\widetilde{\FG}\right|^2}{\FM_{*}}dxd\xi ,
\end{split}
\end{equation}
where we have also used the fact that $\sigma_0<3/4$ and $1<\zeta_0<5/4$.

On the other hand, by choosing $1\gg\ka_{8}\gg\ka_{9}>0$, it follows from the summation of $\eqref{zero.g.eng1.}\times\ka_{9}$
 and $\eqref{2d.F.sum2}\times\ka_{8}$ that
\begin{equation}\label{mi.diss3}
\begin{split}
\ka_{9}\frac{d}{dt}&\int_{{\R}\times{\R}^3}\frac{\left|\widetilde{\FG}\right|^2}{{\FM_{*}}}dxd\xi
+\ka_{8}\frac{d}{dt}\sum\limits_{1\leq|\al|\leq2}\int_{{\R}\times{\R}^3}\frac{\left|\pa^{\al} F\right|^2}{\FM_*}dxd\xi
\\
&+\la\sum\limits_{1\leq|\al|\leq2}\int_{\R\times{\R}^3}\frac{(1+|\xi|)\left|\pa^{\al} \FG\right|^2}{\FM_*}dxd\xi
+\la{\displaystyle
\int_{{\R}\times{\R}^3}}\frac{(1+|\xi|)\left|\widetilde{\FG}\right|^2}{\FM_{*}}dxd\xi
\\
\lesssim&(\ka_{8}+\ka_{9}){C_\eta}\sum\limits_{1\leq|\al|\leq2}
\left\{\left\|\pa^{\al}\left[\widetilde{\rho},\widetilde{u},\widetilde{\ta}\right]\right\|^2 +\left\|\pa^{\al}\widetilde{\phi}\right\|^2_{H^1}\right\}
\\
&+{C_\eta}\eps(1+t)^{-2}+(\eps_{0}+\eta)\sum\limits_{|\al|\leq1}\int_{\R\times\R^3}\frac{|\pa_{\xi_1}\pa^{\al}\widetilde{\FG}|^2}{\FM}dxd\xi.
\end{split}
\end{equation}

\noindent{\bf Step 3.} {\it Energy estimates with mixed derivatives.} In what follows, we deduce the energy estimates on the mixed derivative term $\pa^{\al}\pa^\be \widetilde{\FG}$. To do so, letting $|\be|\geq1$ and $|\al|+|\be|\leq 2$, acting $\pa^{\al}\pa^\be $ to \eqref{g.eq1.} and taking the inner product of the resulting equation with $\frac{\pa^{\al}\pa^\be \widetilde{\FG}}{\FM_*}$  over ${\R}\times{\R}^3$, one has
\begin{equation}\label{mixd.ip}
\begin{split}
\frac{1}{2}\frac{d}{dt}&\int_{{\R}\times{\R}^3}\frac{\left|\pa^{\al}\pa^\be \widetilde{\FG}\right|^2}{\FM_{*}}dxd\xi
-{\displaystyle\int_{{\R}\times{\R}^3}}\frac{\pa^{\al}\pa^\be \widetilde{\FG} L_{\FM}\pa^{\al}\pa^\be \widetilde{\FG}}{\FM_{*}}dxd\xi \\
=&{\sum\limits_{|\al'|+|\be'|\leq |\al|+|\be|-1\atop{\al'\leq\al,\be'\leq\be}}C^{\al,\be}_{\al',\be'}}\left(Q\left(\pa^{\al-\al'}\pa^{\be-\be'}\FM,\pa^{\al'}\pa^{\be'}\widetilde{\FG}\right)
+Q\left(\pa^{\al'}\pa^{\be'}\widetilde{\FG},\pa^{\al-\al'}\pa^{\be-\be'}\FM\right),
\frac{\pa^{\al}\pa^\be\widetilde{\FG}}{\FM_{*}}\right)\\
&-\left(\pa^{\al}\pa^\be\left(\frac{3}{2\ta}
{\FP^{\FM}_1}\left[\xi_1\FM\left(\xi_1\pa_x\widetilde{u}_1
+\frac{|\xi-u|^2}{2\ta}\pa_x\widetilde{\ta}\right)\right]\right),\frac{\pa^{\al}\pa^\be\widetilde{\FG}}{\FM_{*}}\right)
\\&
-\left(\pa^{\al}\pa^\be\left(\FP_1^{\FM}\left(\xi_1\pa_x\FG\right)\right),\frac{\pa^{\al}\pa^\be\widetilde{\FG}}{\FM_{*}}\right)
+\left(\pa^{\al}\pa^\be\left(\pa_x\phi\pa_{\xi_1}\FG\right),\frac{\pa^{\al}\pa^\be\widetilde{\FG}}{\FM_{*}}\right)
\\
&+\left(\pa^{\al}\pa^\be Q(\FG,\FG),\frac{\pa^{\al}\pa^\be\widetilde{\FG}}{\FM_{*}}\right)
-\left(\pa_t\pa^{\al}\pa^\be\overline{\FG},\frac{\pa^{\al}\pa^\be\widetilde{\FG}}{\FM_{*}}\right).
\end{split}
\end{equation}
Similar to those calculations in the above Step 2, we have 
\begin{equation}\label{mi.diss4}
\begin{split}
\frac{d}{dt}&\sum\limits_{|\al|+|\be|\leq 2\atop{|\be|\geq1}}C_{\al\be}\int_{{\R}\times{\R}^3}
\frac{\left|\pa^{\al}\pa^\be \widetilde{\FG}\right|^2}{\FM_{*}}dxd\xi
+\la \sum_{|\al|+|\be|\leq 2\atop{|\be|\geq1}}\int_{{\R}\times{\R}^3}\frac{(1+|\xi|)
\left|\pa^{\al}\pa^\be\widetilde{\FG}\right|^2}{\FM_{*}}dxd\xi \\
\lesssim&\sum\limits_{1\leq|\al|\leq2}\int_{{\R}\times{\R}^3}\frac{(1+|\xi|)\left|\pa^{\al}\FG\right|^2}{\FM_{*}}dxd\xi
+\int_{{\R}\times{\R}^3}\frac{(1+|\xi|)\left|\widetilde{\FG}\right|^2}{\FM_{*}}dxd\xi
\\&+ \sum\limits_{1\leq|\al|\leq2}
\left\|\pa^{\al}\left[\widetilde{\rho},\widetilde{u},\widetilde{\ta},\widetilde{\phi}\right]\right\|^2+\eps(1+t)^{-2},
\end{split}
\end{equation}
for suitable constants $C_{\al\be}>0$. The proof of \eqref{mi.diss4} above is given in the appendix.
Consequently, it follows from \eqref{mi.diss2}, \eqref{mi.diss3} and \eqref{mi.diss4} that
\begin{equation}\label{mi.diss5}
\begin{split}
K_0&\ka_7\frac{d}{dt}\CE_1(\widetilde{\rho},\widetilde{u},\widetilde{\theta},\widetilde{\phi})
-K_0\ka_7\ka_0\frac{d}{dt}\sum\limits_{|\al|=1}\left(\pa^{\al}\widetilde{u}_1,\pa^{\al}\pa_x\widetilde{\rho}\right)
\\
&+K_0\frac{d}{dt}\sum\limits_{1\leq|\al|\leq2}\left\{\int_{{\R}\times{\R}^3}\frac{R\theta\left|\pa^{\al} F\right|^2}{\FM}dxd\xi
+\CE_2(\widetilde{\phi})
\right\}
+\ka_{9}\frac{d}{dt}\int_{{\R}\times{\R}^3}\frac{\left|\widetilde{\FG}\right|^2}{{\FM_{*}}}dxd\xi
\\&
+\ka_{8}\frac{d}{dt}\sum\limits_{1\leq|\al|\leq2}\int_{{\R}\times{\R}^3}\frac{\left|\pa^{\al} F\right|^2}{\FM_*}dxd\xi
+\ka_{10}\frac{d}{dt}\sum\limits_{|\al|+|\be|\leq 2\atop{|\be|\geq1}}
\int_{{\R}\times{\R}^3}\frac{\left|\pa^{\al}\pa^\be \widetilde{\FG}\right|^2}{\FM_{*}}dxd\xi
\\&+\la\sum\limits_{1\leq|\al|\leq2}\int_{\R\times{\R}^3}\frac{(1+|\xi|)\left|\pa^{\al} \FG\right|^2}{\FM_*}dxd\xi
+\la{\displaystyle
\int_{{\R}\times{\R}^3}}\frac{(1+|\xi|)\left|\widetilde{\FG}\right|^2}{\FM_{*}}dxd\xi
\\&+\la \sum\limits_{1\leq|\al|\leq2}\int_{{\R}\times{\R}^3}\frac{(1+|\xi|)\left|\pa^{\al}\FG\right|^2}{\FM}dxd\xi
+\la \sum_{|\al|+|\be|\leq 2\atop{|\be|\geq1}}
\int_{{\R}\times{\R}^3}\frac{(1+|\xi|)\left|\pa^{\al}\pa^\be\widetilde{\FG}\right|^2}{\FM_{*}}dxd\xi
\\&+\la\int_{\R}\left|\left[\widetilde{\rho},\widetilde{u}_1,\widetilde{S}\right]\right|^2\pa_xu^r_{1}dx
+\la\sum\limits_{1\leq|\al|\leq 2}\left\|\pa^{\al}\left[\widetilde{\rho},\widetilde{u},\widetilde{\theta}\right]\right\|^2
+\la\sum\limits_{1\leq|\al|\leq2}\left\|\pa^{\al}\widetilde{\phi}\right\|_{H^1}^2\\
\lesssim&(1+t)^{-\zeta_1}\left\|\left[\widetilde{\rho},\widetilde{u},\widetilde{\ta},\widetilde{\phi}\right]\right\|^2
+\eps^{\sigma_0}(1+t)^{-\zeta_0},
\end{split}
\end{equation}
where $K_0$ is a positive large constant and $\ka_{10}$ is also a positive constant but suitably small. Then \eqref{g.eng.} follows from
\eqref{mi.diss5} and Gronwall's inequality. This ends the proof of Proposition \ref{g.eng.lem.}.
\end{proof}

\section{Local existence}
In this section,  we show the existence of the local-in-time solution in the function space $\widetilde{\CE}([0,T])$
for a small $T>0$ to the Cauchy problem \eqref{VPB}, \eqref{BE.Idata1} and \eqref{con.phi}. We adopt the iteration method as in \cite{G}
for the proof, which is based on a uniform energy estimate for the following sequence of iterating approximate solutions:
\begin{eqnarray}\label{iterate.F}
\begin{split}
\left\{\begin{array}{rll}
\begin{split}
&\left\{\partial_t+\xi_1\pa_x -\pa_x\phi^{n}\pa_{\xi_1}\right\}F^{n+1}+F^{n+1}(\xi)\int_{\R^3\times \S^2}|(\xi-\xi_\ast)\cdot\omega|
F^n(\xi_\ast)\,d\xi_\ast d\omega\\[2mm]
&\qquad\qquad\qquad
=\int_{\R^3\times \S^2}|(\xi-\xi_\ast)\cdot\omega|
F^n(\xi'_{\ast})F^n(\xi')\,d\xi_\ast d\omega,\\[2mm]
&-\pa_x^2\phi^{n}=\int_{\R^3}F^n(\xi)\,d\xi-\rho_e(\phi^n),\\[2mm]
&F^{n+1}(0,x,\xi)=F_0(x,\xi),\ n\geq0,\\[2mm]
&F^0(t,x,\xi)=F_0(x,\xi).
\end{split}
\end{array}\right.
\end{split}
\end{eqnarray}
Set $\FM_r=\FM_{[\rho^r(t,x),u^r(t,x),\theta^r(t,x)]}(\xi)$. Let
$$
F^{n}=g^n+\FM_r,\quad
\widetilde{\phi}^{\,n}=\phi^n-\phi^r,\ n\geq0.
$$
Then \eqref{iterate.F} is equivalent to
\begin{eqnarray}\label{iterate.g}
\begin{split}
\left\{\begin{array}{rll}
\begin{split}
&\left\{\partial_t+\xi_1\pa_x -\pa_x\widetilde{\phi}^{\,n}\pa_{\xi_1}\right\}g^{n+1}+\nu_{\FM_r}(\xi)g^{n+1}-Kg^n
=Q_{\textrm{gain}}(g^n,g^n)-Q_{\textrm{loss}}(g^n,g^{n+1})\\[2mm]
&\qquad\qquad-\left\{\pa_t+\xi_1\pa_x-\pa_x\widetilde{\phi}^{\,n}\pa_{\xi_1}\right\}\FM_r
+\pa_x\phi^r\pa_{\xi_1}g^{n+1}+\pa_x\phi^r\pa_{\xi_1}\FM_r,\\[2mm]
&-\pa_x^2\widetilde{\phi}^{\,n}=\int_{\R^3}g^n(\xi)\,d\xi+\rho_e(\phi^r)-\rho_e(\phi^n)+\pa_x^2\phi^r,\\[2mm]
&g^{n+1}(0,x,\xi)=F_0(x,\xi)-\FM_{[\rho^r(0,x),u^r(0,x),\ta^r(0,x)]},\ g^0=F_0(x,\xi)-\FM_{[\rho^r(0,x),u^r(0,x),\ta^r(0,x)]},
\end{split}
\end{array}\right.
\end{split}
\end{eqnarray}
where $\nu_{\FM_r}(\xi)$ is a multiplier, given by
$$
\nu_{\FM_r}(\xi)=\int_{\R^3\times \S_+^2}
|(\xi-\xi_\ast)\cdot\omega|\FM_r\,d\xi_\ast d\omega,
$$
and
$K(\xi,\xi_\ast)$ is a self-adjoint $L^2$ compact operator, defined by
$$
K g^n=Q_{\textrm{gain}}(\FM_r,g^n)-Q_{\textrm{loss}}(g^n,\FM_r)+Q_{\textrm{gain}}(g^n,\FM_r).
$$
As in \cite{YZ3}, $K(\xi,\xi_\ast)$ can be also presented as
\begin{eqnarray*}
\begin{split}
\left\{\begin{array}{rll}
\begin{split}
Kh=&\sqrt{{\bf M}_r(\xi)}
K_{{\bf M}_r}\left(\left(\frac{h}{\sqrt{{\bf M}_r}}\right)(\xi)\right),\
K_{{\bf M}_r}=K_{2{\bf M}_r}-K_{1{\bf M}_r},\\
K_{1{\bf M}_r}h=&\int_{\R^3\times \S_+^2}
|(\xi-\xi_\ast)\cdot\omega|\sqrt{\FM_r(\xi)}\sqrt{\FM_r(\xi_\ast)}h(\xi_\ast)\,d\xi_\ast d\omega,\\
K_{2{\bf M}_r}h=&\int_{\R^3\times \S_+^2}
|(\xi-\xi_\ast)\cdot\omega|\sqrt{\FM_r(\xi_\ast)}\left\{\sqrt{\FM_r(\xi')}h(\xi'_\ast)+\sqrt{\FM_r(\xi_\ast')}h(\xi')\right\}\,d\xi_\ast d\omega.
\end{split}
\end{array}\right.
\end{split}
\end{eqnarray*}
In what follows, we begin with the uniform bound in $n$ for $\|g^{n}\|_{\widetilde{\CE}_T}$ for a small time $T>0$.

\begin{lemma}\label{gn.bdd.lem}
The solution sequence $\{g^{n}\}_{n=1}^\infty$ is well defined. For  a sufficiently small constant $\eps_0>0$,
there exists $T^{\ast}=T^{\ast}(\eps_0)>0$ such that if
$$
\sum\limits_{|\al|+|\beta|\leq 2}\int_{{\R}\times{\R}^3}\frac{\left|\partial^\al \partial^\beta g^0(x,\xi)\right|^2}{{\bf M}_*}dxd\xi  +\eps\leq{\eps^2_0},
$$
then for any $n$, it holds that
\begin{equation}\label{gn.bdd}
\widetilde{Y}_T (g^n):=\widetilde{\CE}_T (g^{n})+\widetilde{\mathcal {D}}_T (g^{n})\leq 2\eps^2_0, \ \ \forall\,T\in[0,T^*),
\end{equation}
where $\widetilde{\mathcal
{D}}_T (h)$ is defined by
\begin{eqnarray*}
\widetilde{\mathcal
{D}}_T (h)=\sum\limits_{|\al|+|\beta|\leq 2}\int_0^{T}
\int_{{\R}\times{\R}^3}\frac{(1+|\xi|)\left|\partial^\al\partial^\beta h(t,x,\xi)\right|^2}{{\bf M}_*}dxd\xi .
\end{eqnarray*}
\end{lemma}
\begin{proof}
We intend to prove \eqref{gn.bdd} by induction on $n$. Namely, for each integer $l\geq0$, we are going to verify:
\begin{equation}\label{yl.bdd}
\widetilde{Y}_T(g^l)\leq 2\eps^2_0,
\end{equation}
for $0\leq T<T^\ast$, where $T^\ast>0$ will be suitably chosen later on.
Clearly the case $l=0$ holds. We assume (\ref{yl.bdd}) is true for $l=n$. Let $|\al|+|\be|\leq2$,
take the inner product of $\pa^\al\pa^\be \eqref{iterate.g}_1$ with $\frac{\pa^\al\pa^\be g^{n+1}}{\FM_\ast}$ over $\R\times\R^3$, to obtain
\begin{equation}\label{gn.inp}
\begin{split}
\frac{1}{2}&\frac{d}{dt}\left\|\frac{\pa^\al\pa^\be g^{n+1}}{\sqrt{\FM_\ast}}\right\|^2+\left(\nu_{\FM_r}\pa^\al\pa^\be g^{n+1},\frac{\pa^\al\pa^\be g^{n+1}}{\FM_\ast}\right)
\\&=-{\sum\limits_{|\al'|+|\be'|\leq |\al|+|\be|-1\atop{\al'\leq\al,\be'\leq\be}}C^{\al,\be}_{\al',\be'}}\left(\pa^{\al'}\pa^{\be'}\nu_{\FM_r}\pa^{\al-\al'}\pa^{\be-\be'}g^{n+1},\frac{\pa^\al\pa^\be g^{n+1}}{\FM_\ast}\right)\\
&-\sum\limits_{|\be'|=1}C^{\be}_{\be'}\left(\pa^{\be'}\xi_1\pa^\al\pa^{\be-\be'} g^{n+1},\frac{\pa^\al\pa^\be g^{n+1}}{\FM_\ast}\right)
+{\sum\limits_{0<\al'\leq \al}C^{\al}_{\al'}}\left(\pa^{\al'}\pa_x\widetilde{\phi}^{\,n}\pa^{\al-\al'}\pa^{\be} g^{n+1},\frac{\pa^\al\pa^\be g^{n+1}}{\FM_\ast}\right)
\\&+{\sum\limits_{\al'\leq\al,\be'\leq\be}C^{\al,\be}_{\al',\be'}}\left(Q_{\textrm{gain}}\left(\pa^{\al'}\pa^{\be'}\FM_r,\pa^{\al-\al'}\pa^{\be-\be'} g^{n}\right)-Q_{\textrm{loss}}\left(\pa^{\al-\al'}\pa^{\be-\be'} g^{n},\pa^{\al'}\pa^{\be'}\FM_r\right),\frac{\pa^\al\pa^\be g^{n+1}}{\FM_\ast}\right)
\\&+{\sum\limits_{\al'\leq\al,\be'\leq\be}C^{\al,\be}_{\al',\be'}}\left(Q_{\textrm{gain}}\left(\pa^{\al-\al'}\pa^{\be-\be'} g^{n},\pa^{\al'}\pa^{\be'}\FM_r\right),\frac{\pa^\al\pa^\be g^{n+1}}{\FM_\ast}\right)
\\&+\left(\pa^\al\pa^\be Q_{\textrm{gain}} (g^{n},g^{n}),\frac{\pa^\al\pa^\be g^{n+1}}{\FM_\ast}\right)-
\left(\pa^\al\pa^\be Q_{\textrm{loss}} (g^{n},g^{n+1}),\frac{\pa^\al\pa^\be g^{n+1}}{\FM_\ast}\right)
\\&-\left(\pa^\al\pa^\be \left(\left\{\pa_t+\xi_1\pa_x-\pa_x\widetilde{\phi}^{\,n}\pa_{\xi_1}\right\}\FM_r\right),\frac{\pa^\al\pa^\be g^{n+1}}{\FM_\ast}\right)
\\&+\left(\pa^\al\pa^\be \left(\pa_x\phi^r\pa_{\xi_1}g^{n+1}\right),\frac{\pa^\al\pa^\be g^{n+1}}{\FM_\ast}\right)
+\left(\pa^\al\pa^\be \left(\pa_x\phi^r\pa_{\xi_1}\FM_r\right),\frac{\pa^\al\pa^\be g^{n+1}}{\FM_\ast}\right).
\end{split}
\end{equation}
Now by integrating \eqref{gn.inp} with respect to the time variable over $[0,t]$ with $0\leq t\leq T$,
and performing the similar calculations as the proof of \eqref{2d.F.sum2},
one has
\begin{equation}\label{gn.eng.}
\begin{split}
&\widetilde{\mathcal {E}}_T (g^{n+1})+\la\widetilde{\mathcal {D}}_T (g^{n+1})\\
&\leq \sum\limits_{|\al|+|\beta|\leq 2}\int_{{\R}\times{\R}^3}\frac{\left|\partial_x\partial^\beta g^0(x,\xi)\right|^2}{{\bf M}_*}dxd\xi
+C(\eps^2_0+T)\widetilde{Y}_T (g^{n})
+C\sum\limits_{|\al|\leq2}\int_0^T\left\|\pa^{\al}\pa_x\widetilde{\phi}\right\|^2dt
+C\eps\ln ({e}+T).
\end{split}
\end{equation}
On the other hand, from $\eqref{iterate.g}_2$, it follows that
\begin{equation}\label{phi.n}
\begin{split}
\sum\limits_{|\al|\leq2}\left\|\pa^{\al}\pa_x\widetilde{\phi}^{\,n}\right\|^2\leq C\widetilde{\mathcal {E}}_T (g^{n})+C\eps(1+t)^{-2}.
\end{split}
\end{equation}
Consequently, \eqref{gn.eng.} and \eqref{phi.n} yield
\begin{equation*}
\begin{split}
\widetilde{\mathcal {E}}_T (g^{n+1})+\la\widetilde{\mathcal {D}}_T (g^{n+1})
\leq \sum\limits_{|\al|+|\beta|\leq 2}\int_{{\R}\times{\R}^3}\frac{\left|\partial^\al\partial^\beta g^0(x,\xi)\right|^2}{{\bf M}_*}dxd\xi
+C(\eps^2_0+T)\widetilde{Y}_T (g^{n})
+C\eps\ln ({e}+T).
\end{split}
\end{equation*}
This then implies \eqref{yl.bdd} for $l=n+1$, since $\eps>0$ can be  small enough and both $T^*>0$ and $\eps_0>0$ can be chosen to be  suitably small. The proof of Lemma \ref{gn.bdd.lem} is therefore complete.
\end{proof}

With the uniform bound on  the iterative solution sequence in terms of \eqref{iterate.g} by Lemma \ref{gn.bdd.lem}, we
can give the proof of  the local existence of solutions in the following lemma. We remark that the approach used here is due  to
Guo \cite{G}.

\begin{lemma}\label{local.existence}
For a sufficiently small $\eps_0>0$, there exists $T^{\ast}=T^{\ast}(\eps_0)>0$ such that if
$$
\sum\limits_{|\al|+|\beta|\leq 2}\int_{{\R}\times{\R}^3}\frac{\left|\partial^\al\partial^\beta g^0(x,\xi)\right|^2}{{\bf M}_*}dxd\xi  +\eps\leq{\eps^2_0},
$$
then there is a unique strong solution $F(t,x,\xi)$ to the VPB system (\ref{VPB})
in $(0,T^{\ast})\times {\R}\times {\R}^3$ with initial data $F(0,x,\xi)=F_0(x,\xi)$, such that
\begin{equation}\notag
\widetilde{Y}_T (F-\FM_r)\leq 2\eps^2_0,
\end{equation}
for any $T\in [0,T^\ast)$,
where $\widetilde{Y}_T (\cdot)$ is defined in \eqref{gn.bdd}.
Moreover if $
F_0(x,\xi)\geq0$, then
$F(t,x,\xi)\geq 0$ holds true for $T\in [0,T^\ast)$.
\end{lemma}
\begin{proof}
Recalling \eqref{gn.bdd} and \eqref{iterate.F}, the limit function $g(t,x,\xi)+\FM_r$ of the approximate solution sequence $\{g^n+\FM_r\}_{n=1}^\infty$ must be the solution to \eqref{VPB} with $F(0,x,\xi)=F_0(x,\xi)$ in the sense of distribution. The distribution solution turns out to be a strong solution because
it  can be shown to be unique as in \cite{G}. We omit the details for the proof of the uniqueness. The proof of the positivity of the solution is also quite standard, cf. \cite{G} for instance.
This ends the proof of Lemma \ref{local.existence}.
\end{proof}

\section{Global existence and large time behavior}

We are now in a position to complete the 

\begin{proof}[Proof of Theorem \ref{main.Res.}]
From the energy estimates obtained in Proposition \ref{g.eng.lem.}, one sees that
\begin{equation}\label{ap.es}
\begin{split}
\sup\limits_{0\leq t\leq T}&\sum\limits_{|\al|+|\be|\leq2}\left\|\pa^{\al}\pa^\be\left(F(t,x,\xi)-\FM_{[\rho^r,u^r,\ta^r](t,x)}(\xi)\right)
\right\|_{L_x^2\left(L^2_\xi\left(\frac{1}{\sqrt{\FM_*(\xi)}}\right)\right)}^2
\\ \lesssim&\sup\limits_{0\leq t\leq T}\sum\limits_{|\al|\leq2}
\left\|\pa^{\al}\left[\widetilde{\rho}, \widetilde{u}, \widetilde{\ta}\right](t)\right\|^2+
\sup\limits_{0\leq t\leq T} \sum\limits_{1\leq|\al|\leq
2} {\displaystyle\int_{\R\times\R^3}}\frac{\left|\partial^{\al} F(t,x,\xi)\right|^2}{{\bf M}_*}dxd\xi
\\&+\sup\limits_{0\leq t\leq T} {\displaystyle\int_{\R\times\R^3}}\frac{\left|\widetilde{\FG}(t,x,\xi)\right|^2}{{\bf M}_*}dxd\xi
+\sup\limits_{0\leq t\leq T} \sum\limits_{|\al|+|\be|\leq
2\atop{|\be|\geq1}} {\displaystyle\int_{\R\times\R^3}}\frac{\left|\pa^{\al}\partial^\be\widetilde{\FG}(t,x,\xi)\right|^2}{{\bf M}_*}dxd\xi
+\eps
\\ \leq& C_0N^2(0)+C_0\eps^{\si_0}.
\end{split}
\end{equation}
Notice that $\eps>0$ is a parameter independent of $\eps_0$. By letting $\eps>0$ be small enough, the global existence of the solution of the Cauchy problem \eqref{VPB}, \eqref{BE.Idata1}, \eqref{VPB.b1}, \eqref{con.phi}
 then follows from the standard continuation argument based on the local existence obtained in Lemma \ref{local.existence} and the a priori
estimate \eqref{ap.es}. In addition, \eqref{ap.es} implies \eqref{VPB.sol.}. It now remains to prove the large time behavior as \eqref{sol.Lab}.
For this, we start with the justification of the following two limits:
\begin{equation}\label{latm1}
\lim\limits_{t\rightarrow+\infty}\left\|\frac{\pa_x\left(F(t,x,\xi)-\FM_{[\rho^r,u^r,\ta^r](t,x)}(\xi)\right)}{\sqrt{\FM_\ast}}
\right\|_{L^2_{x,\xi}}^2= 0,
\end{equation}
and
\begin{equation}\label{latm2}
\lim\limits_{t\rightarrow+\infty}\left\|\pa_x\widetilde{\phi}(t)\right\|^2
= 0.
\end{equation}
Indeed, by the global existence, utilizing \eqref{g.eng.} and Lemma \ref{cl.Re.Re2.}, one can show that
\begin{equation}\label{latm3}
\begin{split}
\int_{0}^{+\infty}&\left|\frac{d}{dt}\left\|\frac{\pa_x\left(F(t,x,\xi)-\FM_{[\rho^r,u^r,\ta^r](t,x)}(\xi)\right)}{\sqrt{\FM_\ast}}
\right\|_{L^2_{x,\xi}}^2\right|dt
\\=&2\int_{0}^{+\infty}\left|\left(\FM_\ast^{-1}\pa_t\pa_x\left(F(t,x,\xi)-\FM_{[\rho^r,u^r,\ta^r](t,x)}(\xi)\right),
\pa_x\left(F(t,x,\xi)-\FM_{[\rho^r,u^r,\ta^r](t,x)}(\xi)\right)\right)\right|dt
\\ \leq &\int_{0}^{+\infty}\left\|\FM_\ast^{-1/2}\pa_t\pa_x\left(F(t,x,\xi)-\FM_{[\rho^r,u^r,\ta^r](t,x)}(\xi)\right)\right\|_{L^2_{x,\xi}}^2dt
\\&+\int_{0}^{+\infty}\left\|\FM_\ast^{-1/2}\pa_x\left(F(t,x,\xi)-\FM_{[\rho^r,u^r,\ta^r](t,x)}(\xi)\right)\right\|_{L^2_{x,\xi}}^2dt
\\ \leq &C\sum\limits_{1\leq|\al|\leq 2}\int_{0}^{+\infty}\left\|\pa^{\al}\left[\widetilde{\rho},\widetilde{u},\widetilde{\theta}\right]\right\|^2dt
+C\sum\limits_{1\leq|\al|\leq 2}\int_{0}^{+\infty}\left\|\frac{\pa^{\al}\FG}{\sqrt{\FM_\ast}}\right\|_{L^2_{x,\xi}}^2dt
<+\infty,
\end{split}
\end{equation}
and
\begin{equation}\label{latm4}
\int_{0}^{+\infty}\left|\frac{d}{dt}\left\|\pa_x\widetilde{\phi}\right\|^2\right|dt
=\frac{1}{2}\int_{0}^{+\infty}\left|\left(\pa_t\pa_x\widetilde{\phi},\pa_x\widetilde{\phi}\right)\right|dt<+\infty.
\end{equation}
Thus \eqref{latm3} and \eqref{latm4} and  give \eqref{latm1} and \eqref{latm2}.
With \eqref{latm1} and \eqref{latm2} in hand, we now get from Sobolev's inequality \eqref{sob.ine.} and \eqref{ap.es} that
\begin{equation}\label{latm5}
\begin{split}
\sup\limits_{x\in\R}&\left\|\frac{F(t,x,\xi)-\FM_{[\rho^r,u^r,\ta^r](t,x)}(\xi)}{\sqrt{\FM_\ast}}
\right\|_{L^2_{\xi}}^2\\
\leq &\left\|\frac{ \sup\limits_{x\in\R}|F(t,x,\xi)-\FM_{[\rho^r,u^r,\ta^r](t,x)}(\xi)|}{\sqrt{\FM_\ast}}
\right\|_{L^2_{\xi}}^2\\
\leq &\left\|\frac{ \sqrt{2}\|F(t,x,\xi)-\FM_{[\rho^r,u^r,\ta^r](t,x)}(\xi)\|_{L^2_x}^{1/2}\|\pa_x (F(t,x,\xi)-\FM_{[\rho^r,u^r,\ta^r](t,x)}(\xi))\|_{L^2_x}^{1/2}}{\sqrt{\FM_\ast}}
\right\|_{L^2_{\xi}}^2\\
\leq & 2 \left\|\frac{F(t,x,\xi)-\FM_{[\rho^r,u^r,\ta^r](t,x)}(\xi)}{\sqrt{\FM_\ast}}
\right\|_{L^2_{x,\xi}} \left\|\frac{\pa_x (F(t,x,\xi)-\FM_{[\rho^r,u^r,\ta^r](t,x)}(\xi))}{\sqrt{\FM_\ast}}
\right\|_{L^2_{x,\xi}}
\\ \leq& C\eps_0\left\|\frac{\pa_x\left(F(t,x,\xi)-\FM_{[\rho^r,u^r,\ta^r](t,x)}(\xi)\right)}{\sqrt{\FM_\ast}}
\right\|_{L^2_{x,\xi}}\rightarrow 0,
\end{split}
\end{equation}
as $t\to +\infty$. 
Similarly,
\begin{equation}\label{latm6}
\begin{split}
\sup\limits_{x\in\R}\left|\widetilde{\phi}\right|\leq \sqrt{2}\left\|\pa_x\widetilde{\phi}\right\|^{1/2}\left\|\widetilde{\phi}\right\|^{1/2}\leq C\eps_0 \left\|\pa_x\widetilde{\phi}\right\|^{1/2}\rightarrow 0\ \  \textrm{as}\ t\rightarrow+\infty.
\end{split}
\end{equation}
Then \eqref{sol.Lab} follows from \eqref{latm5}, \eqref{latm6}
and $(iii)$ in Lemma \ref{cl.Re.Re2.}.
This completes the proof of Theorem \ref{main.Res.}.
\end{proof}

\section{Appendix}

In this appendix, we first list some basic inequalities used in the paper. The following two lemmas,  borrowed from \cite{GPS}, are concerned with estimates on the nonlinear
and linearized  collision operators  $Q(\cdot, \cdot)$ and $L_{\bf M}{\bf
G}$, respectively.

\begin{lemma}\label{est.nonop}
There exists a positive constant $C>0$ such
that
\begin{equation}\label{est.nonop.ine.}
{\displaystyle\int_{{\R}^3}}\frac{(1+|\xi|)^{-1}
Q(f,g)^2}{\widehat{\bf M}}d\xi
 \leq C\Bigg\{{\displaystyle\int_{{\R}^3}}
\frac{(1+|\xi|)f^2}{\widehat{\bf M}} d\xi\cdot
{\displaystyle\int_{{\R}^3}}\frac{g^2}{\widehat{\bf M}} d\xi
+{\displaystyle\int_{{\R}^3}}\frac{f^2}{\widehat{\bf M}}
d\xi\cdot{\displaystyle\int_{{\R}^3}}\frac{(1+|\xi|)g^2}{\widehat{\bf M}} d\xi\Bigg\},
\end{equation}
where $\widehat{\bf M}$ is any Maxwellian such that the above
integrals are well defined.
\end{lemma}
\begin{remark}
In fact both of $Q_{\textrm{gain}}(f,g)$ and $Q_{\textrm{loss}}(f,g)$ enjoy the estimate \eqref{est.nonop.ine.}, and they will be used to compute \eqref{gn.inp}.
\end{remark}

To perform the energy estimates for the
Boltzmann equation, 
${\bf P}_1^{\widehat{{\bf M}}}F$, the
microscopic projection of its solution $F(t,x,\xi)$ with respect
to a given Maxwellian $\widehat{{\bf M}}$, the dissipative effect through
the microscopic $H$-theorem should be used. In short, the
microscopic $H$-theorem states that the linearized collision
operator $L_{\widehat{\FM}}$ around a fixed global Maxwellian state $\widehat{\FM}$ is negative definite on the non-fluid element ${\bf
P}_1^{\widehat{\FM}}F$, \cite{Car}, i.e., the coercivity property
$$
 -\int_{{\R}^3}\frac{{{\bf P}_1^{\widehat{\FM}}F}L_{\widehat{\FM}}\left({\bf
P}_1^{\widehat{\FM}}F\right)}{\widehat{\FM}}d\xi \geq \de\int_{{\R}^3}\frac{(1+|\xi|)\left|{\bf P}_1^{\widehat{\FM}}F\right|^2}{\widehat{\FM}}d\xi
$$
holds true for some positive constant $\de>0$.
Furthermore, one can vary the background for linearization and the weight function.
That is, we also have the following result whose proof is based on Lemma
\ref{est.nonop},  cf. \cite{LYYZ}.

\begin{lemma}\label{co.est.} If $\frac \theta 2<\widehat{\theta}$, then there exist two positive constants
$\de=\de(\rho, u,\theta;\widehat{\rho},\widehat{u},\widehat{\theta})$
and $\eta_0=\eta_0(\rho, u,\theta;\widehat{\rho},\widehat{u},\widehat{\theta})$ such that
if $|\rho-\widehat{\rho}|+|u-\widehat{u}|+|\theta-\widehat{\theta}|<\eta_0$, we have for
$h(\xi)\in {\mathcal{N}}^\bot$,
\begin{eqnarray*}
-\int_{{\R}^3}\frac{ h L_{\bf M} h}{\widehat{\bf M}}d\xi \geq
\de\int_{{\R}^3}\frac{(1+|\xi|)h^2}{\widehat{\bf
M}}d\xi.
\end{eqnarray*}
Here ${\bf M}\equiv {\bf M}_{[\rho,u,\theta]}(\xi)$, $ \widehat{\bf
M}= {\bf M}_{[\widehat{\rho},\widehat{u},\widehat{\theta}]}(\xi)$
and
$$
{\mathcal{N}}^\bot=\left\{ f(\xi):\ \ \int_{{\R}^3}\psi_i(\xi)f(\xi)d\xi=0,\ i=0,1,2,3,4\right\}.
$$
\end{lemma}

\begin{remark}
The constant $\eta_0$ in Lemma \ref{co.est.} is some positive constant depending on
the first non-zero eigenvalue of the linearized operator $L_{\bf
M}$. Note that $\eta_0$ is not necessarily small, cf.
\cite{LYYZ}.
\end{remark}

A direct consequence of Lemma \ref{co.est.} and the Cauchy inequality
is the following corollary, cf. \cite{LYYZ}.

\begin{corollary}\label{inv.L.} Under the assumptions in Lemma \ref{co.est.}, we have
for  $h(\xi)\in {\mathcal{N}}^\bot$,
\begin{eqnarray*}
{\displaystyle\int_{{\R}^3}}\frac{(1+|\xi|)}{\widehat{\bf
M}}\left|L^{-1}_{\bf M}h\right|^2d\xi\leq
\de^{-2}{\displaystyle\int_{{\R}^3}}\frac{(1+|\xi|)^{-1}h^2(\xi)}{\widehat{\bf M}}d\xi.
\end{eqnarray*}
\end{corollary}

In the rest part of the appendix, we undertake to give the detailed proofs of \eqref{Bur.fun.}, \eqref{ID.VPB2}, \eqref{CJ8.c}, \eqref{2d.F.sum2} 
and \eqref{mi.diss4} {one by one}.

\begin{proof}[Proof of \eqref{Bur.fun.}]
To verify \eqref{Bur.fun.}, it suffices to compute
\begin{eqnarray*}
\CH_l&=&\pa_x\int_{\R^3}\xi_l\xi_1 L^{-1}_{\FM}\left(\FP_1^{\FM}\left(\xi_1\pa_x\FM\right)\right)\,d\xi,\ l=1,2,3,\\
\CH_4&=&\frac{1}{2}\pa_x\int_{\R^3}|\xi|^2\xi_1 L^{-1}_{\FM}\left(\FP_1^{\FM}\left(\xi_1\pa_x\FM\right)\right)\,d\xi.
\end{eqnarray*}
For this, let us first consider $\CH_l$ $(l=1,2,3)$. By direct calculations, one has
\begin{equation*}
\begin{split}
\CH_l=&\pa_x\int_{\R^3}\xi_l\xi_1 L^{-1}_{\FM}\FP_1^{\FM}\left(\xi_1\FM\left[\sum\limits_{j=1}^3\frac{\xi_j-u_j}{R\ta}\pa_xu_j
+\frac{|\xi-u|^2}{2R\ta^2}\pa_x\ta\right]\right)\,d\xi\\
=&\pa_x\int_{\R^3}(\xi_l-u_l)(\xi_1-u_1) L^{-1}_{\FM}\FP_1^{\FM}\left((\xi_1-u_1)\FM\left[\sum\limits_{j=1}^3\frac{\xi_j-u_j}{R\ta}\pa_xu_j
+\frac{|\xi-u|^2}{2R\ta^2}\pa_x\ta\right]\right)\,d\xi
\\=&\pa_x\left(\int_{\R^3}(\xi_l-u_l)(\xi_1-u_1) L^{-1}_{\FM}\FP_1^{\FM}\left(\frac{(\xi_l-u_l)(\xi_1-u_1)}{R\ta}\FM\right)\,d\xi~\pa_xu_l\right)
\\=&\pa_x\left(\frac{1}{R\ta}\int_{\R^3}\xi_l\xi_1 L^{-1}_{\FM}\FP_1^{\FM}\left(\xi_l\xi_1\FM\right)\,d\xi~\pa_xu_l\right).
\end{split}
\end{equation*}
On the other hand, one can claim that
\begin{equation}\label{Lij}
\begin{split}
\int_{\R^3}\xi^2_i L^{-1}_{\FM}\FP_1^{\FM}\left(\xi^2_i\FM\right)\,d\xi
=3\int_{\R^3}\xi_i\xi_j L^{-1}_{\FM}\FP_1^{\FM}\left(\xi_i\xi_j\FM\right)\,d\xi, \ \textrm{for} \ i\neq j.
\end{split}
\end{equation}
To prove \eqref{Lij}, we get from the rotational invariance of $L_{\FM}^{-1}$ and integration that for $i\neq j$,
\begin{equation*}
\begin{split}
\int_{\R^3}\xi_i\xi_j L^{-1}_{\FM}\FP_1^{\FM}\left(\xi_i\xi_j\FM\right)\,d\xi
=&\int_{\R^3}(\xi_i-u_i)(\xi_j-u_j) L^{-1}_{\FM}\FP_1^{\FM}\left((\xi_i-u_i)(\xi_j-u_j)\FM\right)\,d\xi
\\=&\int_{\R^3}(\xi_i-u_i)^2 L^{-1}_{\FM}\FP_1^{\FM}\left((\xi_j-u_j)^2\FM\right)\,d\xi=\int_{\R^3}\xi_i^2 L^{-1}_{\FM}\FP_1^{\FM}\left(\xi_j^2\FM\right)\,d\xi,
\end{split}
\end{equation*}
and
\begin{equation}\label{Lij.p1}
\begin{split}
\int_{\R^3}\xi_i^2 L^{-1}_{\FM}\FP_1^{\FM}\left(\xi_i^2\FM\right)\,d\xi
=&\int_{\R^3}\left(\frac{\xi_i+\xi_j}{\sqrt{2}}\right)^2 L^{-1}_{\FM}\FP_1^{\FM}\left(\left(\frac{\xi_i+\xi_j}{\sqrt{2}}\right)^2\FM\right)\,d\xi
\\=&\int_{\R^3}\left(\frac{\xi^2_i+\xi^2_j}{2}+\xi_i\xi_j\right)
L^{-1}_{\FM}\FP_1^{\FM}\left(\left(\frac{\xi^2_i+\xi^2_j}{2}+\xi_i\xi_j\right)\FM\right)\,d\xi
\\=&\frac{1}{2}\int_{\R^3}\xi_i^2 L^{-1}_{\FM}\FP_1^{\FM}\left(\xi_i^2\FM\right)\,d\xi+\frac{3}{2}\int_{\R^3}\xi_i\xi_j L^{-1}_{\FM}\FP_1^{\FM}\left(\xi_i\xi_j\FM\right)\,d\xi.
\end{split}
\end{equation}
Then \eqref{Lij} follows from \eqref{Lij.p1}. We next define
$$
\mu(\ta)=-\frac{3}{2\theta}\int_{{\R}^3}\xi_1\xi_iL^{-1}_{\FM}
\left(\xi_1\xi_i\FM\right)>0,\
i=2, 3,
$$
which further equals to
$$
-\frac{3}{2\theta}\int_{{\R}^3}\xi_1\xi_iL^{-1}_{\FM_{[1,u,\ta]}}
\left(\xi_1\xi_i\FM_{[1,u,\ta]}\right)\,d\xi,\quad  i=2, 3,
$$
see \cite{Sone}, for instance.
Therefore the first formula in $\eqref{Bur.fun.}$ holds.

Similarly, for $\CH_4$, we have
\begin{equation*}
\begin{split}
\CH_4=&\frac{1}{2}\pa_x\int_{\R^3}|\xi|^2\xi_1 L^{-1}_{\FM}\FP_1^{\FM}\left(\xi_1\FM\left[\sum\limits_{j=1}^3\frac{\xi_j-u_j}{R\ta}\pa_xu_j
+\frac{|\xi-u|^2}{2R\ta^2}\pa_x\ta\right]\right)\,d\xi\\
=&\frac{1}{2}\pa_x\int_{\R^3}|\xi-u|^2(\xi_1-u_1) L^{-1}_{\FM}\FP_1^{\FM}\left((\xi_1-u_1)\FM\left[\sum\limits_{j=1}^3\frac{\xi_j-u_j}{R\ta}\pa_xu_j
+\frac{|\xi-u|^2}{2R\ta^2}\pa_x\ta\right]\right)\,d\xi
\\&+\pa_x\int_{\R^3}\xi\cdot u(\xi_1-u_1) L^{-1}_{\FM}\FP_1^{\FM}\left((\xi_1-u_1)\FM\left[\sum\limits_{j=1}^3\frac{\xi_j-u_j}{R\ta}\pa_xu_j
+\frac{|\xi-u|^2}{2R\ta^2}\pa_x\ta\right]\right)\,d\xi
\\
=&\frac{1}{2}\pa_x\int_{\R^3}|\xi-u|^2(\xi_1-u_1) L^{-1}_{\FM}\FP_1^{\FM}\left((\xi_1-u_1)\FM\left[\frac{|\xi-u|^2}{2R\ta^2}\pa_x\ta\right]\right)\,d\xi
\\&+\pa_x\int_{\R^3}\xi\cdot u(\xi_1-u_1) L^{-1}_{\FM}\FP_1^{\FM}\left((\xi_1-u_1)\FM\left[\sum\limits_{j=1}^3\frac{\xi_j-u_j}{R\ta}\pa_xu_j
\right]\right)\,d\xi
\\=&\frac{1}{4R}\pa_x\left(\frac{1}{\ta^2}\int_{\R^3}|\xi-u|^2\xi_1 L^{-1}_{\FM_{[1,u,\ta]}}\left(|\xi-u|^2\xi_1 \FM_{[1,u,\ta]}\right)\,d\xi~\pa_x\ta\right)
\\&+\sum\limits_{j=1}^3\pa_x\left(\frac{1}{R\ta}u_j\int_{\R^3}\xi_j\xi_1 L^{-1}_{\FM_{[1,u,\ta]}}\FP_1^{\FM}\left(\xi_j\xi_1\FM_{[1,u,\ta]}\right)\,d\xi~\pa_xu_j\right),
\end{split}
\end{equation*}
which can be reduced to
$-\pa_x\left(\kappa(\theta)\pa_x \theta\right)-3\pa_x\left(\mu(\theta)u_1\pa_x u_1\right)-
\sum\limits_{i=2}^3\pa_x\left(\mu(\theta)u_i\pa_x u_i\right),$
by defining
$$
\ka(\ta)=-\frac{3}{8\ta^2}\int_{{\R}^3}|\xi-u|^2\xi_iL^{-1}_{\FM_{[1,u,\ta]}}
\left(|\xi-u|^2\xi_i\FM_{[1,u,\ta]}\right)\,d\xi>0,\
i=1, 2, 3.
$$
Thus the second formula in $\eqref{Bur.fun.}$ follows, and this completes the proof of \eqref{Bur.fun.}.
\end{proof}

\begin{proof}[Proof of \eqref{ID.VPB2}]
For brevity we set $\FM_0=\FM_{[\rho_0(x),u_0(x),\ta_0(x)]}(\xi)$ and $\FM_{r0}=\FM_{[\rho^r_0(x),u^r_0(x),\ta^r_0(x)]}(\xi)$ with $[\rho^r_0,u^r_0,\ta^r_0](x)=[\rho^r,u^r,\ta^r](0,x)$.
First of all, we show that the first norm on the left hand side of \eqref{ID.VPB2} is bounded by $C\eps_0^2$ for a constant $C$. Notice that the macro-micro decomposition of $F_0-\FM_{r0}$ with respect to the global Maxwellian $\FM_\ast$:
\begin{equation*}
F_0-\FM_{r0}=\FP_0^{\FM_\ast} (\FM_0-\FM_{r0})+\FP_1^{\FM_\ast} (F_0-\FM_{r0}),
\end{equation*}
implies 
\begin{equation*}
\left\|\pa^\al \FP_0^{\FM_\ast} (\FM_0-\FM_{r0})\right\|_{L^2_\xi\left(\frac{1}{\sqrt{\FM_*(\xi)}}\right)}\leq \left\|\pa^\al (F_0-\FM_{r0})\right\|_{L^2_\xi\left(\frac{1}{\sqrt{\FM_*(\xi)}}\right)},
\end{equation*}
for each $\al$ with $|\al|\leq 2$. The further integration in $x$ and using \eqref{ID.VPB-E} lead to
\begin{equation*}
\left\|\pa^\al \FP_0^{\FM_\ast} (\FM_0-\FM_{r0})\right\|_{L_x^2\left(L^2_\xi\left(\frac{1}{\sqrt{\FM_*(\xi)}}\right)\right)}\leq \|\pa^\al (F_0-\FM_{r0})\|_{L_x^2\left(L^2_\xi\left(\frac{1}{\sqrt{\FM_*(\xi)}}\right)\right)}\leq C\eps_0.
\end{equation*}
On the other hand, from direct computations,
\begin{equation*}
\pa^\al \FP_0^{\FM_\ast} (\FM_0-\FM_{r0})=\sum_{j=0}^4\pa^\al \left\langle \FM_0-\FM_{r0},\chi^{\FM_\ast}_j\right\rangle_{{\FM_\ast}} \chi_j^{\FM_\ast},
\end{equation*}
with the inner product terms given by
\begin{equation*}
\left\langle \FM_0-\FM_{r0},\chi^{\FM_\ast}_0\right\rangle_{{\FM_\ast}} =\rho_0-\rho_0^r,
\end{equation*}
\begin{equation*}
\left\langle \FM_0-\FM_{r0},\chi^{\FM_\ast}_j\right\rangle_{{\FM_\ast}} =\frac{(\rho_0u_{0j}-\rho_0^r u_{0j}^r)-(\rho_0-\rho_0^r) u_{\ast j}}{\sqrt{\frac{2}{3} \rho_\ast \theta_\ast}},\quad j=1,2,3,
\end{equation*}
and
\begin{eqnarray*}
\left\langle \FM_0-\FM_{r0},\chi^{\FM_\ast}_4\right\rangle_{{\FM_\ast}} =\frac{3(\rho_0\theta_0-\rho_0^r\theta_0^r)}{\theta_\ast \sqrt{6\rho_\ast}} -\frac{3(\rho_0-\rho_0^r)}{\sqrt{6\rho_\ast}}
+\frac{3\rho_0|u_0-u_\ast|^2-3\rho_0^r|u_{0r}-u_\ast|^2}{2\theta_\ast \sqrt{6\rho_\ast}}.
\end{eqnarray*}
Since $\left\{\chi_j^{\FM_\ast},\ 0\leq j\leq 4\right\}$ is an orthonormal set of $L^2_\xi\left(\frac{1}{\sqrt{\FM_*(\xi)}}\right)$,
\begin{equation*}
\max_{0\leq j\leq 4} \left|\pa^\al \left\langle \FM_0-\FM_{r0},\chi^{\FM_\ast}_j\right\rangle_{{\FM_\ast}}\right|\leq \left\|\pa^\al \FP_0^{\FM_\ast} (\FM_0-\FM_{r0})\right\|_{L^2_\xi\left(\frac{1}{\sqrt{\FM_*(\xi)}}\right)}.
\end{equation*}
Therefore, by integrating it in $x$, taking the proper linear combination and using smallness of $\eps_0$, one has
\begin{equation}
\label{rja-p1}
\sum_{|\al|\leq 2} \|\pa^\al [\rho_0-\rho_0^r,u_0-u_0^r,\theta_0-\theta_0^r]\|\leq C\eps_0.
\end{equation}
Finally, the second norm on the left side of  \eqref{ID.VPB2} is bounded by $C\eps_0^2$ due to the mean-value property as well as \eqref{rja-p1} and smallness of $\eps_0$, and further the estimate on the third norm immediately follows by noticing $\FG_0=F_0-\FM_0=(F_0-\FM_{r0})-(\FM_0-\FM_{r0})$, and
\begin{equation*}
\sum\limits_{|\al|+|\be|\leq2}\left\|\pa^{\al}\pa^\be\left(\FM_0
-\FM_{r0}\right)\right\|^2_{L_x^2\left(L^2_\xi\left(\frac{1}{\sqrt{\FM_*(\xi)}}\right)\right)}\leq C\eps_0^2.
\end{equation*}
Then \eqref{ID.VPB2} is proved.
\end{proof}

\begin{proof}[Proof of \eqref{CJ8.c}]
Notice that
\begin{equation*}
\label{ }
\CJ_8=\left(\pa^\al \pa_x\phi,(\xi_1-u_1)\pa^\al \FM\right),\quad 1\leq |\al|\leq 2.
\end{equation*}
Let $\pa_i\in\{\pa_t,\pa_x\}$ and $\pa_j\in\{\pa_t,\pa_x\}$. By a simple calculation, 
\begin{eqnarray*}
\pa_i\FM =\frac{\pa_i \rho}{\rho}\FM +\frac{\xi-u}{R\theta}\cdot \pa_i u\FM +\left(\frac{|\xi-u|^2}{2R\theta}-{\frac{3}{2}}\right)\frac{\pa_i \theta}{\theta}\FM.
\end{eqnarray*}
Then, one can see that for $|\al|=1$, 
\begin{equation*}
\CJ_{8}=\int_{\R}\pa^{\al}\pa_x\phi \rho\pa^{\al} u_1dx.
\end{equation*}
Furthermore, for $|\al|=2$ with $\pa^\al=\pa_j\pa_i$, one can also obtain from direct calculations that 
\begin{equation*}
\CJ_{8}=\int_{\R}\pa^{\al}\pa_x\phi \rho\pa^{\al} u_1\,dx
+\underbrace{\int_{\R}\pa^{\al}\pa_x\phi (\pa_i\rho\pa_j u_1+\pa_j\rho\pa_i u_1)\,dx}_{\CJ_{8,0}}.
\end{equation*}
It is easy to see that
\begin{equation*}
|\CJ_{8,0}|\lesssim (\eps_0+\eta)\sum\limits_{|\al|=2}\left\|\pa^{\al}\pa_x\widetilde{\phi}\right\|^2
+(\eps_0+\eta)\sum\limits_{|\al|=1}\left\|\pa^{\al}\left[\widetilde{\rh},\widetilde{u}_1\right]\right\|^2+C_\eta\eps(1+t)^{-2}.
\end{equation*}
Therefore, to compute $\sum\limits_{1\leq|\al|\leq 2}\CJ_{8}$, it suffices to calculate
\begin{equation*}
\begin{split}
\sum\limits_{1\leq|\al|\leq 2}\int_{\R}\pa^{\al}\pa_x\phi \rho\pa^{\al} u_1dx
=&\underbrace{\sum\limits_{1\leq|\al|\leq 2}\int_{\R}\pa^{\al}\pa_x\widetilde{\phi} \rho\pa^{\al} \widetilde{u}_1dx}_{\CJ_{8,1}}
+\underbrace{\sum\limits_{1\leq|\al|\leq 2}\int_{\R}\pa^{\al}\pa_x\phi^r \rho\pa^{\al} \widetilde{u}_1dx}_{\CJ_{8,2}}\\&+\underbrace{\sum\limits_{1\leq|\al|\leq 2}\int_{\R}\pa^{\al}\pa_x\widetilde{\phi} \rho\pa^{\al} u^r_1dx}_{\CJ_{8,3}}+
\underbrace{\sum\limits_{1\leq|\al|\leq 2}\int_{\R}\pa^{\al}\pa_x\phi^r \rho\pa^{\al} u^r_1dx}_{\CJ_{8,4}}.
\end{split}
\end{equation*}
For this, we now turn to estimate $\CJ_{8,l}$ $(1\leq l\leq4)$ term by term. We first have by using \eqref{trho0} and integrating by parts that
\begin{equation}\label{CJ8p1}
\begin{split}
\CJ_{8,1}=&\sum\limits_{1\leq|\al|\leq2}\left(\pa^{\al}\widetilde{\phi},\pa^{\al}\pa_t\widetilde{\rho}\right)
+\sum\limits_{1\leq|\al|\leq2,\atop{0<\al'\leq\al}}C_{\al'}^{\al}\left(\pa^{\al}\widetilde{\phi},\pa_x\left(\pa^{\al'}\rho\pa^{\al-\al'}\widetilde{u}_1\right)\right)
\\&+\sum\limits_{1\leq|\al|\leq2,\atop{0\leq\al'\leq\al}}C_{\al'}^{\al}\left(\pa^{\al}\widetilde{\phi},\pa_x\pa^{\al'}\widetilde{\rho}\pa^{\al-\al'}u^r_1\right)
+\sum\limits_{1\leq|\al|\leq2,\atop{0\leq\al'\leq\al}}C_{\al'}^{\al}\left(\pa^{\al}\widetilde{\phi},\pa^{\al'}\widetilde{\rho}\pa_x\pa^{\al-\al'}u^r_1)\right)
\\=&\sum\limits_{1\leq|\al|\leq2}\left(\pa^{\al}\widetilde{\phi},\pa^{\al}\pa_t\widetilde{\rho}\right)
-\sum\limits_{1\leq|\al|\leq2,\atop{0<\al'\leq\al}}C_{\al'}^{\al}\left(\pa_x \pa^{\al}\widetilde{\phi},\pa^{\al'}\rho\pa^{\al-\al'}\widetilde{u}_1\right)
\\&-\sum\limits_{1\leq|\al|\leq2,\atop{0\leq\al'\leq\al}}C_{\al'}^{\al}\left(\pa_x\pa^{\al}\widetilde{\phi},\pa^{\al'}\widetilde{\rho}\pa^{\al-\al'}u^r_1\right)
+\sum\limits_{1\leq|\al|\leq2,\atop{0\leq\al'\leq\al}}C_{\al'}^{\al}\left(\pa^{\al}\widetilde{\phi},\pa^{\al'}\widetilde{\rho}\pa_x\pa^{\al-\al'}u^r_1)\right).
\end{split}
\end{equation}
In light of \eqref{trho0} and \eqref{taylor2}, the first term on the right hand side of \eqref{CJ8p1} further equals to
\begin{equation*}
\begin{split}
-&\sum\limits_{1\leq|\al|\leq2}\left(\pa^{\al}\widetilde{\phi},\pa^{\al}\pa_t\pa^2_x\widetilde{\phi}\right)
-\sum\limits_{1\leq|\al|\leq2}\left(\pa^{\al}\widetilde{\phi},\pa^{\al}\pa_t\left(\rho_{e}(\phi^r)-\rho_{e}(\phi)\right)\right)
-\sum\limits_{1\leq|\al|\leq2}\left(\pa^{\al}\widetilde{\phi},\pa^{\al}\pa_t\pa^2_x\phi^r\right)
\\=&\frac{1}{2}\frac{d}{dt}\sum\limits_{1\leq|\al|\leq2}\left\{\left\|\pa^{\al}\pa_x\widetilde{\phi}\right\|^2
+\left\|\sqrt{\rho_{e}'(\phi^r)}\pa^{\al}\widetilde{\phi}\right\|^2\right\}
\underbrace{-\sum\limits_{1\leq|\al|\leq2}\left(\pa^{\al}\widetilde{\phi},\pa^{\al}\pa_tJ_{10}\right)}_{\CK_1}
\\&+\underbrace{\sum\limits_{1\leq|\al|\leq2,\atop{0<\al'\leq\al}}C_{\al'}^{\al}
\left(\pa^{\al}\widetilde{\phi},\pa^{\al'}\rho_{e}'(\phi^r)\pa_t\pa^{\al-\al'}\widetilde{\phi}\right)
+\sum\limits_{1\leq|\al|\leq2,\atop{0\leq\al'\leq\al}}C_{\al'}^{\al}
\left(\pa^{\al}\widetilde{\phi},\pa^{\al'}\pa_t\rho_{e}'(\phi^r)\pa^{\al-\al'}\widetilde{\phi}\right)}_{\CK_2}
\\&\underbrace{-\frac{1}{2}\left(\pa^{\al}\widetilde{\phi},\pa_t\rho_{e}'(\phi^r)\pa^{\al}\widetilde{\phi}\right)
-\sum\limits_{1\leq|\al|\leq2}\left(\pa^{\al}\widetilde{\phi},\pa^{\al}\pa_t\pa^2_x\phi^r\right)}_{\CK_3}.
\end{split}
\end{equation*}
Here $\CK_1$, $\CK_2$ and $\CK_3$ will be estimated as follows.
From \eqref{xJ10}, it follows that
\begin{equation*}
\begin{split}
\CK_1=&-\sum\limits_{1\leq|\al|\leq2}\left(\pa^\al\left(\pa_t\phi\int_{\phi^r}^{\phi}\rho_{e}''(\varrho) d\varrho
+\widetilde{\phi}\pa_t\phi^r\int_{\phi^r}^{\phi}\rho_{e}''(\varrho) d\varrho\right),\pa^{\al}\widetilde{\phi}\right)
\\=&-\frac{1}{2}\frac{d}{dt}\sum\limits_{1\leq|\al|\leq2}\left(\pa^\al\widetilde{\phi}\int_{\phi^r}^{\phi}\rho_{e}''(\varrho) d\varrho
,\pa^{\al}\widetilde{\phi}\right)+\frac{1}{2}\sum\limits_{1\leq|\al|\leq2}\left(\pa^\al\widetilde{\phi}\pa_t\left(\int_{\phi^r}^{\phi}\rho_{e}''(\varrho) d\varrho\right)
,\pa^{\al}\widetilde{\phi}\right)
\\&-\sum\limits_{1\leq|\al|\leq2}\left(\pa^\al\pa_t\phi^r\int_{\phi^r}^{\phi}\rho_{e}''(\varrho) d\varrho
,\pa^{\al}\widetilde{\phi}\right)-\sum\limits_{1\leq|\al|\leq2,\atop{0<\al'\leq\al}}C_{\al'}^{\al}
\left(\pa^{\al-\al'}\pa_t\phi\pa^{\al'}\left(\int_{\phi^r}^{\phi}\rho_{e}''(\varrho) d\varrho\right)
,\pa^{\al}\widetilde{\phi}\right)
\\&-\sum\limits_{1\leq|\al|\leq2}\left(\pa^\al\left(\widetilde{\phi}\pa_t\phi^r\int_{\phi^r}^{\phi}\rho_{e}''(\varrho) d\varrho\right),\pa^{\al}\widetilde{\phi}\right).
\end{split}
\end{equation*}
We thus find by utilizing Cauchy-Schwarz's inequality with $0<\eta<1$ and Sobolev's inequality \eqref{sob.ine.}, the a priori assumption \eqref{aps}, as well as Lemma \ref{cl.Re.Re2.} that
\begin{equation*}
\begin{split}
&\left|\CK_1+\frac{1}{2}\frac{d}{dt}\sum\limits_{1\leq|\al|\leq2}\left(\pa^\al\widetilde{\phi}\int_{\phi^r}^{\phi}\rho_{e}''(\varrho) d\varrho
,\pa^{\al}\widetilde{\phi}\right)\right|
\\& \qquad  \lesssim
(\eps_0+\eta)\sum\limits_{1\leq|\al|\leq2}\left\|\pa^\al\widetilde{\phi}\right\|^2+{C_\eta}(1+t)^{-2}\left\|\widetilde{\phi}\right\|^2
+{C_\eta}\eps(1+t)^{-2}.
\end{split}
\end{equation*}
Similarly, it follows that
\begin{equation*}
\begin{split}
|\CK_2|+|\CK_3|
\lesssim
(\eps_0+\eta)\sum\limits_{1\leq|\al|\leq2}\left\|\pa^\al\widetilde{\phi}\right\|_{H^1}^2+{C_\eta}(1+t)^{-2}\left\|\widetilde{\phi}\right\|^2
+{C_\eta}\eps(1+t)^{-2}.
\end{split}
\end{equation*}
The second term on the right hand side of \eqref{CJ8p1} can be bounded by
$$
(\eps_0+\eta)\sum\limits_{1\leq|\al|\leq2}\left\|\pa_x\pa^\al\widetilde{\phi}\right\|^2
+\eps_0\sum\limits_{1\leq|\al|\leq2}\left\|\pa^\al\widetilde{\rho}\right\|^2
+\eps_0\sum\limits_{|\al|=1}\left\|\pa^\al\widetilde{u}_1\right\|^2+{C_\eta}(1+t)^{-2}\left\|\widetilde{u}_1\right\|^2.
$$
As to the third term on the right hand side of \eqref{CJ8p1},
when $|\al-\al'|\geq1$, it is bounded by
$$
(\eta+\eps)\sum\limits_{1\leq|\al|\leq2}\left\|\pa_x\pa^\al\widetilde{\phi}\right\|^2
+\eps_0\sum\limits_{|\al|=1}\left\|\pa^\al\widetilde{\rho}\right\|^2
+{C_\eta}(1+t)^{-2}\left\|\widetilde{\rho}\right\|^2.
$$
If $|\al-\al'|=0$, using \eqref{tphy} again, it reads
\begin{equation*}
\begin{split}
\sum\limits_{1\leq|\al|\leq2}\left(\pa_x\pa^{\al}\widetilde{\phi},\pa^{\al}\pa_x^2\widetilde{\phi}u^r_1\right)
+\sum\limits_{1\leq|\al|\leq2}\left(\pa_x\pa^{\al}\widetilde{\phi},\pa^{\al}\left(\rho_{e}(\phi^r)-\rho_{e}(\phi)\right)u^r_1\right)
+\sum\limits_{1\leq|\al|\leq2}\left(\pa_x\pa^{\al}\widetilde{\phi},\pa^{\al}\pa^2_x\phi^ru^r_1\right),
\end{split}
\end{equation*}
which is further dominated by
$$
(\eta+\eps_0)\sum\limits_{1\leq|\al|\leq2}\left\|\pa^\al\widetilde{\phi}\right\|_{H^1}^2+{C_\eta}(1+t)^{-2}\left\|\widetilde{\phi}\right\|^2
+{C_\eta}\eps(1+t)^{-2}.
$$
The last term on the right hand side of \eqref{CJ8p1} is controlled by
$$
(\eta+\eps_0)\sum\limits_{1\leq|\al|\leq2}\left\|\pa^\al\widetilde{\phi}\right\|^2
+\eps_0\sum\limits_{1\leq|\al|\leq2}\left\|\pa^\al\widetilde{\rho}\right\|^2+{C_\eta}(1+t)^{-2}\left\|\widetilde{\rho}\right\|^2.
$$
We now conclude from the above estimates that
\begin{equation}\label{CJ81s}
\begin{split}
&\left|\CJ_{8,1}-\frac{1}{2}\frac{d}{dt}\sum\limits_{1\leq|\al|\leq2}\left\{\left\|\pa^{\al}\pa_x\widetilde{\phi}\right\|^2
+\left\|\sqrt{\rho_{e}'(\phi^r)}\pa^{\al}\widetilde{\phi}\right\|^2\right\}
+\frac{1}{2}\frac{d}{dt}\sum\limits_{1\leq|\al|\leq2}\left(\pa^\al\widetilde{\phi}\int_{\phi^r}^{\phi}\rho_{e}''(\varrho) d\varrho
,\pa^{\al}\widetilde{\phi}\right)\right|
\\& \qquad  \lesssim
(\eps_0+\eta)\sum\limits_{1\leq|\al|\leq2}\left\|\pa^\al\widetilde{\phi}\right\|_{H^1}^2
+\eps_0\sum\limits_{1\leq|\al|\leq2}\left\|\pa^\al\left[\widetilde{\rho},\widetilde{u}_1\right]\right\|^2
+{C_\eta}(1+t)^{-2}\left\|\left[\widetilde{\rho},\widetilde{u}_1,\widetilde{\phi}\right]\right\|^2
+{C_\eta}\eps(1+t)^{-2}.
\end{split}
\end{equation}
It remains now to estimate $\CJ_{8,2}$, $\CJ_{8,3}$ and $\CJ_{8,4}$. By a similar argument as above, one sees that
\begin{equation}\label{CJ82-3}
\begin{split}
|\CJ_{8,2}|+|\CJ_{8,3}|
 \lesssim
(\eps_0+\eta)\sum\limits_{1\leq|\al|\leq2}\left\|\pa^\al\widetilde{\phi}\right\|_{H^1}^2
+\eta\sum\limits_{1\leq|\al|\leq2}\left\|\pa^\al\widetilde{u}_1\right\|^2
+\eps_0\sum\limits_{|\al|=1}\left\|\pa^\al\widetilde{\rho}\right\|^2
+{C_\eta}\eps(1+t)^{-2}.
\end{split}
\end{equation}
For $\CJ_{8,4}$, when $|\al|=2$, we directly get from Lemma \ref{cl.Re.Re2.} and H\"{o}lder's inequality that
\begin{equation}\label{CJ84.p1}
\begin{split}
|\CJ_{8,4}|
 \lesssim \sum\limits_{|\al|=2}\left\|\pa^{\al}\pa_x\phi^r\right\|_{L^1} \left\|\pa^{\al} u^r_1\right\|_{L^1}\lesssim \eps(1+t)^{-2}.
\end{split}
\end{equation}
If $|\al|=1$, thanks to \eqref{sob.ine.} and Lemma \ref{cl.Re.Re2.}, one has
\begin{equation}\label{CJ84.p2}
\begin{split}
|\CJ_{8,4}|
 \lesssim \sum\limits_{|\al|=1}\left\|\pa^{\al}\pa_x\phi^r\right\|_{L^1}
 \left\|\pa_x\pa^{\al} u^r_1\right\|^{1/2}\left\|\pa^{\al} u^r_1\right\|^{1/2}\lesssim \eps^{3/4}(1+t)^{-5/4}.
\end{split}
\end{equation}
Consequently,
\eqref{CJ8.c} follows from
\eqref{CJ81s}, \eqref{CJ82-3}, \eqref{CJ84.p1} and \eqref{CJ84.p2}. This ends the proof of \eqref{CJ8.c}.
\end{proof}


\begin{proof}[Proof of \eqref{2d.F.sum2}]
Performing the similar calculations as for obtaining \eqref{2d.F}, we have
\begin{equation}\label{2d.F2-rjap-adp1}
\begin{split}
\frac{1}{2}\frac{d}{dt}&\int_{{\R}\times{\R}^3}\frac{\left|\pa^{\al} F\right|^2}{\FM_*}dxd\xi
-\left(L_{\FM}\pa^{\al}\FG,\frac{\pa^{\al} \FG}{\FM_*}\right)
\\
=&\underbrace{\left(\pa^{\al}\pa_x\phi\pa_{\xi_1}\FM,\frac{\pa^{\al} \FM}{\FM_*}\right)}_{\CK_{4}}+\sum\limits_{0<\al'\leq \al}C_{\al'}^{\al}\left(Q(\pa^{\al'} \FM, \pa^{\al-\al'}\FG)+Q(\pa^{\al-\al'}\FG,\pa^{\al'} \FM),\frac{\pa^{\al} F}{\FM_*}\right)
\\&\underbrace{+\left(L_{\FM}\pa^{\al}\FG, \FP_1^{\FM}\left(\frac{\pa^{\al} \FM}{\FM_*}\right)\right)}_{\CK_{5}}
+\sum\limits_{0<\al'<\al}C_{\al'}^{\al}\left(\pa^{\al-\al'} \pa_x\phi \pa^{\al'} \pa_{\xi_1}F,\frac{\pa^{\al} F}{\FM_*}\right)
\\&+\left(\pa_x\phi\pa^{\al}\pa_{\xi_1}F,\frac{\pa^{\al} F}{\FM_*}\right)
+\left(\pa^{\al}\pa_x\phi\pa_{\xi_1}\FG,\frac{\pa^{\al} F}{\FM_*}\right)
+\left(\pa^{\al} Q(\FG,\FG),\frac{\pa^{\al} F}{\FM_*}\right),
\end{split}
\end{equation}
where $1\leq|\al|\leq2$.
We now estimate $\CK_{4}$ and $\CK_{5}$ only, since that other terms can be treated in the same way.
When $|\al|=2$, by applying H\"{o}lder's inequality, one has
\begin{equation*}
\begin{split}
|\CK_{4}|\leq& {\sum\limits_{|\al|=2}\int_{\R}\left|\pa^\al\pa_x\phi\pa^{\al}\left[\rh,u,\ta\right]\right|dx
+\sum\limits_{|\al|=2,|\al'|=1}\int_{\R}\left|\pa^\al\pa_x\phi\right|\left|\pa^{\al'}\left[\rh,u,\ta\right]\right|^2dx}
\\
\leq&
C\sum\limits_{1\leq|\al|\leq2}\left\|\pa^\al\left[\widetilde{\rho},\widetilde{u},\widetilde{\ta}\right]\right\|^2+C\sum\limits_{|\al|=2}
\left\|\pa^{\al}\pa_x\widetilde{\phi}\right\|^2
+C\eps(1+t)^{-2}.
\end{split}
\end{equation*}
When $|\al|=1$, we get from  H\"{o}lder's inequality and the same argument as in proving \eqref {CJ84.p2} that
\begin{equation*}
\begin{split}
|\CK_{4}|\leq& C\sum\limits_{|\al|=1}\left\|\pa^\al\left[\widetilde{\rho},\widetilde{u},\widetilde{\ta}\right]\right\|^2+C\sum\limits_{1\leq |\al|\leq2}
\left\|\pa^{\al}\widetilde{\phi}\right\|^2
+\eps^{3/4}(1+t)^{-5/4}.
\end{split}
\end{equation*}
As to $\CK_{5}$,  when $|\al|=2$,
by using Cauchy-Schwarz's inequality with $0<\eta<1$ and Lemma \ref{est.nonop}, we have
\begin{equation*}
\begin{split}
|\CK_{5}|\leq& {\eta\left\|\frac{(1+\xi)^{-1/2}\left[Q(\FM,\pa^{\al}\FG)+Q(\pa^{\al}\FG,\FM)\right]}{\sqrt{\FM_\ast}}\right\|^2
+C_\eta\left\|\sqrt{1+|\xi|}\sqrt{\FM_\ast}\FP_1^{\FM}\left(\frac{\pa^{\al} \FM}{\FM_*}\right)\right\|}
\\ \leq&
\eta\sum\limits_{|\al|=2}\int_{\R\times{\R}^3}\frac{(1+|\xi|)\left|\pa^{\al} \FG\right|^2}{\FM_*}dxd\xi
+C_\eta\sum\limits_{1\leq |\al|\leq 2}\left\|\pa^\al\left[\widetilde{\rho},\widetilde{u},\widetilde{\ta}\right]\right\|^2
+{C_\eta}\eps(1+t)^{-2}.
\end{split}
\end{equation*}
As to $|\al|=1$, it will be more complicated, and we first rewrite $\CK_{5}$ as
\begin{equation*}
\begin{split}
\sum\limits_{|\al|=1}\left(L_{\FM}\pa^{\al}\FG, \FP_1^{\FM}\left(\frac{\pa^{\al} \FM}{\FM_*}\right)\right)
=&\sum\limits_{|\al|=1}\left(Q\left(\FM,\pa^{\al}(\widetilde{\FG}+\overline{\FG})\right)
+Q\left(\pa^{\al}(\widetilde{\FG}+\overline{\FG}),\FM\right),\FP_1^{\FM}\left(\frac{\pa^{\al} \FM}{\FM_*}\right)\right)
\\=&\underbrace{-\sum\limits_{|\al|=1}\left(Q\left(\pa^{\al}\FM,\widetilde{\FG}\right)
+Q\left(\widetilde{\FG},\pa^{\al}\FM\right),\FP_1^{\FM}\left(\frac{\pa^{\al} \FM}{\FM_*}\right)\right)}_{\CK_{5,1}}
\\&\underbrace{-\sum\limits_{|\al|=1}\left(Q\left(\FM,\widetilde{\FG}\right)+Q\left(\widetilde{\FG},\FM\right),\pa^{\al}\FP_1^{\FM}\left(\frac{\pa^{\al} \FM}{\FM_*}\right)\right)}_{\CK_{5,2}}
\\&\underbrace{+\sum\limits_{|\al|=1}\left(Q\left(\FM,\pa^{\al}\overline{\FG}\right)
+Q\left(\pa^{\al}\overline{\FG},\FM\right),\FP_1^{\FM}\left(\frac{\pa^{\al} \FM}{\FM_*}\right)\right)}_{\CK_{5,3}}.
\end{split}
\end{equation*}
Then utilizing Lemmas \ref{est.nonop} and \ref{cl.Re.Re2.}, Cauchy-Schwarz's inequality with $0<\eta<1$, Sobolev's inequality \eqref{sob.ine.} as well as the a priori assumption \eqref{aps2}, one can show that

$$
|\CK_{5,1}|+|\CK_{5,2}|\lesssim(\eta+\eps_0)\int_{\R\times{\R}^3}\frac{(1+|\xi|)\left| \widetilde{\FG}\right|^2}{\FM_*}dxd\xi
+(C_\eta+\eps_0)\sum\limits_{1\leq|\al|\leq2}\left\|\pa^\al\left[\widetilde{\rho},\widetilde{u},\widetilde{\ta}\right]\right\|^2
+{C_\eta}\eps(1+t)^{-2},
$$
and
$$
|\CK_{5,3}|\lesssim
\sum\limits_{|\al|=1}\left\|\pa^\al\left[\widetilde{\rho},\widetilde{u},\widetilde{\ta}\right]\right\|^2+ \eps^{3/4}(1+t)^{-5/4}.
$$
Here, the details of derivations are omitted for brevity. This completes the proof of \eqref{2d.F.sum2} after taking the summation of \eqref{2d.F2-rjap-adp1} over $1\leq |\al|\leq 2$ and applying all the estimates above.
\end{proof}

\begin{proof}[Proof of \eqref{mi.diss4}] To prove \eqref{mi.diss4}, for fixed $\al$ and $\be$ satisfying $|\al|+|\be|\leq 2$ and $|\be|\geq1$, it suffices to estimate all the terms on the right hand side of \eqref{mixd.ip},
since the second term on the left hand side can be bonded below by
$$
\delta\int_{{\R}\times{\R}^3}\frac{(1+|\xi|)
\left|\pa^{\al}\pa^\be\widetilde{\FG}\right|^2}{\FM_{*}}dxd\xi ,
$$
according to Lemma \ref{co.est.}.

From Lemma \ref{est.nonop} and Cauchy-Schwarz's inequality with $0<\eta<1$, the first term on the right hand side
of \eqref{mixd.ip} is bounded by
\begin{equation*}
\begin{split}
\eta&\int_{{\R}\times{\R}^3}\frac{(1+|\xi|)
\left|\pa^{\al}\pa^\be\widetilde{\FG}\right|^2}{\FM_{*}}dxd\xi +C_\eta \sum\limits_{\al'\leq\al,\be'<\be}\int_{{\R}\times{\R}^3}\frac{(1+|\xi|)
\left|\pa^{\al'}\pa^{\be'}\widetilde{\FG}\right|^2}{\FM_{*}}dxd\xi
\\&+C_\eta\eps_0\sum\limits_{\al'<\al,\be'=\be}\int_{{\R}\times{\R}^3}\frac{(1+|\xi|)
\left|\pa^{\al'}\pa^{\be'}\widetilde{\FG}\right|^2}{\FM_{*}}dxd\xi .
\end{split}
\end{equation*}
The second and sixth terms are dominated by
\begin{equation*}
\begin{split}
\eta\int_{{\R}\times{\R}^3}\frac{(1+|\xi|)
\left|\pa^{\al}\pa^\be\widetilde{\FG}\right|^2}{\FM_{*}}dxd\xi +C_\eta\sum\limits_{1\leq|\al|\leq2}\left\|\pa^\al\left[\widetilde{\rho},\widetilde{u},\widetilde{\ta}\right]\right\|^2+{C_\eta}\eps(1+t)^{-2}.
\end{split}
\end{equation*}
Applying the splitting $\FG=\widetilde{\FG}+\overline{\FG}$ and the macro-micro decomposition $\xi_1\pa_x\FG=\FP_0^{\FM}\left(\xi_1\pa_x\FG\right)+\FP_1^{\FM}\left(\xi_1\pa_x\FG\right)$, we see that the third term can be rewritten as
\begin{equation*}
\begin{split}
-\left(\pa^{\al}\pa^\be\left(\xi_1\pa_x\widetilde{\FG}\right),\frac{\pa^{\al}\pa^\be\widetilde{\FG}}{\FM_{*}}\right)
-\left(\pa^{\al}\pa^\be\left(\xi_1\pa_x\overline{\FG}\right),\frac{\pa^{\al}\pa^\be\widetilde{\FG}}{\FM_{*}}\right)
+\left(\pa^{\al}\pa^\be\left(\FP_0^{\FM}\left(\xi_1\pa_x\FG\right)\right),\frac{\pa^{\al}\pa^\be\widetilde{\FG}}{\FM_{*}}\right),
\end{split}
\end{equation*}
which can be further bounded by
\begin{equation*}
\begin{split}
\eta&\int_{{\R}\times{\R}^3}\frac{(1+|\xi|)
\left|\pa^{\al}\pa^\be\widetilde{\FG}\right|^2}{\FM_{*}}dxd\xi +C_\eta \sum\limits_{|\al'|+|\be'|\leq |\al|+|\be|\atop{\be'<\be}}\int_{{\R}\times{\R}^3}\frac{(1+|\xi|)
\left|\pa^{\al'}\pa^{\be'}\widetilde{\FG}\right|^2}{\FM_{*}}dxd\xi
\\&+C_\eta\sum\limits_{1\leq|\al|\leq2}\int_{{\R}\times{\R}^3}\frac{(1+|\xi|)\left|\pa^{\al}\FG\right|^2}{\FM_{*}}dxd\xi
{+C_\eta\sum\limits_{1\leq|\al|\leq2}\left\|\pa^\al\left[\widetilde{\rho},\widetilde{u},\widetilde{\ta}\right]\right\|^2 +C_\eta\eps(1+t)^{-2}}.
\end{split}
\end{equation*}
Similarly, the fourth term is dominated by
\begin{equation*}
\begin{split}
(\eta+\eps_0)&\int_{{\R}\times{\R}^3}\frac{(1+|\xi|)
\left|\pa^{\al}\pa^\be\widetilde{\FG}\right|^2}{\FM_{*}}dxd\xi
+\eps_0\sum_{|\al'|+|\be'|\leq |\al|+|\be|}\int_{{\R}\times{\R}^3}\frac{(1+|\xi|)
\left|\pa^{\al'}\pa^{\be'}\widetilde{\FG}\right|^2}{\FM_{*}}dxd\xi
\\&+C_\eta\sum\limits_{1\leq|\al|\leq2}\left\|\pa^\al\left[\widetilde{\rho},\widetilde{u},\widetilde{\ta},\widetilde{\phi}\right]\right\|^2
+{C_\eta}\eps(1+t)^{-2}.
\end{split}
\end{equation*}
Using $\FG=\widetilde{\FG}+\overline{\FG}$ again, one can rewrite the fifth term as
\begin{equation*}
\begin{split}
&\left(\pa^{\al}\pa^\be Q\left(\widetilde{\FG},\widetilde{\FG}\right),\frac{\pa^{\al}\pa^\be\widetilde{\FG}}{\FM_{*}}\right)
+\left(\pa^{\al}\pa^\be Q\left(\widetilde{\FG},\overline{\FG}\right),\frac{\pa^{\al}\pa^\be\widetilde{\FG}}{\FM_{*}}\right)
\\&+\left(\pa^{\al}\pa^\be Q\left(\overline{\FG},\widetilde{\FG}\right),\frac{\pa^{\al}\pa^\be\widetilde{\FG}}{\FM_{*}}\right)
+\left(\pa^{\al}\pa^\be Q\left(\overline{\FG},\overline{\FG}\right),\frac{\pa^{\al}\pa^\be\widetilde{\FG}}{\FM_{*}}\right),
\end{split}
\end{equation*}
which can be controlled by
\begin{equation*}
\begin{split}
(\eta+\eps_0)&\int_{{\R}\times{\R}^3}\frac{(1+|\xi|)
\left|\pa^{\al}\pa^\be\widetilde{\FG}\right|^2}{\FM_{*}}dxd\xi
+\eps_0\sum_{\al'\leq\al,\be'\leq\be}\int_{{\R}\times{\R}^3}\frac{(1+|\xi|)
\left|\pa^{\al'}\pa^{\be'}\widetilde{\FG}\right|^2}{\FM_{*}}dxd\xi
\\&+C_\eta\sum\limits_{1\leq|\al|\leq2}\left\|\pa^\al\left[\widetilde{\rho},\widetilde{u},\widetilde{\ta}\right]\right\|^2
+{C_\eta}\eps(1+t)^{-2},
\end{split}
\end{equation*}
according to Lemmas \ref{est.nonop} and \ref{cl.Re.Re2.}, Cauchy-Schwarz's inequality with $0<\eta<1$, Sobolev's inequality \eqref{sob.ine.} as well as the a priori assumption \eqref{aps2}.

Finally, taking a suitable linear combination of the above estimates for all the cases that $|\al|+|\be|\leq 2$ and $|\be|\geq1$, 
and noticing that
\begin{equation*}
\begin{split}
\sum_{1\leq|\al|\leq2}&\int_{{\R}\times{\R}^3}\frac{(1+|\xi|)
\left|\pa^{\al}\widetilde{\FG}\right|^2}{\FM_{*}}dxd\xi
\\ \lesssim& \sum_{1\leq|\al|\leq2}\int_{{\R}\times{\R}^3}\frac{(1+|\xi|)
\left|\pa^{\al}\FG\right|^2}{\FM_{*}}dxd\xi+\sum_{1\leq|\al|\leq2}\int_{{\R}\times{\R}^3}\frac{(1+|\xi|)
\left|\pa^{\al}\overline{\FG}\right|^2}{\FM_{*}}dxd\xi
\\ \lesssim& \sum_{1\leq|\al|\leq2}\int_{{\R}\times{\R}^3}\frac{(1+|\xi|)
\left|\pa^{\al}\FG\right|^2}{\FM_{*}}dxd\xi
+\sum\limits_{1\leq|\al|\leq2}\left\|\pa^\al\left[\widetilde{\rho},\widetilde{u},\widetilde{\ta}\right]\right\|^2
+\eps(1+t)^{-2},
\end{split}
\end{equation*}
one sees that
\eqref{mi.diss4} holds, and this ends the proof of \eqref{mi.diss4}.
\end{proof}

\medskip
\noindent {\bf Acknowledgements:} RJD was supported by the General Research Fund (Project No.~400912) from RGC of Hong Kong. SQL was
supported by grants from the National Natural Science Foundation of China under contracts 11101188 and 11271160. The authors would like to thank Professor Tong Yang and Professor Huijiang Zhao for many fruitful discussions on the topic of the paper.


\end{document}